\documentclass{gtpart}     
\usepackage{amscd, amsfonts, amsmath, amssymb, amsthm, color}
\usepackage[all]{xy}
\usepackage{pinlabel}
\usepackage{graphicx}

\newtheorem{thm}{Theorem}[section]
\newtheorem{prp}[thm]{Proposition}
\newtheorem{cor}[thm]{Corollary}
\newtheorem{lma}[thm]{Lemma}

\theoremstyle{definition}

\newtheorem{rmk}[thm]{Remark}

\newenvironment{pf}{\begin{proof}}{\end{proof}}

\numberwithin{equation}{section}

\newcommand{\K}{{\mathbb{K}}}
\newcommand{\bbH}{{\mathbb{H}}}

\newcommand{\Cc}{{\mathcal{C}}}

\newcommand{\bq}{{\mathbf q}}

\newcommand{\M}{{\mathcal{M}}}

\newcommand{\ev}{\operatorname{ev}}
\newcommand{\la}{\langle}
\newcommand{\ra}{\rangle}
\newcommand{\pa}{\partial}

\newcommand{\ind}{\operatorname{index}}

\newcommand{\img}{\operatorname{im}}
\newcommand{\Morse}{\operatorname{Morse}}

\newcommand{\tb}{\operatorname{tb}}
\newcommand{\pr}{\operatorname{pr}}

\renewcommand{\qed}{{\flushright{\hfill \rule{2mm}{2mm}}}}

\newcommand{\wt}{\widetilde}
\newcommand{\wh}{\widehat}
\newcommand{\ol}{\overline}
\newcommand{\ul}{\underline}
\newcommand{\p}{\partial}

\newcommand{\om}{\omega}

\newcommand{\eps}{\varepsilon}

\newcommand{\oL}{\overline{L}}
\newcommand{\oX}{\overline{X}}
\newcommand{\oW}{\overline{W}}

\newcommand{\ug}{{\underline{\gamma}}}
 \newcommand{\bh}{{\mathbf{h}}} 
\newcommand{\q}{{\bf q}}

\newcommand{\A}{{\bf A}}
\newcommand{\B}{{\bf B}}

\newcommand{\bQ}{{\bf Q}}
\newcommand{\im}{{\rm im }}        
\newcommand{\st}{{\rm st}}

\newcommand{\const}{{\rm const}}

\newcommand{\Int}{{\rm int\,}} 

\renewcommand{\min}{{\rm min}}
\renewcommand{\max}{{\rm max}}

\newcommand{\codim}{\operatorname{codim}}

\newcommand{\dist}{{\rm dist}}

\renewcommand{\Im}{{\rm Im\,}}

\newcommand{\Id}{\mathrm{id}}

\newcommand{\Ker}{\mathrm{Ker}}

\newcommand{\good}{\mathrm{good}}
\newcommand{\bad}{\mathrm{bad}}
\newcommand{\Ho}{\mathrm{Ho}}

\newcommand{\Core}{\mathrm{Core}}

\newcommand{\MM}{\mathcal{M}}
\newcommand{\CC}{\mathcal{C}}

\newcommand{\PP}{\mathcal{P}}

\newcommand{\II}{\mathcal{I}}

\newcommand{\raq}{\overrightarrow{q}}
\newcommand{\Raq}{\overrightarrow{\bf{q}}}
\newcommand{\laq}{\overleftarrow{q}}
\newcommand{\Laq}{\overleftarrow{\bf{q}}}

\def\lr{\leftrightarrow}
\def\ainf{A_\infty}
\def\ot{\otimes}
\def\op{\oplus}

\def\hom{\mathrm{Hom}}
\def\td{T(\mathcal{D}_+[1])}
\def\ccd{(\mathcal{D} \otimes \td)^{\mathrm{diag}}}
\def\L{\Lambda}
\def\d{\delta}
\def\lhc{LHO}
\def\lhcc{\widecheck{LHO}(\L)}
\def\lhch{\widehat{LHO}(\L)}
\def\lhwt{LH^{\mathrm{Ho}}(\L)}

\def\A{\mathcal{A}}
\def\B{\mathcal{B}}

\def\D{\mathcal{D}}

\def\ZZ{\mathcal{Z}}

\makeatletter
\DeclareRobustCommand\widecheck[1]{{\mathpalette\@widecheck{#1}}}
\def\@widecheck#1#2{%
    \setbox\z@\hbox{\m@th$#1#2$}%
    \setbox\tw@\hbox{\m@th$#1%
       \widehat{%
          \vrule\@width\z@\@height\ht\z@
          \vrule\@height\z@\@width\wd\z@}$}%
    \dp\tw@-\ht\z@
    \@tempdima\ht\z@ \advance\@tempdima2\ht\tw@ \divide\@tempdima\thr@@
    \setbox\tw@\hbox{%
       \raise\@tempdima\hbox{\scalebox{1}[-1]{\lower\@tempdima\box
\tw@}}}%
    {\ooalign{\box\tw@ \cr \box\z@}}}
\makeatother

\title
{Effect of Legendrian Surgery}
\author{Fr{\'e}d{\'e}ric Bourgeois}
\givenname{Fr{\'e}d{\'e}ric}
\surname{Bourgeois}
\address{Universit\'e Libre de Bruxelles }
\email{fbourgeo@ulb.ac.be}

\author{Tobias Ekholm}
\givenname{Tobias}
\surname{Ekholm}
\address{Uppsala University}
\email{tobias@math.uu.se}

\author{Yakov Eliashberg}
\givenname{Yakov}
\surname{Eliashberg}
\address{Stanford University}
\email{eliash@math.stanford.edu}


\volumenumber{}
\issuenumber{}
\publicationyear{}
\papernumber{}
\startpage{}
\endpage{}
\doi{}
\MR{}
\Zbl{}
\received{}
\revised{}
\accepted{}
\published{}
\publishedonline{}
\proposed{}
\seconded{}
\corresponding{}
\editor{}
\version{}

\begin{document}
\begin{abstract}
The paper is a summary of the results of the authors concerning computations of symplectic invariants of Weinstein manifolds and contains some examples and applications. Proofs are sketched.  The detailed proofs will appear in our forthcoming paper \cite{BEE}.

In the  Appendix written by S.~Ganatra  and M.~Maydanskiy  it is shown that the results of this paper imply P.~Seidel's conjecture from \cite{shhh}.
\end{abstract}
\maketitle
\section*{Introduction}
We study in this paper how attaching of a Lagrangian handle in the sense of \cite{We,Eli-pluri,CE} to a symplectic manifold with contact boundary affects symplectic invariants of the manifold and contact invariants of its boundary. We establish several surgery exact triangles for these invariants.  As explained in Section \ref{sec:Lefschetz} below, symplectic handlebody presentations and Lefschetz fibration presentations of Liouville symplectic manifolds are closely related. In this sense our results can be viewed as generalizations of P.~Seidel's exact triangles for symplectic Dehn twists, see \cite{Se-mut,Se-more-mut,Se-long}. In particular, as shown in the Appendix written by S.~Ganatra  and M.~Maydanskiy, our results imply  P.~Seidel's conjecture from \cite{shhh}.

This is the first paper in a series devoted to this subject. In order to make the  results more accessible and  the algebraic formalism not too heavy,   we make here  a number of simplifying assumptions (like vanishing of the first Chern class). Though proofs are only sketched, we indicate the main ideas and provide some details which should help  specialists to reconstruct the proofs. The general setup and detailed proofs will appear in the forthcoming paper \cite{BEE}.

{\em Plan of the paper.} In Section \ref{sec:Weinstein} we review the notions related to Liouville and Weinstein  manifolds and domains. In Section \ref{Sec:mdli}, various moduli spaces of holomorphic curves needed for our algebraic constructions are defined. In Section \ref{sec:invariants} we define the symplectic and contact invariants that are computed in this paper, namely linearized contact homology, and reduced and full symplectic homology. The definitions we use are due to F.~Bourgeois and A.~Oancea, see \cite{BO}. In Section \ref{sec:algebra} we recall the definition of  Legendrian homology algebra,  see \cite{Chek,SFT, EES}, and define several constructions derived from it in the spirit of cyclic and Hochschild homology.  In Section \ref{sec:triangles} we formulate and sketch the proof of our main results. These are: \begin{itemize}
\item[--] Theorems \ref{thm:f-ch}, and \ref{thm:f-sh1} respectively \ref{thm:SH}, which describe surgery exact triangles for linearized contact homology, and reduced respectively full symplectic homology;
\item[--] Corollary \ref{cor:SH-subcrit}, which gives a closed form formula for symplectic homology;
\item[--] Theorem \ref{thm:LHA}, which  relates the Legendrian homology algebras of a Legendrian submanifold before and after surgery; and
\item[--] Theorem \ref{thm:LH}, which provides a formula for the linearized Legendrian homology of the so-called co-core (the meridian of the handle) Legendrian sphere after surgery.
\end{itemize}

Section  \ref{sec:examples} is devoted to  first examples and applications.  It is worthwhile  to point out that already quite primitive computations yield interesting geometric applications. In particular, we show that Legendrian  surgery on Y.~Chekanov's two famous Legendrian $(5,2)$-knots in $S^3$, see \cite{Chek}, give non-contactomorphic $3$-manifolds, and that attaching a Lagrangian handle to the ball along stabilized trivial Legendrian knots  produce examples of exotic Weinstein symplectic structures on $T^*S^n$ (exotic structures on $T^*S^n$ were first constructed by  M. McLean in \cite{McL} and also in M.~Maydanskiy and P.~Seidel in \cite{MS}).

In Section \ref{sec:Lefschetz} we explain the relation between the Weinstein handlebody and the Lefschetz fibration formalisms. In the Appendix, written by M.~Maydanskiy and S.~Ganatra, this description is used to deduce P.~Seidel's conjecture \cite{shhh} from the results of the current paper.
\medskip

This paper was conceived several years ago and over this period we benefited a lot from discussions with several mathematicians. We are especially thankful   (in alphabetical order) to Mohammed Abouzaid, Denis Auroux, Cheolhyun Cho, Kai Cieliebak, Jian He, Helmut  Hofer, Janko Latschev, Lenny Ng,  Alexandru Oancea, Josh Sabloff, Eric Schoenfeld,
Paul Seidel, Ivan Smith, and Alexander Voronov.
We are very grateful to Maxim Maydanskiy and Sheel Ganatra for writing the Appendix of this paper.

\section{Weinstein manifolds}\label{sec:Weinstein}
Let $(X,\omega)$ be an exact symplectic manifold of dimension $2n$. A primitive of $\omega$, i.e.~a $1$-form $\lambda$ such that $d\lambda=\omega$, is called a {\em Liouville form} 
on $X$ and the vector field $Z$ that is $\omega$-dual to $\lambda$, i.e.~such that the 
contraction $i_{Z}\omega$ satisfies $i_Z\omega=\lambda$, is called  the {\em Liouville vector field}. Note that the equation $d\lambda=\omega$ is equivalent to $L_Z\omega=\omega$, 
where $L$ denotes the Lie-derivative. Hence, if $Z$ integrates to a flow $\Phi_Z^t:X\to X$ then $(\Phi_Z^t)^*\omega=e^t\omega$, i.e.~the Liouville vector field $Z$ is (symplectically) 
{\em expanding}, while $-Z$ is {\em contracting}. By a {\em Liouville manifold} we will mean a  
pair $(X,\lambda)$, or equivalently a triple $(X,\om,Z)$ where $Z$ is an expanding vector field for $\om$. Note that
\begin{equation}\label{eq:lambda}
i_Z\lambda=0,\qquad i_Zd\lambda=\lambda,\qquad\text{and}\qquad
L_Z\lambda=\lambda,
\end{equation}
and hence the flow of $Z$ expands also the Liouville form: $(\Phi_Z^t)^*\lambda=e^t\lambda$.

We say that a Liouville manifold $(X,\omega,Z)$ has a (positive) {\em cylindrical end} if there exists a compact domain $\oX\subset X$ with a smooth  boundary $Y=\p\oX$ transverse to 
$Z$, and such that $X\setminus \oX =\bigcup_{t=0}^\infty \Phi_Z^t(Y)$.  Then $X\setminus\Int\oX$ is diffeomorphic to $[0,\infty)\times Y$ in such a way that the Liouville vector field $Z$ corresponds to $\frac{\p}{\p s}$ and the Liouville form $\lambda=i_Z\omega$ can be written as $e^s\alpha$, where $s\in[0,\infty)$ is the parameter of the flow and $\alpha=\lambda|_{Y}$. 
The form $\alpha$ is contact, and thus $(X\setminus\Int\oX,\om)$ can be identified with the 
positive half  $[0,\infty)\times Y$ of the symplectization of the contact manifold $(Y,\xi=\ker\alpha)$. In fact, the whole symplectization of $(Y,\xi)$ sits in $X$ as  $\bigcup_{t\in\R}\Phi_Z^t(Y)$ and this embedding is canonical in the sense that the image is independent of 
the choice of $Y$. Its complement $X\setminus \bigcup_{t\in\R}\Phi_Z^t(Y)$ is equal to $\bigcup_K\bigcap_{t>0} \Phi_Z^{-t}(K)$, where the union is taken over all compact subset $K\subset X$. It is called    the {\em core} of the Liouville manifold $(X,\om, Z)$ and is denoted by $\Core(X,\om, Z)$. The Liouville manifold $(X,\om,Z)$ defines the contact manifold $(Y,\xi)$ canonically. We will write $(Y,\xi)=\partial(X,\om, Z)$ and call it the {\em ideal contact 
boundary} of the Liouville manifold $(X,\omega,Z)$ with cylindrical end. Equivalently, we can view $(Y,\xi)$ as the contact boundary of the compact  symplectic manifold $\oX$. We will refer to $\oX$ as a {\em Liouville domain with contact boundary}. Contact manifolds which 
arise as ideal boundaries of Liouville symplectic manifolds with cylindrical ends are called {\em strongly symplectically fillable}.

Given a Liouville domain $(\oX,\lambda)$ one can always {\em complete} it by attaching the cylindrical end $\left([0,\infty)\times Y, d(e^s\alpha)\right)$, $Y=\pa\oX$ to get a Liouville 
manifold $X$. We will keep the  notation $\lambda$ and $Z$ for the extended Liouville form and    vector field. {\em All Liouville manifolds considered  in this paper will be assumed  to have cylindrical ends}.

A map $\psi\colon (X_0,\lambda_0)\to (X_1,\lambda_1)$ between Liouville manifolds is called {\em exact symplectic} if $\psi^*\lambda_1-\lambda_0$ is exact.

\begin{lma}\label{l:exactsympl}
Given a symplectomorphism  $\psi:(X_0,\lambda_0)\to (X_1,\lambda_1)$ between Liouville manifolds with cylindrical ends,  there exists a symplectic isotopy $\phi_t: (X_0,\lambda_0)\to (X_0,\lambda_0)$, $t\in[0,1]$, $\phi_0=\Id$ such that the  symplectomorphism $\psi\circ\phi_1$ is exact, and $\dist((x,s),\phi_1(x,s))\leq Ce^{-s}$ for $(x,s)\in E$ and any cylindrical metric on $E=[0,\infty)\times Y$.
\end{lma}

\begin{proof}
Let $\theta:= \psi^*\lambda_1-\lambda_0$. The restriction of the closed 1-form $\theta$ to the end $E=[0,\infty)\times Y$ can be written as $\pi^*\overline\theta+dH$, where $\overline\theta=\theta|_Y$, $\pi:E\to Y$ is the projection and $H:E\to\R$ is a smooth function. Take a cut-off function $\zeta:[0,\infty)\to[0,1]$ which is equal to $1$ near $0$ and to $0$ on $[1,\infty)$. Set $\wt\theta:=\pi^*\overline\theta+d(\zeta(s) H)$ and $\wt H:= (1-\zeta(s))H$.

Let $V$ be the symplectic vector field dual to the closed 1-form $-\wt\theta $ with respect to $d\lambda_0$. Computing $V$ on $[1,\infty) \times Y$ we observe that $V$ has the form $e^{-s}\wt V$, where $\wt V$ is independent of $s$. Hence, $V$ is integrable and integrates to an isotopy $\phi_t$ such that $\dist((x,s),\phi_1(x,s))\leq Ce^{-s}$ for any cylindrical metric on $E$.
Computing the Lie derivative $L_{V}\lambda_0$, we get
$L_V\lambda_0=-\wt\theta+d(\lambda_0(V))$. Hence
$$(\psi\circ\phi_1)^*\lambda_1=\phi_1^*\lambda_0+\phi_1^*\wt\theta + d\wt H\circ\phi_1=\lambda_0+dG$$
for some smooth function $G$.
\end{proof}

\begin{rmk}
It is an open problem whether  the isotopy in the above lemma can be chosen in such a way that
$\psi\circ\phi_1$ is Liouville at infinity, i.e.~so that $(\psi\circ\phi_1)^*\lambda_1-\lambda_0=dH$ for some function $H$ with compact support.
\end{rmk}

Let $(X,\om,Z)$ be a Liouville manifold with cylindrical end. It is called
{\em Weinstein}\footnote{Weinstein manifolds are symplectic counterparts of Stein (or affine) complex manifolds, see \cite{EliGrom-convex, Eli-pluri, CE}.}, see \cite{EliGrom-convex}, if  the vector field $Z$ is gradient-like for an exhausting (i.e.~proper and bounded below) Morse function $H\colon X\to \R$.

Let us recall that  $Z$ is called  {\em gradient-like} for a smooth Morse function $H\colon X\to\R$, and  $H\colon X\to\R$ is called a {\em Lyapunov function} for $Z$  if $dH(Z)\geq\delta|Z|^2$ for some $\delta>0$, where $|Z|$ is the norm with respect to some Riemannian metric on $X$.  We will always assume that $H$ restricted to the end $E=[0,\infty)\times  Y$ depends only on the coordinate $s$. The corresponding compact Liouville domain $\oX$ will be called in this case a {\em Weinstein domain}. Note that critical points of a Lyapunov Morse  function for a Liouville vector field $Z$ have indices $\leq n=\frac12\dim X$, and that the stable manifold $L_p$ of  the vector field $Z$ for any critical point $p$  of $H$ is isotropic  of dimension equal to the Morse  index of $p$. The intersection of these stable manifolds with regular levels $H^{-1}(c)$ of the function $H$ are isotropic for the induced contact structure $\ker(\lambda|_{H^{-1}(c)})$ on $H^{-1}(c)$.

A Weinstein manifold is called {\em subcritical} if all critical points of $H$ have index $<n$. Symplectic topology of subcritical manifolds is not very interesting. According to a theorem of K.~Cieliebak, see \cite{Ciel1}, any subcritical Weinstein manifold $X$ of dimension $2n$  is symplectomorphic to a product $X'\times\R^2$, where $X'$ is a Weinstein manifold of dimension $2n-2$. A theorem from \cite{EliGrom-convex} states that any  symplectic tangential homotopy equivalence between two subcritical Weinstein manifolds is homotopic to a symplectomorphism.

A {\em Liouville cobordism} is a pair $(\oW,\lambda)$, where $\oW$ is a compact manifold  with boundary $\p\oW$ partitioned into the union of two open-closed subsets, $\p\oW=\p_-\oW\sqcup\p_+\oW$, and where $\lambda$ is a Liouville form such that the corresponding Liouville vector field $Z$ is transverse to $\p\oW$, inward along $\p_-\oW$ and outward along $\p_+\oW$.
A Liouville cobordism is called {\em Weinstein} if there exists a Morse function $H:\oW\to\R$ which is  constant on $\p_-\oW $ and on $\p_+\oW$, which has no boundary critical points, and which is Lyapunov for the corresponding  Liouville vector field $Z$.

The {\em completion} of a Liouville cobordism $\oW$ is the manifold
\[
W=(-\infty,0]\times\p_-\oW\;\;\cup\;\; \oW\;\;\cup\;\; [0,\infty)\times\p_+\oW
\]
obtained by attaching  to $W$ along $\p_-\oW$ the negative part of the  symplectization of the contact manifold $\p_-\oW$ with contact form $\alpha_-=\lambda|_{\p_-\oW}$, and attaching  along $\p_+\oW$ the positive  part of the  symplectization of the contact manifold $\p_+\oW$ with contact form $\alpha_+=\lambda|_{\p_+\oW}$.

Important examples of Weinstein cobordisms are provided by domains $H^{-1}([c,C])\subset X$ in a Weinstein manifold $(X,\lambda,H)$ with all structures induced from $X$. Here $c$ and $C$ are two regular values of the Lyapunov function $H$.
\section{Moduli spaces of holomorphic curves}\label{Sec:mdli}
\subsection{Preliminaries}
In this section we will discuss the moduli spaces of holomorphic curves in a Liouville manifold $X$, in the symplectization $\R\times Y$ of its ideal  contact boundary  $Y$, and in completed symplectic cobordisms, which will be used in the algebraic constructions of later sections. We will pretend that we are in a ``transverse'' situation, i.e.~that all the Fredholm operators involved are regular at their zeroes, that all evaluations maps are transverse to the chosen targets, and in particular, that all moduli spaces have dimension as predicted by the appropriate index formula.
It is well known that this ideal situation rarely can be achieved, and that one needs to work with objects more general than holomorphic curves to be able to achieve the required transversality. However for the purposes of this exposition we ignore this difficulty and claim that the correct version can be built according to the scheme presented here using one of the, for this purpose, currently developed technologies, notably, the polyfold theory of H.~Hofer, K.~Wysocki, and E.~Zehnder, see \cite{HWZ, HWZ2, HWZ3}.

Given a contact manifold $Y$ with a fixed contact form $\alpha$, we say that an almost complex structure $J$ on the symplectization $(\R\times Y,d(e^s\alpha))$ is adjusted to $\alpha$ if it is independent of the coordinate $s$ on $\R$, if it is compatible with $d\alpha$ on the contact hyperplanes $\ker\alpha$ of the slices $\{s\}\times Y$, and  if $J\frac{\p}{\p s}=R_\alpha$, where $R_\alpha$ is the Reeb vector field of the contact form $\alpha$.

Given a Liouville manifold $(X,\om,Z)$ with a cylindrical end, we call an almost complex structure $J$ on $X$ {\it adjusted}  to the Liouville structure if it is compatible with $\om$  everywhere, and if on the end $E$, which is identified with the positive part of the  symplectization of $(Y,\ker\alpha)$,  it is adjusted to $\alpha$ in the above sense. Adjusted almost complex structures on a completed Liouville cobordism are defined in a similar way.

Throughout this paper we will assume, for the sake of simplicity, that the first Chern class of all symplectic manifolds considered is trivial, and moreover, that the canonical bundle  is trivialized. Note that at the cylindrical end the complex tangent bundle splits: $TX=\xi\oplus\eps^1$, where $\eps^1$ is the trivial complex line bundle generated by the vector field $Z$. We will assume that the trivialization of the canonical bundle at the end is compatible with this splitting.
Similarly, for each Lagrangian submanifold $L\subset X$ considered (e.g. the symplectization  of a  Legendrian manifold) we will assume that its relative Maslov class vanishes. For Lagrangian $L\subset X$, the $1$-dimensional determinant of the cotangent  bundle of $L$ is a subbundle of the trivialized complex canonical bundle. The triviality of the relative Maslov class first implies that this subbundle is trivial, and hence it gives rise to an $S^1=\R/2\pi\Z$-valued function, and second allows us to lift this function to a {\em phase function}  $\phi_L:L\to\R$ which is unique up to additive constant.

Given a contact manifold $(Y,\xi)$ with a fixed contact form $\alpha$ we will denote by $\PP(Y)$  the set of periodic orbits of the Reeb field $R_\alpha$, including multiples. We assume all orbits to be non-degenerate. For $\gamma\in\PP(Y)$, let $\kappa(\gamma)$ denote its multiplicity. If $\gamma\in\PP(Y)$ then we write $\ug$ for the simple geometric orbit underlying $\gamma$. For each $\gamma\in P(Y)$ fix a point $p_{\gamma}\in\ug$ and note that $p_\gamma$ determines a unique parameterization of $\gamma$: use $p_\gamma$ as starting point and then follow the flow of $R_\alpha$ $\kappa(\gamma)$ times around $\ug$.


The chosen trivialization of the canonical bundle   allows one to canonically assign an  integer Conley-Zehnder index  $CZ(\gamma)$ to any orbit $\gamma\in\PP(Y)$.  We will grade the orbits by this index,
$|\gamma|:=CZ(\gamma)$.
The parity of $  CZ(\gamma)+n-3$ is called the {\it parity} of the orbit $\gamma$.\footnote{It is customary in the contact homology theory to grade orbits by $CZ(\gamma)+n-3$. However, for the  linearized version of contact homology considered in this paper, and its relation with symplectic homology the  grading by $CZ(\gamma)$ seems to be more natural.}
The parity coincides with the sign of the
determinant $\det(I-A_\gamma)$, where $A_\gamma$ is the linearized Poincar\'e return map around the orbit
$\gamma$.
An orbit $\gamma\in\PP(Y)$ is called {\it bad} if it is an even multiple    of another orbit $\overline\gamma$ whose parity differs from the parity of $\gamma$. The set of bad orbits is denoted by $\PP_\bad(Y)$.
Non-bad orbits are called {\it good}  and the set of good orbits is denoted by $\PP_\good(Y)$.

Consider a Legendrian submanifold $\Lambda\subset Y$. Assuming it is in general position with respect to the  Reeb flow of $\alpha$ we have a discrete set $\CC(\Lambda)$ of non-degenerate chords of the Reeb foliation with ends on $\Lambda$.
One can assign, see \cite{EES},  an integer grading $|c|$ to each chord $c\in\CC(\Lambda)$.  Note that with our assumptions this grading can be defined as follows. Let $x_0,x_1$  be the  beginning and the end points  of the chord $c$, respectively.  The linearized Reeb flow along the chord allows us to transport the tangent space   $T_{x_0}\Lambda$  to all points of $c$.  Let us denote by  $\Phi_c$  the transport operator along the chord $c$. We can assume that for each chord $c$ we have $JT_{x_1} \Lambda=\Phi_cT_{x_0}\Lambda$. Indeed, this can be achieved by a homotopy  of the almost complex structure $J$ in a neighborhood of the chords, while the defined grading depends only on the homotopy class of $J$. Using the linearized Reeb flow we can also  extend by continuity the phase function $\phi_\Lambda$ to $c$ beginning from the point $x_0$. Let $\phi_-$ be the value of  the extended  phase function at the top end $x_1$ of the chord $c$. Then we define
$$
|c|:= \frac{\phi_\Lambda(x_1)-\phi_-}\pi+\frac{n-3}2.
$$
See \cite{EES} for the details.

\subsection{Holomorphic curves anchored in a Weinstein manifold}\label{sec:anchored}

Let $(\oX,\om,Z)$ be a Liouville domain and $\oX_0\subset\Int\oX$ be a subdomain such that the Liouville vector field $Z$ is outward transverse to $Y_0=\p X_0$. Then $\oW=\oX\setminus \Int \oX_0$ is a Liouville cobordism with $\p_-\oW=Y_0$ and $\p_+\oW=Y=\p\oX$. Let $W$ and $X_0$ be completions of $\oW$ and $\oX_0$. Let $\oL$ be  an exact Lagrangian submanifold $(\oL,\p\oL)\subset(W,\p W)$ which intersect $\p W_\pm$ along the Legendrian submanifolds $\Lambda_\pm=\oL\cap\p_\pm \oW$, and which is tangent to $Z$ along $\Lambda_\pm$. We complete $\oL$ to $L$ by adding cylindrical ends:
\begin{align*}
L&=(-\infty,0]\times\Lambda_-\;\;\cup\;\;\oL\;\;\cup\;\;[0,\infty)\times\Lambda_+\\
&\subset (-\infty,0]\times Y_0\;\;\cup\;\; \oW\;\;\cup\;\; [0,\infty)\times Y\;\;=\;\;W.
\end{align*}
We will refer to $L$ as an {\em exact Lagrangian cobordism} between the Legendrian submanifolds $\Lambda_-$ and $\Lambda_+$, and to $\Lambda_{-}$ and $\Lambda_+$ as the {\em negative} and {\em positive} ends of $L$.

Let $\Sigma$ be a Riemann surface with conformal structure $\mathsf{j}$. A set of punctures on $\Sigma$ is a finite collection $\{z_1, \ldots, z_k\} \subset \Sigma$ of distinct points. An {\em asymptotic marker} for an interior puncture $z$ of $\Sigma$ is a half-line $\ell \subset T_z\Sigma$. The pair $(z,\ell)$ is called a {\em decorated puncture}. Given a Riemann surface $\Sigma$ with decorated punctures $\{ (z_1, \ell_1), \ldots, (z_m,\ell_m)\}$ and boundary punctures $z'_1, \ldots, z'_{m'}$, the corresponding {\em punctured decorated Riemann surface} is the Riemann surface $\Sigma \setminus \{ z_1, \ldots, z_m, z'_1, \ldots, z'_{m'} \}$ where the interior punctures are equipped with the asymptotic markers $\ell_1, \dots, \ell_m$. A map
\[
f\colon \Sigma \setminus \{ z_1, \ldots, z_m, z'_1, \ldots, z'_{m'} \} \to W
\]
is holomorphic if $df \circ \mathsf{j} = J \circ df$. Holomorphic maps $f$ of Riemann surfaces $\Sigma$ with non-empty boundary will be required to satisfy the Lagrangian boundary condition $f(\partial \Sigma \setminus \{ z'_1, \ldots, z'_{m'}\}) \subset L$. We say that $f$ is asymptotic to the orbit $\gamma_+ \in \PP(Y)$ of period $T_+$ at the interior puncture $z_+$ decorated with the marker $\ell_+$ at $+\infty$ if
\begin{itemize}
\item $f$ maps a pointed neighborhood $U_+$ of $z_+$ into $[0, \infty) \times Y$, so that $f(z) = (a(z), u(z))$ for all $z \in U_+$;
\item  $\lim_{z \to z_+} a(z) = + \infty$;
\item in holomorphic polar coordinates $(\rho, \theta)$ centered at $z_+$ and such that $\theta=0$ along $\ell_+$, $\lim_{\rho \to 0} u(\rho, \theta) = \gamma_+\left(-\frac{T_+}{2\pi} \theta\right)$ for the parametrization $\gamma_{+}\colon[0,T_+]\to Y$ of the Reeb orbit $\gamma_+$ determined by $p_{\gamma_+}\in\ug_+$.
\end{itemize}
Similarly, we say that $f$ is asymptotic to the orbit $\gamma_- \in \PP(Y_0)$ of period $T_-$ at the interior puncture $z_-$ decorated with the marker $\ell_-$ at $-\infty$ if
\begin{itemize}
\item $f$ maps a pointed neighborhood $U_-$ of $z_-$ into $(-\infty,0] \times Y_0$, so that
$f(z) = (a(z), u(z))$ for all $z \in U_-$;
\item  $\lim_{z \to z_-} a(z) = - \infty$;
\item in holomorphic polar coordinates $(\rho, \theta)$ centered at $z_-$ such that $\theta = 0$ along $\ell_-$, $\lim_{\rho \to 0} u(\rho, \theta) = \gamma_-\left(\frac{T_-}{2\pi} \theta\right)$ for the parametrization $\gamma_{-}\colon[0,T_-]\to Y_0$ of the Reeb orbit $\gamma_-$ determined by $p_{\gamma_-}\in\ug_-$.
\end{itemize}
We say that $f$ is asymptotic to the chord $c_+ \in \CC(\Lambda)$ of length $T_+$ at the boundary puncture $z'_+$ at $+\infty$ if
\begin{itemize}
\item $f$ maps a pointed neighborhood $U_+$ of $z'_+$ into $[0, \infty) \times Y$, so that $f(z) = (a(z), u(z))$
for all $z \in U_+$;
\item  $\lim_{z \to z'_+} a(z) = + \infty$;
\item in holomorphic polar coordinates $(\rho, \theta)$, $\theta\in[-\pi,0]$ centered at $z'_+$ (where $\theta\in\{-\pi, 0\}$ along $\partial \Sigma$), $\lim_{\rho \to 0} u(\rho, \theta) = c_+\left(-\frac{T_+}{\pi} \theta\right)$.
\end{itemize}
Similarly, we say that $f$ is asymptotic to the chord $c_- \in \CC(\Lambda_0)$ of length $T_-$ at the boundary puncture $z'_-$ at $-\infty$ if
\begin{itemize}
\item $f$ maps a pointed neighborhood $U_-$ of $z'_-$ into $[0, \infty) \times Y$, so that $f(z) = (a(z), u(z))$
for all $z \in U_-$;
\item  $\lim_{z \to z'_-} a(z) = - \infty$;
\item in holomorphic polar coordinates $(\rho, \theta)$, $\theta\in[0,\pi]$ centered at $z'_-$ (where $\theta\in\{0,\pi\}$ along $\partial \Sigma$), $\lim_{\rho \to 0} u(\rho, \theta) = c_-\left(\frac{T_-}{\pi} \theta\right)$.
\end{itemize}

Suppose now that there exists an exact Lagrangian cobordism $L_0\subset X_0$ with positive end $\Lambda_-$ and empty negative end.   We fix adjusted almost complex
structures on $W$ and $X_0$ which agree on the common part.
Consider a punctured decorated  Riemann surface $\Sigma$ with (possibly empty) boundary.
A holomorphic map $f : \Sigma \to W$ {\em anchored in $(X_0, L_0)$}, satisfying some specified conditions on the boundary (e.g.~a Lagrangian boundary condition) and at the punctures of $\Sigma$ (e.g.~being asymptotic to specified Reeb orbit cylinders), consists of the following objects:
\begin{description}
\item{(i)} a holomorphic map $f : (\Sigma \setminus \{ x_1, \ldots, x_l \}, \pa\Sigma \setminus \{ y_1, \ldots, y_m \}) \to W$ satisfying the specified conditions on the boundary and at the punctures of $\Sigma$, which is asymptotic to orbits $\delta_1,\dots\delta_l \in\PP(Y_0)$ at the additional interior, decorated punctures $x_1,\dots, x_l$ at $-\infty$, and which is asymptotic to chords $c'_1,\dots c'_m  \in\CC(\Lambda_-)$ at the additional boundary punctures $y_1,\dots, y_m$ at $-\infty$;
\item{(ii)} holomorphic planes $h_j:\C\to X_0$ asymptotic to the orbits $\delta_j$, $j=1,\dots, l$, at $+\infty$
(here the puncture is $z=\infty$ and the asymptotic marker is the positive real axis);
\item{(iii)} holomorphic half-planes $h'_j:(D^2,\pa D^2 \setminus \{ x\}) \to (X_0, L_0)$ asymptotic to the chords $c'_j$, $j=1, \ldots, m$, at $+\infty$.
\end{description}
%
%
When there are no additional boundary punctures and no holomorphic half-planes, we say that the holomorphic map $f$ is {\em anchored in $X_0$}. See Figure \ref{fig:anchor}.
\medskip

At one occasion, we will also consider a slightly extended  class of admissible anchors. In addition to (ii) and (iii) above, we also allow
\begin{description}
\item{(iv)}   holomorphic strips $h''_j:(D^2,\pa D^2 \setminus \{ x,y\}) \to (X_0, L_0)$ asymptotic at $+\infty$ to the chords $c'_j$, $c'_i$ $i,j=1, \ldots, m$, $i\neq j$, provided that the punctures $y_i$ and $y_j$ belong to different connected components of the holomorphic curve $f$.
\end{description}
When such anchors are present we say that the holomorphic curve  $f$ has {\em generalized anchors} in $(X_0,L_0)$. We use disks with generalized anchors only   in 
the definition of the moduli space $\MM^Y_{\Lambda, L}(c_1,c_2;b)$ in Section \ref{sec:mYLambda}.  In turn, this moduli space is only needed at one point in Section \ref{sec:LH-surgery} for the definition of the product on linearized Legendrian homology.
 
 %
%
Let $f : \Sigma \to W$ and $\wt f : \wt\Sigma \to W$ be two holomorphic maps anchored in $(X_0, L_0)$.
We say that $f$ and $\tilde f$ are equivalent if $\Sigma$ and $\wt\Sigma$ have the same number
of additional internal and boundary punctures, and if there exist:
\begin{itemize}
\item a biholomorphism $\varphi : \Sigma \to \wt\Sigma$ such that $\varphi$ maps the punctures of $\Sigma$
with their asymptotic markers to the punctures of $\wt\Sigma$ with their asymptotic markers, and such that 
$\wt f = f \circ \varphi$;
\item biholomorphisms $\varphi_j : \C \to \C$, $j=1,\ldots,l$ preserving the direction of the positive real axis 
such that $\wt h_j = h_j \circ \varphi_j$;
\item biholomorphisms  $\varphi'_j : D^2 \to D^2$, $j=1,\ldots,m$ such that $\wt h'_j = h'_j \circ \varphi'_j$.
\end{itemize}
The various moduli spaces of holomorphic curves considered in this paper will always consist of
equivalence classes of holomorphic maps.

Let us note that when the positive asymptotics of a holomorphic curve is fixed then the total action of the orbits and  chords at the positive end bounds by Stokes' formula the area of the curve, and hence provides an a priori upper bound on the possible number of anchors.

\begin{figure}
\labellist
\small\hair 2pt
\pinlabel $\gamma$   at 40 302
\pinlabel $\beta$   at 56 41
\pinlabel $\text{in $W$}$ at 220 198
\pinlabel $\text{in $X_0$}$ at 220 15
\pinlabel $\text{in $(W,L)$}$ at 640 198
\pinlabel $\text{in $(X_0,L_0)$}$ at 640 15
\pinlabel $c$ at 347 303
\pinlabel $b$ at 369 35
\endlabellist
\centering
\includegraphics[width=.6\linewidth]{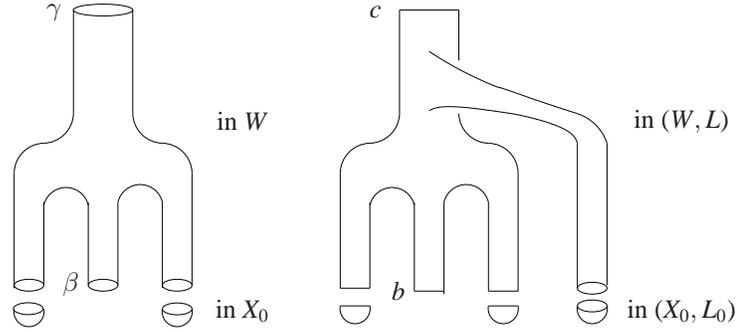}
\caption{On the left, a cylinder connecting $\gamma$ and $\beta$ anchored in $X_0$. On the right, a strip connecting $c$ to $b$ anchored in $(X_0,L_0)$.}
\label{fig:anchor}
\end{figure}

\subsection{\texorpdfstring{Moduli spaces $\MM^Y(\gamma;\beta)$ and ${\widecheck\MM}{}^Y(\gamma;\beta)$ } {Moduli spaces MY(gamma+,gamma-) and checkMY(gamma+,gamma-) }}\label{sec:holcyl+MBcyl}
Let $X$ be a Weinstein manifold with ideal contact boundary $Y$. Assume that $X$, $\R\times Y$, and the completion of any Weinstein cobordism in $X$ are endowed with almost complex structures adjusted to all considered symplectic objects, and that all these almost complex structures agree in the common parts of the spaces considered.

We denote by $\MM^Y(\gamma,\beta)$ the moduli space of  holomorphic  cylinders  in $\R \times Y$ anchored in $X$, asymptotic to orbits $\gamma$ at $+\infty$ and $\beta$ at $-\infty$. The dimension of $\MM^Y(\gamma;\beta)$ is equal to $|\gamma|-|\beta|$ and can be expressed as the sum of the dimensions of all the moduli spaces of holomorphic curves involved in the anchored curve. In particular, when the total dimension is equal to $1$, all the moduli spaces of holomorphic planes in $X$ which are involved are necessarily $0$-dimensional. The moduli space of punctured cylinders in $\R\times Y$ is then $1$-dimensional, and hence consists of holomorphic curves which are rigid up to translations.
In what follows all holomorphic cylinders in $\R\times Y$ will be allowed to anchor in $X$, and we will not always explicitly mention this.

The asymptotic marker $\ell_+$ at $+\infty$ fixes a choice of parametrization for the $S^1$ factor of the domain $\R \times S^1$, where $\ell_+$ corresponds to $0\in S^{1}=\R/(2\pi\Z)$, for elements of $\MM^Y(\gamma;\beta)$. Evaluation at $0\in S^{1}=\R/(2\pi\Z)$ at $-\infty$ then defines an evaluation map ${\ev}_0\colon\MM^Y(\gamma;\beta)\to \ul{\beta}$.
Define
\begin{equation}\label{eq:constr-cyl1}
\widecheck\MM{}^{\prime Y}(\gamma;\beta)={\ev}_0^{-1}(p_{\beta}).
\end{equation}
Then $\widecheck\MM{}^{\prime Y}(\gamma;\beta)$
can be interpreted as the moduli space of holomorphic cylinders such that the points $p_{\beta}$ and $p_{\gamma}$ are aligned along a single line $\R\times \{t\}\subset\R\times S^1$. In particular, one may equivalently describe the space $\widecheck\MM{}^{\prime Y}(\gamma;\beta)$ by using a different interpretation of \eqref{eq:constr-cyl1} where the marker at $-\infty$ determines the parametrization of $S^{1}$ and we use the preimage of $p_{\gamma}$ under the induced evaluation map ${\ev}_0\colon \MM^Y(\gamma;\beta)\to \ul{\gamma}$ instead.

Later we will use the moduli space $\widecheck\MM{}^{\prime Y}(\gamma;\beta)$ to describe parts of the differential in a chain complex which computes symplectic homology in Morse-Bott terms following \cite{BO} (see also \cite{F} in the case of Morse theory).
In these terms elements of $\widecheck\MM{}^{\prime Y}(\gamma;\beta)$ correspond to degenerate ``Morse-Bott'' curves with one holomorphic level only. There are similar contributions to the differential also from Morse-Bott curves with two holomorphic levels connected by a Morse flow level. We next define the moduli space $\widecheck\MM{}^{\prime \prime Y}(\gamma;\beta)$ relevant for describing such objects.

Let $\gamma,\gamma_0,\beta\in\PP(Y)$. Let $f_+\in\MM^{Y}(\gamma;\gamma_0)$ and $f_-\in\MM^{Y}(\gamma_0;\beta)$. Then, as discussed above, the markers $\ell_{+}$ at $+\infty$ in the domain of $f_+$ and  $\ell_{-}$ at $-\infty$ in the domain of $f_-$ give evaluation maps $\ev_0^{+}\colon \MM^{Y}(\gamma;\gamma_0)\to\ul{\gamma_0}$ and $\ev_0^{-}\colon\MM^{Y}(\gamma_0;\beta)\to\ul{\gamma_0}$, respectively. Let
\[
\Delta(\gamma,\gamma_0,\beta)=
(\MM^{Y}(\gamma;\gamma_0)/\R)\times\MM^{Y}(\gamma_0;\beta)
\]
and define 
$\widecheck\MM{}^{\prime \prime Y}(\gamma;\beta)$ as the subset of
pairs $(f_+, f_-) \in \Delta(\gamma,\gamma_0,\beta)$, for some $\gamma_0 \in \PP(Y)$, such that the points $p_{\gamma_0}$, $\ev_0^{+}(f_+)$, and $\ev_0^{-}(f_-)$ lie in the cyclic ordering
\[
\bigl(\,p_{\gamma_0}\,,\, \ev_0^{+}(f_+)\,,\, \ev_0^{-}(f_-)\,\bigr)
\]
on $\ul{\gamma_0}$ with orientation induced by the Reeb field $R_\alpha$.

Finally, we define the moduli space
$$
\widecheck\MM{}^Y(\gamma;\beta)=\widecheck\MM{}^{\prime Y}(\gamma;\beta) \cup
\widecheck\MM{}^{\prime \prime Y}(\gamma;\beta).
$$
We have $\dim  \widecheck\MM{}^Y(\gamma;\beta)=|\gamma|-|\beta|-1.$ See Figure \ref{fig:morsebott}.

\begin{figure}
\centering
\includegraphics[width=.6\linewidth]{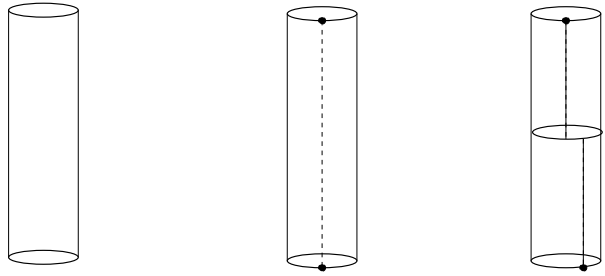}
\caption{From left to right, cylinders in $\MM$, in $\MM'$, and in $\MM''$.}
\label{fig:morsebott}
\end{figure}

\subsection{\texorpdfstring{Moduli spaces $\MM^X(\gamma)$ and $\MM^X(\gamma;p)$}{Moduli spaces MX(gamma) and MX(gamma,p)}}\label{sec:holplanes}
Let $H:X\to\R$ be any exhausting Morse function which in the end depends only on the coordinate $s$, and let $\wt  Z$  be a vector field which is gradient-like with respect to it and which at infinity coincides with $\frac{\p}{\p s}$.
If $X$ is Weinstein, we will assume that  $\wt Z$ is the Liouville vector field $Z$ and that  the function $H$
is the Lyapunov function for $Z$, which is part of the Weinstein structure. Let $L_p$ be  the $\wt Z$-stable manifold for a critical point $p\in X$. We have
$\codim L_p=2n-\ind p$.

Consider  the moduli space
$\MM^X(\gamma)$ of holomorphic maps $\C\to X$  with a marked point at $0\in\C$, which are asymptotic to an orbit $\gamma\in\PP(Y)$ at $+\infty$.
We have $\dim\MM^X(\gamma)=|\gamma|+n-1$ (taking into account the quotient by 
reparameterizations preserving $0$). Consider the evaluation map at $0\in\C$,
$\ev_0:\MM^{X}(\gamma)\to X$.
Define the moduli space
$$\MM^X(\gamma;p):=\ev^{-1}_0(L_p)\subset\MM^{X}(\gamma),$$
see Figure \ref{fig:intomorse}.
We have $$\dim\MM^{X}(\gamma;p)=|\gamma|+n-1+\ind p -2n=|\gamma|-|p|-1,$$ where we let $|p|:=n-\ind p$. Thus $|p|$ is the Morse index of $p$ as a critical point of the function $-H$, decreased by $n$.
In particular, the space $\MM^X(\gamma;p) $ consists of curves that are rigid up to reparameterization if and only if $|\gamma|-|p|=1$.

\begin{figure}
\labellist
\small\hair 2pt
\pinlabel $\gamma$   at 95 187
\pinlabel $p$   at 0 37
\pinlabel $\nabla H$ at 73 45
\endlabellist
\centering
\includegraphics[width=.4\linewidth]{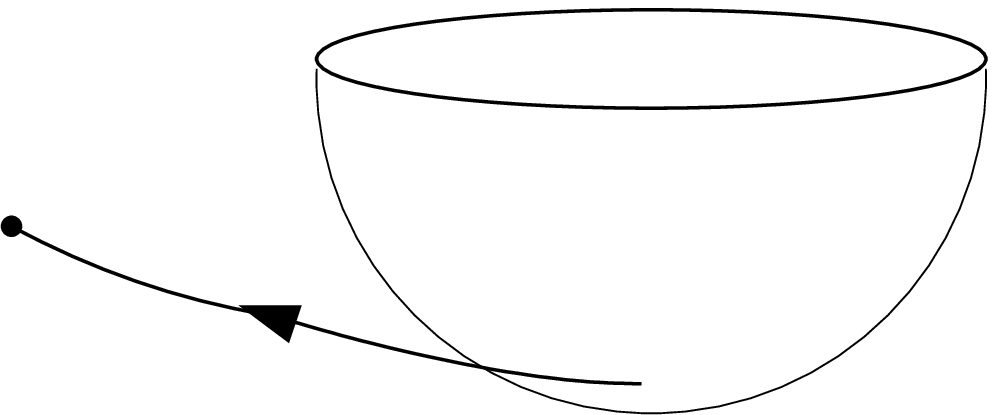}
\caption{A curve in $\MM^{X}(\gamma,p)$.}
\label{fig:intomorse}
\end{figure}

\subsection{\texorpdfstring{Moduli spaces $\MM^Y_\Lambda(c;b_1,\dots, b_m)$, $\MM^Y_{\Lambda,L}(c;b)$ and  $\MM^Y_{\Lambda, L}(c_1,c_2;b)$}{Moduli spaces MYLambda(c+;c1,..., cm) and MYLambda,L(c+,c-)}}\label{sec:mYLambda}
Consider a Legendrian submanifold $\Lambda\subset Y$. Let $D$ denote the unit disk in the complex plane. A {\em holomorphic tree} in $(\R\times Y,\R\times\Lambda)$, connecting a chord $c\in\CC(\Lambda)$ and chords
$b_1,\dots, b_m\in\CC(\Lambda)$, $m\geq0$, anchored into $X$ is a holomorphic map
$$
f\colon\left(D,\p D\setminus\{z^+, z_1,\dots, z_m\}\right)\to (\R\times Y,\R\times\Lambda),
$$
anchored in $X$, where  $z^+,z_1,\dots, z_m$, $m\geq 0$, are boundary punctures ordered counter-clockwise, such that
at $z^+$ the map is asymptotic to the chord $c$ at $+\infty$, and at the punctures $z_1,\dots, z_m$ the map is asymptotic to the chords $b_1,\dots, b_m$ at $-\infty$.
The moduli space  of holomorphic trees in $(\R\times Y,\R\times\Lambda)$, connecting a chord $c\in\CC(\Lambda)$ and chords
$b_1,\dots, b_m\in\CC(\Lambda)$, $m\geq0$, anchored in $X$, will be denoted by
$\MM^Y_\Lambda(c;b_1,\dots, b_m)$.
We have $$\dim\MM^Y_\Lambda(c;b_1,\dots,b_m)=|c|-\sum\limits_{j=1}^m |b_j|.$$

We will also consider the  moduli spaces $\MM^Y_\Lambda(c;b_1,\dots,b_m;k)$, $k=0,1,\dots$,
of holomorphic trees in $(\R\times Y,\R\times\Lambda)$, connecting a chord $c\in\CC(\Lambda)$ to chords $b_1,\dots, b_m\in\CC(\Lambda)$, $m\geq0$, anchored in $X$, and {\em  with $k$ additional marked points on the boundary}, intertwined in an arbitrary way with the punctures $z_j$. Hence,
the moduli space $\MM^Y_\Lambda(c;b_1,\dots, b_m;k)$ can be further specialized as
\begin{align}\label{eq:deform}
\MM^Y_\Lambda(c;b_1,\dots, b_m;k)=\coprod\limits_{\sum k_j=k} \MM^Y_\Lambda(c;b_1,\dots, b_m;k_0,\dots, k_m),
\end{align}
where $ \MM^Y_\Lambda(c;b_1,\dots, b_m;k_0,\dots, k_m)$ consists  of the trees with exactly $k_0$ marked points preceding $b_1$ and exactly $k_j$ marked points following $b_j$, $j=1,\dots,m$.
Note that in this notation
$\MM^Y_\Lambda(c;b_1,\dots, b_m;0)= \MM^Y_\Lambda(c;b_1,\dots, b_m)$. Evaluation at the marked points give evaluation maps
\[
\ev_k:\MM^Y_\Lambda(c;b_1,\dots, b_m;k)\to\underbrace{\Lambda\times\dots\times\Lambda}_k.
\]

Suppose now that there exists an exact Lagrangian cobordism $L\subset X$ with positive end $\Lambda\subset Y$. Then we can consider the moduli space
$\MM^Y_{\Lambda,L}(c;b)$   of holomorphic strips in $(\R\times Y,\R\times\Lambda)$ anchored in $(X,L)$, with positive puncture at $c$ and negative puncture at $b$, see Section \ref{sec:anchored} for the definition of holomorphic curves anchored in $(X,L)$. We have
$$
\dim \MM^Y_{\Lambda,L}(c;b) =|c|-|b|.
$$

We also define the moduli spaces $\MM^Y_{\Lambda, L}(c_1,c_2;b)$ in the case when $\Lambda$ is a union of spheres. This moduli space is itself the union of two moduli spaces:
$$
\MM^Y_{\Lambda, L}(c_1,c_2;b)={}'\!\MM^Y_{\Lambda, L}(c_1,c_2;b)\cup{}''\!\MM^Y_{\Lambda, L}(c_1, c_2;b).
$$ 
Here ${}'\!\MM^Y_{\Lambda, L}(c_1, c_2;b)$ is the moduli space  of holomorphic disks in  $(\R\times Y,\R\times\Lambda)$ with {\em generalized} anchors in $(X,L)$ (see  Section
\ref{sec:anchored}), which have three boundary punctures: two positive punctures mapped to the chords $c_1$ and  $c_2$, and one negative puncture mapped to $b$. Our notation is such that the cyclic order or of the punctures induced by the boundary orientation is $(b,c_1,c_2)$. 

To define ${}''\!\MM^Y_{\Lambda, L}(c_1, c_2;b)$, we fix an auxiliary Morse function  $\wt g\colon\Lambda\to \R$,  with one minimum, one maximum, and no other critical points on each component. Define the function  $g\colon\R\times\Lambda \to\R$ as $g=\wt g+s$, $s\in\R$. We also consider a Morse function $L\to\R$, also denoted by $g$, which at infinity coincides with $\wt g+s$.  Define 
\begin{equation}\label{e:extramorse}
{}''\!\MM^Y_{\Lambda, L}(c_1, c_2;b):=\MM^{Y}_\Lambda(c_1;b)\times_g \MM^{Y}_\Lambda(c_2)\cup
\MM^{Y}_\Lambda(c_1)\times_g \MM^{Y}_\Lambda(c_2;b), 
\end{equation}
where an element  in $\MM^{Y}_{\Lambda}(c_1;b)\times_g \MM^{Y}_{\Lambda}(c_2)$ consists  of a pair of holomorphic disks $(u,v)\in\MM^{Y}_\Lambda(c_1;b)\times\MM^{Y}_\Lambda(c_2)$ and a flow line of $\nabla g$ which connects the boundary of $u$ to the boundary of $v$, and where elements in $\MM^{Y}_\Lambda(c_1)\times_g \MM^{Y}_\Lambda(c_2;b)$ are analogous configurations where the negative puncture is in the second disk instead.
\footnote{Let us point out that the gradient flow line may lie in $L$ and connect points on two anchor disks.}
We have
$$
\dim \MM^Y_{\Lambda,L}(c_1, c_2;b) =|c_1|+|c_2|-|b|-(n-3).
$$
\subsection{\texorpdfstring{Moduli spaces $\MM^Y_\Lambda(\gamma;b_1,\dots, b_m)$, $\widecheck\MM{}^Y_\Lambda(\gamma;b_1,\dots, b_m)$, \\ and
$\widehat\MM^Y_\Lambda(\gamma;b_1,\dots, b_m)$}{Moduli spaces MYLambda(gamma;c1,..., cm), checkMYLambda(gamma;c1,..., cm), and hatMYLambda(gamma;c1,..., cm)}}\label{sec:mLambda}
We denote by $\MM^Y_\Lambda(\gamma;b_1,\dots, b_m)$ the moduli space of holomorphic maps
$$
f\colon (D\setminus \{0\}, \p D\setminus\{z_1,\dots, z_m\})\to(\R\times Y,\R\times\Lambda)
$$
anchored in $X$, where $z_1,\dots, z_m\in\p D$ are boundary punctures ordered counter-clockwise, which are asymptotic to the orbit $\gamma \in \PP_\good(Y)$ at the decorated puncture $0$ at $+\infty$, and  to chords $b_1,\dots, b_m \in \CC(\Lambda)$ at punctures $z_1,\dots, z_m$ at $-\infty$. When $m=0$, the moduli space $\MM^Y_\Lambda(\gamma;\varnothing)$ will simply be denoted
by $\MM^Y_\Lambda(\gamma)$.

The interior puncture $0$ and one boundary puncture, which we take to be $z_1$, determines a coordinate system on $D$ and gives, in particular, an identification $\pa D=S^{1}=\R/(2\pi\Z)$ where we think of $S^{1}$ as the polar coordinate circle at $0$ and where $z_1$ corresponds to $0\in\R$. The asymptotic marker $\ell_+$ at $0$ then induces an evaluation map
$$
\ev_0\colon\MM^Y_\Lambda(\gamma;b_1,\dots,b_m) \to \pa D=S^{1},
$$
which associates to a map $u$ the coordinate of $\ell_+$ in the polar $S^{1}$.
Define
$$
\widecheck\MM{}^{\prime Y}_\Lambda(\gamma;b_1,\dots, b_m):=\ev^{-1}_0(z_1)=\ev_0^{-1}(0)
$$
and
$$
\widehat\MM^Y_\Lambda(\gamma;b_1,\dots, b_m):= \ev^{-1}\left((z_1,z_2)\right),
$$
where $(z_1,z_2)$ denotes the open arc in the boundary $\pa D$ which starts at $z_1$ and ends at $z_2$.
Thus $\widecheck\MM{}^{\prime Y}_\Lambda(\gamma;b_1,\dots, b_m)$ (respectively $\widehat\MM^Y_\Lambda(\gamma;b_1,\dots, b_m)$) is a subspace of  the moduli space $\MM^Y_\Lambda(\gamma;b_1,\dots, b_m)$ which is distinguished by the condition that the ray of asymptotic marker mapped to $p_\gamma$ intersects $\pa D$ at the point $z_1$ (respectively at a point of the open arc $(z_1,z_2)$). See Figure \ref{fig:isocurve}.

Note that the moduli spaces $\MM^Y_\Lambda(\gamma;b_1,\dots, b_m)$ and $\MM^Y_\Lambda(\gamma;b'_1,\dots, b'_m)$ coincide
if the chords $b_1,\dots, b_m$ and $b'_1,\dots, b'_m$ differ by a cyclic permutation. We can, therefore, limit ourselves to the study of the above moduli spaces with the specific boundary puncture $z_1$ and boundary arc $(z_1, z_2)$. Note also that $\widecheck\MM{}^{Y}_{\Lambda}(\gamma;b_1,\dots,b_m)$ could be defined by instead using the asymptotic direction $\ell_+$ at $0$ to define a coordinate system on $D$ and then letting $z_1$ induce an evaluation map $\ev_1\colon \MM^{Y}_{\Lambda}(\gamma;b_1,\dots,b_m)\to S^{1} $ which maps $f$ to the coordinate of the puncture $z_1$ and taking the preimage of $0$ under $\ev_1$.

As in the definition of the moduli spaces of Morse-Bott curves with two holomorphic levels described in Section \ref{sec:holcyl+MBcyl} we let
\[
\Delta(\gamma,\gamma_0,b_1\dots b_m)= \MM^Y(\gamma;\gamma_0)/\R \times
\MM^Y_\Lambda(\gamma_0;b_1,\dots, b_m).
\]
If $(f_+,f_-)\in\Delta(\gamma,\gamma_0,b_1\dots b_m)$ then the asymptotic marker $\ell_+$ at the positive puncture of $f_+$ determines an evaluation map $\ev_{0}^{+}\colon \MM^Y(\gamma; \gamma_0)/\R \to \ul{\gamma_0}$ and after composition with $f_-$, the evaluation map $\ev_1$ on $\MM^{Y}_\Lambda(\gamma;b_1,\dots,b_m)$ discussed above gives an evaluation map $\ev_1^{-}\colon \MM^Y_\Lambda(\gamma_0;b_1,\dots, b_m)\to\ul{\gamma_0}$.

Let $\widecheck\MM{}^{\prime \prime Y}_\Lambda(\gamma;b_1,\dots,b_m)$ be the moduli space of 
pairs $(f_+, f) \in \Delta(\gamma,\gamma_0,b_1\dots b_m)$,
for some $\gamma_0 \in \PP$, such that the points $p_{\gamma_0}$, $\ev_0^{+}(f_+)$, and
$\ev_1^{-}(f_{-})$ lie in the cyclic ordering
\[
\bigl(\,p_{\gamma_0}\,,\, \ev_0^{+}(f_+)\,,\, \ev_1^{-}(f_{-})\,\bigr)
\]
on $\ul{\gamma_0}$ with orientation induced by the Reeb field $R_\alpha$.

Finally, we define the moduli space
$$
\widecheck\MM{}^Y_\Lambda(\gamma;b_1,\dots, b_m) = \widecheck\MM{}^{\prime Y}_\Lambda(\gamma;b_1,\dots, b_m) \cup \widecheck\MM{}^{\prime \prime Y}_\Lambda(\gamma;b_1,\dots, b_m).
$$

We have $$\dim \MM^Y_\Lambda(\gamma;b_1,\dots, b_m)=\dim \widehat\MM^Y_\Lambda(\gamma;b_1,\dots, b_m)=
|\gamma|-\sum\limits_{j=1}^m |b_j|,$$
and $$\dim \widecheck\MM{}^Y_\Lambda(\gamma;b_1,\dots, b_m)=
|\gamma|-\sum\limits_{j=1}^m |b_j|-1.$$

\begin{figure}
\labellist
\small\hair 2pt
\pinlabel $\ell_+$   at 345 88
\pinlabel $z_1$   at 415 71
\pinlabel $\ell_+$   at 600 100
\pinlabel $z_1$   at 675 71
\pinlabel $z_2$   at 623 155
\endlabellist
\centering
\includegraphics[width=.7\linewidth]{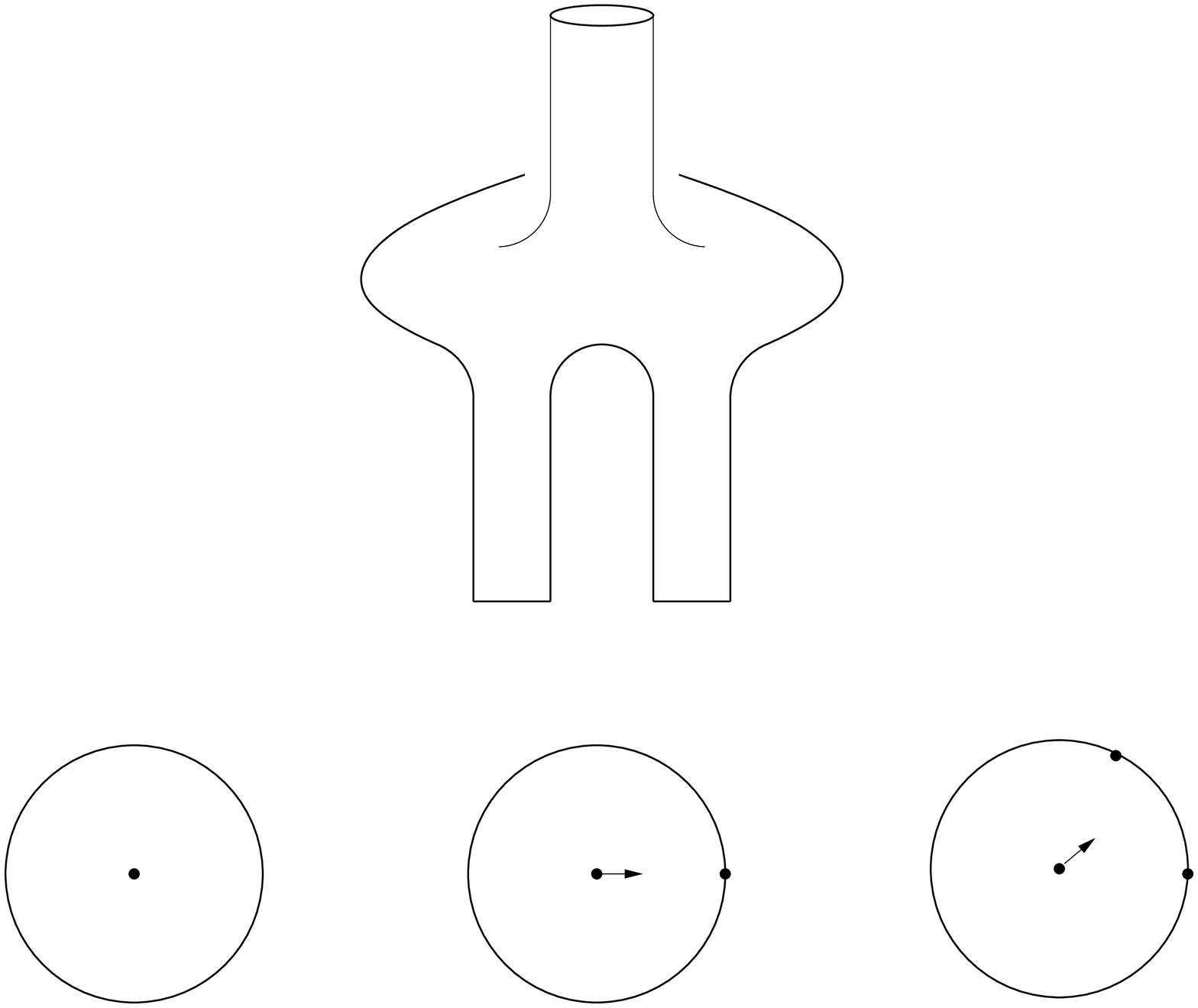}
\caption{A curve in $\MM^{Y}_\Lambda(\gamma;b_1,b_2)$. Lower line shows positions for markers in the source for $\MM^{Y}_\Lambda(\gamma;b_1,b_2)$, $\widecheck\MM{}^{Y}_\Lambda(\gamma;b_1,b_2)$, and $\widehat\MM^{Y}_\Lambda(\gamma;b_1,b_2)$}
\label{fig:isocurve}
\end{figure}

\subsection{\texorpdfstring{Moduli spaces $\MM^W(\gamma;\beta)$ and $\widecheck\MM{}^W(\gamma;\beta)$}{Moduli spaces MW(gamma,gamma0) and checkMW(gamma,gamma0)}}\label{sec:mW}
Let $X,X_0$ and $W$ be as in Section \ref{sec:anchored}. We fix adjusted almost complex structures on $W$ and $X_0$ which agree on the common part.

Given two orbits $\gamma\in \PP(Y)$ and $\beta\in\PP(Y_0)$ we define,  similar to the way it was done  above,   the moduli space $\MM^W(\gamma;\beta)$ of holomorphic cylinders  in $W$, anchored in $X_0$, connecting $\gamma$ at $+\infty$ with $\beta$ at $-\infty$. We next define moduli spaces of Morse-Bott curves (in complete analogy with the corresponding definitions in Section \ref{sec:holcyl+MBcyl}) as follows:

The marker at $+\infty$ fixes a parametrization for the $S^{1}$-factor in the domain $\R\times S^{1}$ of elements in $\MM^{W}(\gamma;\beta)$ which gives an evaluation map $\ev_0\colon \MM^W(\gamma;\beta) \to \ul{\beta}$ and we let
\[
\widecheck\MM{}^{\prime W}(\gamma;\beta)=\ev^{-1}(p_{\beta}).
\]

Further, the markers at $+\infty$ and $-\infty$ of elements in $\MM^{W}(\gamma;\gamma_0)$ and $\MM^{Y_0}(\gamma_0;\beta)$, respectively, give evaluation maps
\[
\ev_0\colon \MM^W(\gamma;\gamma_0)\to \ul{\gamma_0}\quad \text{and}\quad
\ev_0^{-}\colon \MM^{Y_0}(\gamma_0;\beta)\to \ul{\gamma_0}.
\]
We let $\widecheck\MM{}^{\prime \prime W}(\gamma;\beta)$ denote the set of
pairs $(f,f_-) \in \MM^W(\gamma, \gamma_0)  \times \MM^{Y_0}(\gamma_0,\beta)/\R$,
for some $\gamma_0 \in \PP(Y_0)$, such that $\ev_0(f)$, $\ev_0^-(f_-)$ and $p_{\gamma_0}$, lie in the cyclic order $\bigl(p_{\gamma_0}\,,\ev_0(f)\,,\ev_0^{-}(f_-)\bigr)$ on $\ul{\gamma_0}$.

Similarly, we use markers at $+\infty$ and $-\infty$ of elements in $\MM^{Y}(\gamma;\gamma_0)$ and $\MM^{W}(\gamma_0;\beta)$ to define evaluation maps
\[
\ev_0^+\colon \MM^Y(\gamma,\gamma_0)\to \ul{\gamma_0}\quad \text{and}\quad
\ev_0\colon \MM^W(\gamma_0,\beta)\to \ul{\gamma_0}.
\]
We then let $\widecheck\MM{}^{\prime \prime \prime W}(\gamma;\beta)$ consist of pairs
$(f_+,f) \in \MM^Y(\gamma; \gamma_0)/\R \times \MM^W(\gamma_0, \beta)$,
for some $\gamma_0 \in \PP(Y)$, such that $\ev_0^+(f_+)$, $\ev_0(f)$ and $p_{\gamma_0}$ lie in the cyclic order $\bigl(p_{\gamma_0}\,,\ev_0^+(f_+)\,,\ev_0(f)\bigr)$ on $\ul{\gamma_0}$.

Finally, we define the moduli space
\[
\widecheck\MM{}^{W}(\gamma;\beta)=
\widecheck\MM{}^{\prime W}(\gamma;\beta) \,\cup\, \widecheck\MM{}^{\prime \prime W}(\gamma;\beta) \,\cup\,
\widecheck\MM{}^{\prime \prime \prime W}(\gamma;\beta).
\]
We have $\dim(\MM^W(\gamma;\beta))=|\gamma|-|\beta| $ and $\dim(\widecheck\MM{}^W(\gamma;\beta))=|\gamma|-|\beta| -1$.

\subsection{\texorpdfstring{Moduli spaces $\MM^{W}_L(c;b_1,\dots,b_m)$, $\MM^{W}_{L}(c;p)$, and $\MM^{W}_{L}(c_1,c_2;p)$}{Moduli spaces MWL(c+;c1,...,cm) and MWL(c,p)}}\label{sec:mLL}
Let   $ L\subset W$ be  an exact  Lagrangian cobordism  between the Legendrian submanifolds $\Lambda_- \subset Y_0$ and $\Lambda_+\subset Y$.  Given a chord $c\in\CC(\Lambda_+)$ and chords $b_1,\dots,b_m\in\CC(\Lambda_-)$ we denote by
\[
\MM^{W}_L(c;b_1,\dots,b_m)
\]
the moduli space of holomorphic maps
$$
f\colon(D, \p D\setminus\{z^+,z_1,\dots, z_m\})\to(W,L),
$$
anchored in $X_0$, where  $z^+,z_1,\dots, z_m$, $m\geq 0$, are boundary punctures ordered counter-clockwise, such that
at $z^+$ the map is asymptotic to the chord $c$ at $+\infty$, and at the punctures $z_1,\dots, z_m$ the map is asymptotic to the chords $b_1,\dots, b_m$ at $-\infty$. We have
$$\dim\MM^W_L(c;b_1,\dots,b_m)=|c|-\sum\limits_{j=1}^m |b_j|.$$

Suppose that $\Lambda_-=\varnothing$ and assume  that the restriction $ H|_L$ is a Morse function without local maxima. Given a chord $c\in\CC(\Lambda_+)$, let
$\MM^W_L(c)$ denote the moduli space of holomorphic disks
$f\colon (D,\p D\setminus\{z\})\to (W,L)$ anchored in $X_0$, which are asymptotic to the chord $c$ at the puncture $z$. Let $\ev\colon \MM^W_L(c)\to L$ be the evaluation map at the point $-z\in\p D$ opposite to the marked point $z$.
We will assume that $\ev$ is transverse to the stable manifolds $V_p$ of all critical points $p$ of the function $H|_L$ and,
given a critical point $p\in L$ of the function $H|_L$, we let
$$
\MM^W_L(c,p):=\ev^{-1}(V_p).
$$
We have
\begin{equation}\label{eq:dim-Lag-ev}
\dim(\MM^{W}_{L}(c,p))=|c|+1-n+\ind(p),
\end{equation}
where $\ind(p)$ is the index of the critical point $p$ of the function $H|_L$.
In particular, when $\ind(p)=0$ then $\dim(\MM^{W}_{L}(c,p))=|c|-n+1$.

Next, we define the moduli space
$\MM^W_L(c_1,c_2)$ as the union ${}'\!\MM^W_L(c_1,c_2)\cup{}''\!\MM^W_L(c_1,c_2)$, where  
${}'\!\MM^W_L(c_1,c_2)$ denotes the moduli space of holomorphic disks
$$
f\colon(D,\p D\setminus\{z_1,z_2\})\to (W,L)
$$ 
anchored in $X_0$, which are asymptotic at $+\infty$ to the chords $c_1$ and $c_2$ at the punctures $z_1$ and $z_2$, respectively. Let $\ev\colon {}'\!\MM^W_L(c_1,c_2)\to L$ be the evaluation map at point $z\in\p D$ midway between $z_1$ and $z_2$ (or midway between $z_2$ and $z_1$). Under similar transversality conditions as above, we let
$$
{}'\!\MM^W_L(c_1,c_2;p):=\ev^{-1}(V_p).
$$
To define ${}''\!\MM^W_L(c_1,c_2)$ we fix, similarly to the way it was done in Section \ref{sec:mYLambda},  an auxiliary Morse function  $\wt g\colon\Lambda\to \R$  with   one minimum and  one maximum and define  the function  $g\colon L\to\R$     which at the cylindrical end $[0,\infty)\times\Lambda$ of $L$ coincides with $\wt g+s$, $s\in\R$. 
Define
 \begin{equation}\label{e:extramorse2}
{}''\!\MM^Y_{\Lambda, L}(c_1 ,c_{2};p):=\MM^{Y}_\Lambda(c_1;p)\times_g \MM^{Y}_\Lambda(c_2)\cup
\MM^{Y}_\Lambda(c_1 )\times_g \MM^{Y}_\Lambda(c_2;p), 
\end{equation}
where an element  in $\MM^{Y}_{\Lambda}(c_1;p)\times_g \MM^{Y}_{\Lambda}(c_2)$ consists  of a pair of holomorphic disks $(u,v)\in\MM^{Y}_\Lambda(c_1;p)\times\MM^{Y}_\Lambda(c_2)$ and a flow line of $\nabla g$ which connects the boundary of $u$ to the boundary of $v$, and where elements in $\MM^{Y}_{\Lambda}(c_1)\times_g \MM^{Y}_{\Lambda}(c_2;p)$ are analogous configurations where the negative puncture is in the second disk instead.
We have
\begin{equation}\label{eq:dim-Lag-ev-2-pos}
\dim(\MM^{W}_{L}(c_1,c_2;p))=|c_1|+|c_2|-(2n-4)+\ind(p).
\end{equation}

\subsection{\texorpdfstring{Moduli spaces $\MM^W_L(\gamma;b_1,\dots, b_m)$, $\widecheck\MM{}^W_L(\gamma;b_1,\dots, b_m)$,\\ and $\widehat\MM^W_L(\gamma;b_1,\dots, b_m)$}{Moduli spaces MWL(gamma;c1,..., cm), checkMWL(gamma;c1,..., cm), and hatMWL(gamma;c1,..., cm)}}\label{sec:mWL}
Suppose now that $W$ is a Weinstein cobordism,   $p_1,\dots p_k\in\Int\oW $ are critical points  of index $n$ of the Lyapunov function $H$ for the Liouville vector field $Z$, and that there are no other critical points of $H$  in $W$. We assume general position, so that there are no $Z$-trajectory connections between critical points $p_1,\dots, p_k$. Let $L_1,\dots, L_k$ be  Lagrangian stable manifolds of the critical points $p_1,\dots, p_k$, and $\Lambda_j=L_j\cap Y_0$, $j=1,\dots,k$ be the corresponding Legendrian spheres. We write $\Lambda:=\bigcup\limits_{j=1}^k\Lambda_j$ and $L:=\bigcup\limits_{j=1}^k L_j$. Note that $L$ is an exact Lagrangian cobordism between $\Lambda_-=\Lambda$ and $\Lambda_+=\varnothing$.

Given an orbit $\gamma\in\PP(Y)$ and chords $b_1,\dots, b_m\in\CC(\Lambda)$ we define
the moduli space $\MM^W_L(\gamma;b_1,\dots, b_m)$  consisting of holomorphic maps
$$
f\colon(D\setminus \{0\}, \p D\setminus\{z_1,\dots, z_m\})\to(W,L)
$$
anchored in $X_0$, where $z_1,\dots z_m\in\p D$ are boundary punctures,
which are asymptotic to the orbit $\gamma$ at the decorated puncture $0$ at $+\infty$, and to chords $b_1,\dots, b_m$ at
punctures $z_1,\dots, z_m$ at $-\infty$.
Similarly to the construction in Section \ref{sec:mLambda}, using $z_1$ and $0$ to fix coordinates on $D$, the location of the asymptotic marker at 0 induces an evaluation map
$$
\ev_0\colon\MM^W_L(\gamma;b_1,\dots, b_m)\to \p D = S^1.
$$
Define
$$
\widecheck\MM{}^{\prime W}_L(\gamma;b_1,\dots, b_m):=\ev^{-1}_0(z_1)=\ev_0^{-1}(0)
$$
and
$$
\widehat\MM^W_L(\gamma;b_1,\dots, b_m):= \ev^{-1}_0((z_1,z_2)).
$$
Thus $\widecheck\MM{}^{\prime W}_L(\gamma;b_1,\dots, b_m)$ (respectively $\widehat\MM^W_L(\gamma;b_1,\dots, b_m)$) is a subspace of  the moduli space $\MM^W_L(\gamma;b_1,\dots, b_m)$ which is distinguished by the condition that the ray of the asymptotic marker intersects $\pa D$ at $z_1$ (respectively at a point of the open arc $(z_1,z_2)$). Since the chords $b_1, \ldots, b_m$ are listed up to a cyclic permutation, we again restrict ourselves to the boundary puncture $z_1$ and the boundary arc $(z_1,z_2)$.

The asymptotic marker at $+\infty$ of an element in $\MM^{Y}(\gamma;\gamma_0)$ gives an evaluation map $\ev_0^{+}\colon\MM^{Y}(\gamma,\gamma_0)\to\ul{\gamma_0}$. Using the evaluation map $\ev_1\colon\MM^W_L(\gamma_0;b_1,\dots, b_m)\to\ul{\gamma_0}$ discussed above,  
we define the moduli space $\widecheck\MM{}^{\prime \prime W}_L(\gamma;b_1,\dots,b_m)$ as the set of pairs
$(f_+, f) \in\MM^Y(\gamma; \gamma_0)/\R \times \MM^W_L(\gamma_0;b_1,\dots, b_m)$,
for some $\gamma_0 \in \PP(Y)$, such that the points
$p_{\gamma_0}$, $\ev_0^+(f_+)$, and $\ev_1(f)$ lie in the cyclic order $\bigl(p_{\gamma_0}\,,\ev_0^{+}(f_+)\,,\ev_1(f)\bigr)$ on $\ul{\gamma_0}$.

Finally, we define the moduli space
$$
\widecheck\MM{}^W_L(\gamma;b_1,\dots, b_m) = \widecheck\MM{}^{\prime W}_L(\gamma;b_1,\dots, b_m) \cup \widecheck\MM{}^{\prime \prime W}_L(\gamma;b_1,\dots, b_m).
$$
We have
$$
\dim(\MM^W_L(\gamma;b_1,\dots, b_m))=
\dim(\widehat\MM^W_L(\gamma;b_1,\dots, b_m))=
|\gamma|-\sum\limits_{j=1}^m |b_j|,
$$
and
$$
\dim(\widecheck\MM{}^W_L(\gamma;b_1,\dots, b_m))=
|\gamma|-\sum\limits_{j=1}^m |b_j|-1.
$$
\subsection{\texorpdfstring{Moduli spaces $ \MM^W_{L_j}(\gamma)$ and  $\MM^W(\gamma;p)$}{Moduli spaces MWLj(gamma) and  MW(gamma,p)}}\label{sec:evalinW}
We use notation as in Section \ref{sec:mWL}. Take $\gamma\in\PP(Y)$.
Then $\MM^W_{L_j}(\gamma)$ is the moduli space of  holomorphic maps $f\colon (D\setminus\{0\},\p D)\to (W,L_j)$ anchored in $X_0$,
which are asymptotic at $0$ to the orbit $\gamma$ at $+\infty$.

Next, for $\gamma\in\PP(Y)$ we define the moduli space $\MM^W(\gamma)$, consisting of {\em holomorphic planes in the split manifold $W\cup X_0$}, i.e.~holomorphic maps $f:\C \to W$ anchored in $X_0$
which are asymptotic to $\gamma$ at $+\infty$.
Let $ \MM_1^W(\gamma)$ be the moduli space of the same objects with an additional marked point at $0\in\C$.
The  evaluation map at this marked point   is a map
$\ev_0\colon \MM^W_1(\gamma)\to W\cup X_0$.
Given a critical point $p$ of the function $H\colon X_0\to\R$ we
will denote   $$\MM^W(\gamma;p)=\ev_0^{-1}(L_p),$$ where
$L_p$ is $\wt Z$-stable manifold of the critical point $p$.

\subsection{\texorpdfstring{Moduli spaces $\MM^W_{(L,C)}(c;b_1,\dots,b_m)$. }{Moduli spaces MW(L,C)(c;c1,...,cm)}}\label{sec:relmWCL}
As in Section \ref{sec:mWL} let $W$ be Weinstein, $p_1,\dots p_k\in\Int\oW $ be the critical points of index $n$ of the Lyapunov function $H$ for the Liouville vector field $Z$ in general position. Let $L_1,\dots, L_k$ denote the Lagrangian stable manifolds of the critical points $p_1,\dots, p_k$, let  $\Lambda_j=L_j\cap Y_0$, $j=1,\dots,k$ denote the corresponding Legendrian spheres, and write $\Lambda:=\bigcup\limits_{j=1}^k\Lambda_j$, $L:=\bigcup\limits_{j=1}^k L_j$. We also consider the Lagrangian unstable manifolds $C_1,\dots, C_k$ of the critical points $p_1,\dots, p_k$. Let $\Gamma_j=C_j\cap Y$, $j=1,\dots,k$ denote the corresponding Legendrian spheres. Write $\Gamma:=\bigcup\limits_{j=1}^k\Gamma_j$ and $C:=\bigcup\limits_{j=1}^{k} C_j$. Note that $C_i\cap L_j$ is empty if $i\ne j$ and consists of one transverse intersection point $p_i$ if $i=j$.
For $c\in\CC(\Gamma)$ and $b_1,\dots, b_m\in\CC(\Lambda)$
let  us denote by $\MM^W_{(L,C)}(c;b_1,\dots,b_m)$   the moduli space
of holomorphic maps
\[
f\colon (D,\p D\setminus\{z,z_+,z_1,\dots,z_m,z_-\})\to (W,C\cup L),
\]
anchored in $X_0$,  which are asymptotic to the Reeb chord $c$ at the puncture $z$ at $+\infty$,
asymptotic to $b_1,\dots,b_m$  at the punctures $z_1,\dots,z_m$ at $-\infty$, and which maps    $z_\pm$   to intersection points   in $L\cap C$. Here we have
\[
\dim(\MM^W_{(L,C)}(c;b_1,\dots,b_m))=|c|-\sum_{j=1}^{m}|b_j|-(n-2),
\]
see Figure \ref{fig:cocorecurve}.
\begin{figure}
\labellist
\small\hair 2pt
\pinlabel $c$   at 86 362
\pinlabel $C$   at 125 295
\pinlabel $C$ at 157 295
\pinlabel $L$ at 49 263
\pinlabel $b_1$ at 68 -15
\pinlabel $b_2$ at 146 -15
\pinlabel $b_3$ at 216 -15
\endlabellist
\centering
\includegraphics[width=.3\linewidth]{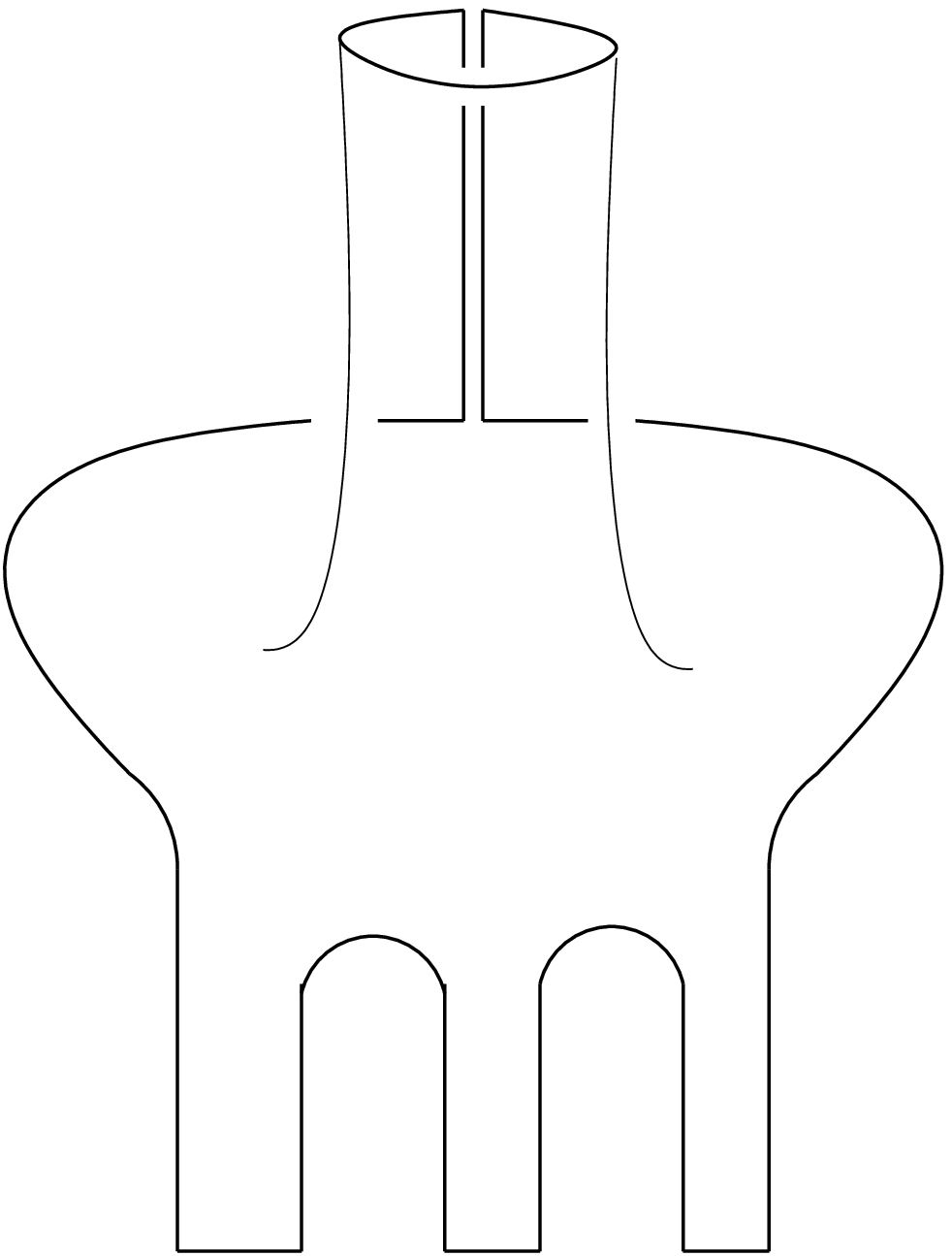}
\caption{A disk in $\MM^{W}_{(L,C)}(c;b_1,b_2,b_3)$.}
\label{fig:cocorecurve}
\end{figure}
   
\begin{rmk}\label{rmk:relativegrading}
Recall that the grading of a Reeb chord connecting a connected Lagrangian submanifold to itself was defined using a phase function, defined only up to additive constant. In order to extend the definition to Lagrangian submanifolds with many components we need to relate the phase functions of distinct components. To this end, we order the components, pick paths connecting each component to its successor in this order, and choose phase functions on the components which agree at the endpoints of the paths. Note that for any holomorphic disk with boundary on a Lagrangian submanifold, punctures mapping to Reeb chords/intersection points between distinct components appear in pairs and hence dimension formulas are insensitive to the choice of paths.
\end{rmk}

\section{Invariants of symplectic manifolds}\label{sec:invariants}

Let $(\oX,\lambda)$ be a Liouville domain with  contact boundary $Y$ and $ X$   its completion. We will consider below three homology theories for the Liouville manifold $X$: linearized contact homology $C\bbH$, reduced symplectic homology $S\bbH_+$, and full symplectic homology $S\bbH$. We use the bold face font letter $\bbH$ for homology, for example $C\bbH(X)$ denotes the linearized contact homology of $X$. For each of the theories, we will also consider the underlying chain complex which we will denote in the same way but using instead the usual $H$, for example $CH(X)$ denotes the chain complex for linearized contact homology. The homologies $C\bbH(X)$, $S\bbH_+(X)$, and $S\bbH(X)$ are invariant under symplectomorphisms of $X$. In the setting introduced below, invariance under symplectomorphism can be established using Lemma \ref{l:exactsympl} and the cobordism maps introduced in Section \ref{sec:cobord} in combination with chain homotopies induced by deformations (which will not be discussed in this paper).

We fix a field $\K$ of characteristic $0$. All complexes and homology theories below are considered over $\K$. Given a set $A$ we denote the vector space over $\K$ generated by the elements of $A$ by $\K\langle A\rangle$ .

Choose adjusted almost complex structures on $X$ and the symplectization $Y\times\R$.
As we already stated above, in this exposition we ignore all transversality problems and assume that regularity is satisfied for all involved moduli spaces. As is also well known, algebraic constructions over $\K$ require a choice of {\em coherent orientation}. For holomorphic curves in $(X,L)$ (i.e.~with Lagrangian boundary condition on the Lagrangian submanifold $L$ in the symplectic manifold $X$) the existence of a system of coherent orientations requires the existence of a relative spin structure on $(X,L)$ and the system of coherent orientations depends on the choice of relative spin structure. In what follows the Lagrangian submanifolds we consider come with distinguished spin structures, see Remarks \ref{rmk:spin} and \ref{rmk:spin2} for details. These spin structures induce a system of coherent orientations which then allows one to assign a sign to each $0$-dimensional component of the moduli spaces
\begin{align*}
&\MM^X(\gamma),\;  \MM^X(\gamma;p),\;
\MM^W(\gamma;\beta), \; \MM^W_L(\gamma;b_1,\dots, b_m), \;
\widecheck\MM{}^W_L(\gamma;b_1,\dots, b_m),
\cr
& \widehat\MM^W_L(\gamma;b_1,\dots, b_m),\; \MM^W_{L_j}(\gamma),\;
\text{and }\MM^W(\gamma;p),
\end{align*}
as well as to each $1$-dimensional component of the moduli spaces
$$\MM^Y_\Lambda(\gamma;b_1,\dots, b_m),\; \MM^Y(\gamma;\beta), \;
\widecheck\MM{}^Y_\Lambda(\gamma;b_1,\dots, b_m), \;\hbox{and}\;
\widehat\MM^Y_\Lambda(\gamma;b_1,\dots, b_m),$$
which were described in Section \ref{Sec:mdli}.

\begin{rmk}\label{rmk:finiteness}
A standard argument from SFT, see \cite{SFT}, shows that  when $\gamma$ is fixed the union of all the above $0$-dimensional   moduli spaces in $X$ and   quotient by the $\R$-action 
$1$-dimensional moduli spaces  in the symplectization of $Y$ {\it with any possible asymptotic at $-\infty$}, are compact. Indeed,  Stokes' theorem provides an upper bound
$\int\limits_\gamma\lambda$  of   the non-negative energy $\int\limits_Cd\lambda$ for any such holomorphic curve $C$, end hence compactness is guaranteed by a version of   Gromov compactness theorem proven in \cite{BEHWZ}.
As a corollary of this observation we note that {\em all the sums which we use below to define the differential in various complexes and chain maps between them are finite.}
\end{rmk}

\subsection{Linearized contact homology}\label{sec:lin-contact}
The {\em linearized contact homology complex} $CH(X)$ \footnote{Traditionally it is called the linearized contact homology of $Y$, but we want to stress its dependence on the symplectic filling $X$.}  is the vector space  $\K\langle\PP_\good(Y)\rangle$ with the differential defined by
\begin{equation}\label{eq:lin-cont-dif}
d_{CH}\gamma=\sum\limits_{|\beta|=|\gamma|-1}\frac{n_{\gamma\beta}}{\kappa(\beta)}\beta,
\end{equation}
where the coefficient $n_{\gamma\beta}$ counts the  algebraic number of $1$-dimensional components of the moduli space $\MM^Y(\gamma;\beta)$, see Section \ref{sec:holcyl+MBcyl}, and $\kappa(\beta)$ is the multiplicity of the orbit $\beta$.

\begin{rmk}[{\bf on the coefficient $n_{\gamma\beta}$}]
The count of components in the moduli space $\MM^{Y}(\gamma;\beta)$ is a count of curves anchored in $X$ and hence a count of broken curves. Explicitly, the contribution to $n_{\gamma\beta}$ from a $1$-dimensional moduli space of spheres in $Y\times\R$ with positive puncture at $\gamma$, $k_j$ negative punctures at $\beta_j$, $j=1,\dots,m$, and a distinguished negative puncture at $\beta$, with anchoring curves in $0$-dimensional moduli spaces of planes with positive punctures at $\beta_j$, $j=1,\dots,m$ is the following:
\begin{equation}\label{eq:anchorcount}
n^{+}\frac{1}{k_1!\,\cdots\, k_m!}\left(\frac{n_1^{-}}{\kappa(\beta_1)}\right)^{k_1}\cdots
\,\,\left(\frac{n_m^{-}}{\kappa(\beta_m)}\right)^{k_m}
\end{equation}
where $n^{+}$ is the algebraic number of components in the moduli space of spheres in $Y\times\R$,   $n^{-}_j$ is the algebraic number of elements in the moduli space of planes in $X$ with positive puncture at $\beta_j$, $j=1,\dots,m$, and where, as usual, $\kappa(\gamma)$ denotes the multiplicity of the Reeb orbit $\gamma$.
\end{rmk}

\begin{prp}\label{prop:cont-homology}
We have  $d_{CH}^2=0.$ The homology
$$
C\bbH(X)=H_*(CH(X), d_{CH})
$$
is independent of all choices and is an invariant of $X$ up to symplectomorphism.
\end{prp}
The homology $H_*(CH(X), d_{CH})$ is denoted by $C\bbH(X)$ and called  the {\em linearized contact homology} of $X$ (or of $Y$).

\begin{rmk}[{\bf on multiplicities}]
One could alternatively define the differential  on  $CH$ by the formula
\begin{equation}\label{eq:lin-cont-dif2}
{\underline{d}}_{CH}\gamma=\sum\limits_{|\beta|=|\gamma|-1}\frac{n_{\gamma\beta}}{\kappa(\gamma)}\beta,
\end{equation}
instead of the differential $d_{CH}$ given by \eqref{eq:lin-cont-dif}. Clearly the complexes $(CH(X), d_{CH})$ and $(CH(X), \underline{d}_{CH})$ are isomorphic via the change of variables $\gamma\mapsto \kappa(\gamma) \gamma$.
\end{rmk}

\begin{rmk}[{\bf on linearizations}]
It is well known that for a general (not necessarily symplectically fillable) contact manifold, contact homology is defined only via a differential graded algebra, rather than  a complex generated by orbits. The {\em linearization scheme} used in the above definition is a special case of a more formal linearization procedure associated with an {\em augmentation}, i.e.~a graded homomorphism of the differential algebra to the ground ring endowed with the trivial differential.  The collection of linearized homologies obtained using all augmentations is a contact invariant.

This procedure was first used by Y.~Chekanov in \cite{Chek} in the context of Legendrian homology algebras of knots in $\R^{3}$ (see Section \ref{sec:Leg-algebra} below). The results of this paper admit a straightforward generalization to this more formal setup.
\end{rmk}

\subsection{Reduced symplectic homology}
We define the {\em reduced symplectic homology} complex as
$$
SH^+(X)=\widecheck{CH}(X)\oplus \widehat{CH}(X),
$$
where
\begin{align*}
\widecheck{CH}(X)&=\K\langle\PP(Y)\rangle,\\
\widehat{CH}(X) &=  \K\langle\PP(Y)\rangle [1],
\end{align*}
i.e.~as a vector space  $\widehat{CH}(X)$  coincides with the complex $\widecheck{CH}(X)$ with grading shifted up by $1$. The differential $d_{SH^+}$ on $SH^{+}(X)$ is given by the block matrix
\begin{equation}\label{eq:d-SHplus}
d_{SH^+}=\left(\begin{matrix}
d_{\widecheck{CH}}&d_{\rm M}\cr
\delta_{SH^+} & d_{\widehat{CH}}\cr
\end{matrix}
\right).
\end{equation}
Here $\delta_{SH^+}:\widecheck{CH}\to\widehat{CH}$ is defined by
\begin{equation}\label{eq:delta}
\delta_{SH^+}\widecheck \gamma=\sum\limits_{|\beta|=|\gamma|-2} m_{\gamma\beta}\widehat\beta,
\end{equation}
where the coefficient $m_{\gamma\beta}$ is the algebraic number of $1$-dimensional components
of the moduli space $\widecheck\MM{}^Y(\gamma;\beta)$, see Section \ref{sec:holcyl+MBcyl}.

The differential  $d_{\widecheck{CH}}\widecheck\gamma$ is  given by  the formula  \eqref{eq:lin-cont-dif}, where $\gamma$ and $\beta$ are replaced by $\widecheck\gamma$ and $\widecheck\beta$, respectively, and where the summation is over all orbits.  The differential  $d_{\widehat{CH}}\widehat\gamma$ is  given by  the formula  \eqref{eq:lin-cont-dif2}, where $\gamma$ and $\beta$ are replaced by $\widehat\gamma$ and $\widehat\beta$, respectively, and where the summation again is over all orbits. The differential $d_{\rm M}\colon\widehat{CH}(X)\to\widecheck{CH}(X)$ vanishes on good orbits and $d_{\rm M}\widehat\gamma= \pm 2 \widecheck \gamma$ if  $\gamma$ is bad.
\begin{rmk}[on good and bad orbits in symplectic homology]\label{rmk:good-bad}
Let us stress the point that the complex $SH^+(X)$ is generated by {\em all} orbits from $\PP(Y)$ and not only by the good orbits as in the case of the contact homology  complex $CH(X)$.
We note that   $\K\langle\PP_{\bad}(Y)\rangle\subset \Ker(d_{\widecheck{CH}})$  and $\Im( d_{\widehat{CH}})\subset \K\langle\PP_{\good}(Y)\rangle[1]$.   
\end{rmk}

\begin{prp}\label{prop:symp-homology-reduced}
We have  $d_{SH^+}^2=0$. The homology
$$
S\bbH^+(X)=H_*(SH^+(X), d_{SH^+})
$$
is independent of all choices and is an invariant of $X$ up to symplectomorphism.
\end{prp}
Note that the homomorphisms $d_{\widecheck{CH}}$, $d_{\widehat{CH}}$ and $\delta_{SH^+}$ decrease the filtration which assigns filtration degree $|\gamma|$ to the elements in $\widecheck{CH}(X)$ and $\widehat{CH}(X)$ which correspond to $\gamma$, while $d_{\rm M}$ preserves it. Hence, the term $E^1$ of the corresponding spectral sequence, together with
Proposition \ref{prop:symp-homology-reduced} yield the following
exact triangle relating $C\bbH(X)$ and $S\bbH^+(X)$:
\begin{equation} \label{eq:CH-SH}
\xymatrix{C\bbH( X)\ar[rr]^{\wt \delta_{SH^+}}& &C\bbH(X) . \ar[ld]\\
&S\bbH^+(X)\ar[lu]&}
\end{equation}
Note that the homomorphism $\wt\delta_{SH^+}$ differs from the one induced  on homology by $\delta_{SH^+}$ because of the contribution from bad orbits.

\begin{rmk} \label{rmk:reduced-SH}
F.~Bourgeois and A.~Oancea proved in \cite{BO} that $S\bbH^+(X)$ is isomorphic to the reduced symplectic homology of A.~Floer and H.~Hofer \cite{FH}, see also \cite{V}, and established the exact triangle \eqref{eq:CH-SH}.
\end{rmk}

\subsection{Full symplectic homology}
Let $H\colon X\to\R$ be an exhausting (i.e.~proper and bounded below)  Morse function which at infinity has no critical points and depends only on the parameter $s$ of the symplectization. In the case when $X$ is Weinstein we always  choose as $H$ the corresponding  Lyapunov function. Pick a gradient like vector field $\wt Z$ on $X$ which coincides with $\frac{\p}{\p s}$ at infinity. In the Weinstein case we  assume that $Z$ is the Liouville vector field.

The {\em full symplectic homology} complex is defined as
$$
SH(X)=SH^+(X)\oplus\Morse(-H)[-n],
$$
where $\Morse(-H)[-n]$ is the Morse homology complex of the function $-H$. The grading is given by Morse indices of the function $-H$ shifted down by $-n$.  Equivalently, we can replace $\Morse(-H)[-n]$
by the Morse {\em cohomology} complex of the function $H$ with the grading by $n-\ind_{\Morse}$.
Note that  the top index of  this complex is   $\leq n$,  while the minimal index in the Weinstein case is equal to $0$.

The differential $d_{SH}$ is given by the block matrix
\begin{equation}\label{eq:d-SH}
d_{SH}=\left(\begin{matrix}
d_{SH^+}&0\cr
\delta_{SH} & d_{\Morse}\cr
\end{matrix}
\right),
\end{equation}
where the homomorphism $\delta_{SH}:SH^+=\widecheck{CH}(X)\oplus \widehat{CH}(X)\to \Morse(-H)[-n]$ is equal to $0$ on $\widehat{CH}(X)$ and equal to \begin{equation}\label{eq:theta}
\delta_{SH}\widecheck \gamma=\sum\limits_{|\gamma|-|p|=1}l_{\gamma p}\, p,
\end{equation}
where $l_{\gamma p}$ is  the algebraic  number of    $0$-dimensional  components of the moduli space
$\MM^X(\gamma;p)$, see Section \ref{sec:holplanes}, on $\widecheck{CH}(X)$. We note that $\delta_{SH}$ vanishes on $\K\langle\PP_\bad(Y)\rangle\subset\widecheck{CH}(X)$.
\begin{prp}\label{prop:sympl-homology}
We have  $d_{SH}^2=0.$ The homology
$$
S\bbH(X)=H_*(SH(X), d_{SH})
$$
is
independent of all choices, is a symplectic  invariant of $X$ and coincides with the Floer-Hofer symplectic homology.
\end{prp}
Proposition \ref{prop:sympl-homology} implies the existence of the following exact homology triangle
\begin{equation} \label{eq:SH-Morse}
\xymatrix{ H^{n-*}(X) \ar[rr]& &S\bbH_*(X) . \ar[ld]\\
&S\bbH_*^+(X)\ar[lu]_{\delta_{SH}}&}
\end{equation}
This triangle is well known for the traditional definition of symplectic homology.
\begin{rmk}
Identifying $S\bbH(X)$ with the Floer-Hofer symplectic homology as in Proposition~\ref{prop:sympl-homology} and $S\bbH^+(X)$ with the reduced Floer-Hofer symplectic homology as in Remark~\ref{rmk:reduced-SH}, the above exact triangle coincides with the tautological exact sequence
obtained via an action filtration argument~\cite{V}.
\end{rmk}

\begin{rmk}[on grading]
In the literature there is no uniform grading convention for symplectic homology. It depends on:
\begin{itemize}
\item different conventions in the definition of Maslov index;
\item different choice of the sign of the action functional ($\int\lambda-Hdt$ versus  $\int Hdt-\lambda$);
\item  symplectic homology versus cohomology.
\end{itemize}
Our symplectic homology $S\bbH(X)$ coincides with the symplectic homology defined by Abbondandolo-Schwarz, see \cite{AS}. In particular, $S\bbH_*(T^*M)=H_*(\Lambda M)$ without any grading shift, where $\Lambda M$ denotes the free loop spaces of a closed spin manifold $M$. Also $S\bbH(X)$ coincides with Seidel's {\em symplectic cohomology} of $X$, see \cite{shhh}, up to sign change of the grading.
\end{rmk}

\subsection{Cobordisms}\label{sec:cobord}
Here we discuss functorial properties of the homology theories $C\bbH$, $S\bbH^+$, and $S\bbH$.
Suppose we are in the framework of Section \ref{sec:mW}.
Let $(\oX,\om,Z)$ be a Liouville domain and $\oX_0\subset\Int\oX$ be a subdomain such that the Liouville vector field $Z$ is outward transverse to $Y_0=\p X_0$. Then $\oW=\oX\setminus \Int \oX_0$ is a Liouville cobordism with $\p_-\oW=Y_0$ and $\p_+\oW=Y=\p\oX$. Let $W$ and $X_0$  be  the completions of $\oW$ and $\oX_0$.
We fix   adjusted   almost complex  structures on $W$ and $X_0$ which agree on the common part.

\subsubsection*{Contact homology}
Define a  homomorphism  $F_{CH}^W\colon CH(X)\to CH(X_0)$ by the formula
$$
F_{CH}^W(\gamma)=\sum\limits_{|\beta|=|\gamma|}\frac{n_{\gamma\beta}}{\kappa(\beta)}\beta,
$$
where $n_{\gamma\beta}$ is the algebraic number of $0$-dimensional components
of the moduli space $\MM^W(\gamma;\beta)$, see Section \ref{sec:mW}, and where the sum ranges over good orbits $\beta$ of grading as indicated. We assume here that the differential on the complex $CH(X)$ is defined by the formula
\eqref{eq:lin-cont-dif}.

\subsubsection*{Reduced symplectic homology}
The  homomorphism $F^W_{SH^+}\colon SH^+(X)\to SH^+(X_0)$ is defined
by the matrix
\begin{equation}\label{eq:map-SHplus}
F_{SH^+}^W=\left(\begin{matrix}
F^W_{\widecheck{CH}}& 0 \cr
\rule{0pt}{3ex} \Psi^W_{SH^+} & F^W_{\widehat{CH}}\cr
\end{matrix}
\right),
\end{equation}
where the entries are the following homomorphisms.
The map $F^W_{\widecheck{CH}}\colon\widecheck{CH}(X)\to\widecheck{CH}(X_0)$ is defined by
\begin{equation}
F_{\widecheck{CH}}^W(\widecheck\gamma)=
\sum\limits_{|\beta|=|\gamma|}\frac{n_{\gamma\beta}}{\kappa(\beta)}\widecheck\beta,
\end{equation}
the map $F^W_{\widecheck{CH}}\colon\widecheck{CH}(X)\to\widecheck{CH}(X_0)$ is defined by
\begin{equation}
F_{\widehat{CH}}^W(\widehat\gamma)=
\sum\limits_{|\beta|=|\gamma|}\frac{n_{\gamma\beta}}{\kappa(\gamma)}\widehat\beta,
\end{equation}
and the map $\Psi^W_{SH^+}\colon\widecheck{CH}(X)\to\widehat{CH}(X_0)$ is defined by
\begin{equation}\label{eq:deltaW}
\Psi^W_{SH^+} \widecheck \gamma=\sum\limits_{|\beta|=|\gamma|-1}m_{\gamma\beta}\widehat\beta,
\end{equation}
where all three sums range over all orbits $\beta$ of grading as indicated, and where $n_{\gamma\beta}$ and $m_{\gamma\beta}$ are the algebraic number of elements of the $0$-dimensional moduli spaces $\MM^W(\gamma;\beta)$ and $\widecheck\MM{}^W(\gamma;\beta)$, respectively, see Section \ref{sec:mW}. We note that $\K\la\PP_{\bad}(Y)\ra\subset\ker(F^W_{\widecheck{CH}})$ and that $\img(F^{W}_{\widehat{CH}})\subset\K\langle\PP_{\good}(Y_0)\rangle[1]\subset\widehat{CH}(X_0)$.

\subsubsection*{Full symplectic homology}
The chain homomorphism $F^W_{SH}\colon SH (X)\to SH (X_0)$ is defined
by the matrix
\begin{equation}\label{eq:map-SH}
F_{SH }^W=\left(\begin{matrix}
F^W_{SH^+}&0\\
\rule{0pt}{3ex}\Psi^W_{SH} & F^W_{\Morse}\\
\end{matrix}
\right).
\end{equation}
Here the map $\Psi^W_{SH}\colon SH^+(X)\to\Morse(-H;X_0)$
vanishes on $\widehat{CH}(X)$ and for a decorated orbit $\widecheck\gamma\in \widecheck{CH}(X)$ it is  defined  by
\begin{equation}\label{eq:thetaW}
\Psi^W_{SH} \widecheck \gamma=\sum\limits_{|\gamma|=|p|}l_{\gamma p} \, p,
\end{equation}
where the coefficient $l_{\gamma p}$ is the algebraic number of $0$-dimensional components of the moduli space
$\MM^W(\gamma;p)$, see Section \ref{sec:evalinW}. Note that $\theta^{W}(\widecheck\gamma)=0$ if $\gamma$ is bad.
The map $$F^W_{\Morse}\colon\Morse(-H;X)\to\Morse(-H,X_0)$$ is the projection.
\begin{prp}
The homomorphisms $F^W_{CH}$, $F^W_{SH^+}$, and $F^W_{SH}$ are chain maps which are independent, up to chain homotopy,  of the choice of gradient like vector field and almost complex structure.
\end{prp}
\section{Legendrian homology algebra, three associated complexes, and linearization}\label{sec:algebra}
\subsection{\texorpdfstring{The Legendrian homology algebra $L\bbH A(\Lambda)$}{The Legendrian homology algebra LHA(Lambda)}}\label{sec:Leg-algebra}
Let $\Lambda_1,\dots,\Lambda_k$ be disjoint embedded Legendrian {\it spheres} in the contact manifold $Y$.
We assume that $\Lambda=\bigcup_{i=1}^{k}\Lambda_i$ is in general position with respect to the Reeb flow on $Y$. Recall that we always assume that the Maslov class of each of these submanifolds is equal to $0$ (for spheres this is an extra assumption only if $\dim(\Lambda_j)=1$), and that the chosen trivialization of the canonical bundle of $\R\times Y$ belongs, along $L=\R\times\Lambda$, to the 1-dimensional real canonical subbundle of $L$. 

For $i\ne j$, we write $\CC_{ij}$ for the set of Reeb chords connecting $\Lambda_i$ to $\Lambda_j$. Note that $\CC_{ij}\neq\CC_{ji}$. We write $\CC_{ii}=\CC_i$ for all Reeb chords connecting $\Lambda_i$ to itself. We will also consider the empty Reeb chord $e_i$ which connects $\Lambda_i$ to itself.  The set of all non-empty chords for $\Lambda$ (there could be countably many of them) is denoted by $\CC$.

Non-empty chords from $\CC$ are graded by the relative Conley-Zehnder index, see Remark \ref{rmk:relativegrading}, and the grading of $e_j$ is $0$ for each $j$. Let $R$ denote the $\K$-algebra with underlying vector space $\K\la e_1,\dots, e_k\ra$ and multiplication defined by
\[
e_i\cdot e_j=\delta_{ij}e_i,\quad 1\le i,j\le k,
\]
where $\delta_{ij}$ is the Kronecker delta. We endow $\K\la\CC\ra$ with the structure of a left-right $R$-module by letting $R$ act as follows
\[
e_i\cdot c=
\begin{cases}
c &\text{if }c\in\CC_{ji} \text{ for some $j$},\\
0 &\text{otherwise}
\end{cases}
\quad\quad
c\cdot e_j=
\begin{cases}
c &\text{if }c\in\CC_{ji} \text{ for some $i$},\\
0 &\text{otherwise}
\end{cases}.
\]

Define the Legendrian homology algebra as
\begin{align*}
LHA(\Lambda)&=\bigoplus_{k=0}^{\infty}\,\,\K\la\CC\ra^{\otimes_R^{k}}\,\,
=\,\,\,R\,\,\oplus\,\,\K\la\CC\ra\oplus\,\, \K\la\CC\ra\otimes_{R}\K\la\CC\ra\,\,\oplus\dots
\end{align*}
Note that the underlying vector space of $LHA(\Lambda)$ is spanned by {\em linearly composable} monomials, i.e.~monomials $c_1\otimes\dots\otimes c_m$ such that the origin of the chord $c_i$ lies on the same Legendrian sphere as the end of $c_{i+1}$, for $i=1,\dots,m-1$. For simpler notation, we will suppress the tensor symbol from the notation writing $c_1\dots c_m$ instead of $c_1\otimes\dots\otimes c_m$.

\begin{rmk}
The algebra $LHA(\Lambda)$ is the path algebra of a quiver~\cite{ASS} associated to the
Legendrian submanifold $\Lambda$. Vertices of the quiver correspond to connected components
$\Lambda_1, \ldots, \Lambda_k$ and edges of the quiver correspond to Reeb chords.
\end{rmk}

The algebra $LHA(\Lambda)$ carries a differential $d_{LHA}\colon LHA(\Lambda)\to LHA(\Lambda)$ which satisfies the
graded Leibniz rule and acts on the generators $c\in\CC$ according to the formula
$$
d_{LHA}c=\sum\limits_{|c|=\sum|b_j|+1} n_{c;b_1\dots b_m} \, b_1\dots b_m,
$$
where $n_{c;b_1\dots b_m}$ is the algebraic number of $1$-dimensional components of the moduli space $\MM^Y_\Lambda(c;b_1,\dots,b_m)$, see Section \ref{sec:mYLambda}.  Note that if $\MM^Y_\Lambda(c;b_1,\dots,b_m)$ is non-empty and $c\in\CC_{ij}$, then the monomial
$b_1 \dots b_m$ is linearly composable. Let us also point out that the differential $d_{LHA}$ acts trivially on all idempotents $e_j$ because there are no holomorphic curves without positive ends.

\begin{rmk}\label{rmk:spin}
The algebraic count of $1$-dimensional components of moduli spaces stems from coherent orientations of moduli spaces. In order to define such a system of orientation one uses a spin structure on $\Lambda$. Since the components of $\Lambda$ are $(n-1)$-spheres each component has a unique spin structure for $n>2$. For $n=2$ any component has two spin structures: the Lie group spin structure which corresponds to the two component double cover of $S^{1}$ and which generates the spin cobordism group of $1$-manifolds, and the null-cobordant spin structure which correspond to the non-trivial double cover of $S^{1}$. Here we will always use the null-cobordant spin structure on any component of $\Lambda$ when $n=2$.  
\end{rmk}

Given a component $\Lambda_i\subset \Lambda$ we denote by $LHA(\Lambda_i;\Lambda)$ the differential subalgebra of $LHA(\Lambda)$ which consists of words which begin and end on $\Lambda_i$.
\begin{prp}\label{prop:leg-homology}
We have  $d_{LHA}^2=0.$ The homologies
$$
L\bbH A(\Lambda)=H_*(LHA(\Lambda), d_{LHA})\;\;\hbox{and}\;\;
L\bbH A(\Lambda_i;\Lambda)=H_*(LHA(\Lambda_i;\Lambda), d_{LHA})
$$
are
independent of all choices, and are Legendrian isotopy  invariants of $\Lambda$.
\end{prp}
In the following sections, we associate several complexes with the differential graded  algebra $(LHA(\Lambda),d_{LHA})$.

Suppose now that we are in the framework of Sections \ref{sec:mLL} and  \ref{sec:cobord}, i.e.
\begin{itemize}
\item $(\oX,\om,Z)$
is a Liouville domain and $\oX_0\subset\Int\oX$ is a subdomain such that the
Liouville vector field $Z$ is outward transverse to $Y_0=\p X_0$;
\item
$\oW=\oX\setminus \Int \oX_0$ is  a Liouville cobordism with $\p_-\oW=Y_0$ and $\p_+\oW=Y=\p\oX$;
\item $W$, $X_0$, and $X$  are the completions of $\oW$, $\oX_0$, and $\oX$;
\item  $L\subset W$  is an exact
Lagrangian cobordism between  Legendrian submanifolds $\Lambda_-\subset \p_-W=Y_0$ and $\Lambda_+ \subset \p_+W=Y$.
\end{itemize}
Then we can define a homomorphism
$F^W_L\colon LHA(\Lambda_+)\to LHA(\Lambda_-)$ by defining it on the generators by the formula
$$
F^W_L(c)=\sum\limits_{|c|=\sum|b_j|} m_{c;b_1\dots b_m} \, b_1\dots b_m,
$$
where $m_{c;b_1\dots b_m}$ is the algebraic number of  elements in the $0$-dimensional moduli space $\MM^W_L(c;b_1,\dots,b_m)$, $c\in\CC(\Lambda_+)$ and $b_1,\dots,b_m\in\CC(\Lambda_-)$, see Section \ref{sec:mLL}.

\begin{prp}\label{prop:leg-homology-hom}
$\quad$
\begin{enumerate}
\item $F^W_L$ is a homomorphism of differential graded algebras which is independent, up to chain  homotopy, of all auxiliary choices.

\item In the special case when $ L$ consists of components $ L_0,\ L_1,\dots, \ L_k$  such that    $L_j\cap \p_+\oW=\varnothing$ for $j \neq 0$ we have
$F^W_L(LHA(\Lambda_+))\subset LHA(\Lambda_{0-};\Lambda_-)$. In particular,  $F^W_L$ induces a homomorphism
$$\left(F^W_L\right)^*:L\bbH A(\Lambda_+)\to L\bbH A(\Lambda_{0-};\Lambda_-), $$
where $\Lambda_{j-}:=L_j\cap\p_-W$ and $\Lambda_-:= \bigcup_{j=0}^k \Lambda_{j-}$.
\end{enumerate}
\end{prp}

{
\subsection{\texorpdfstring{The algebra $L\bbH A(\Lambda;\bq)$}{The algebra LHA(Lambda;q)}}
We define deformations of the algebra $L\bbH A(\Lambda)$ in the spirit of Symplectic Field Theory, as follows.

Choose a cycle $q$ representing a homology class ${\mathbf q}\in H_{k}(\Lambda)$, where $0\le k\le n-2$. We will assume that no point in $q$ is the endpoint a Reeb chord of $\Lambda$.
\footnote{
Though this could be difficult to arrange for all end-points, it can be done  by a small perturbation for all endpoints of Reeb chords below any given finite action. This problem, as well as the problem of general position of the evaluation maps $\ev_k$, belong to the class of transversality problems mentioned in the beginning of Section \ref{sec:invariants} and will not be discussed further here.}
Consider the algebra
\begin{align*}
LHA(\Lambda;\bq)=&\bigoplus_{k=0}^{\infty}\,\,\K\la\CC\cup\{q\}\ra^{\otimes_R^{k}}\,\,
=\cr
&R\,\,\oplus\,\,\K\la\CC\cup\{q\}\ra\oplus\,\, \K\la\CC\cup\{q\}\ra\otimes_{R}\K\la\CC\cup\{q\}\ra\,\,\oplus\dots\;,
\end{align*}
where $q$ is an element of degree $n-k-2$.
The differential  $d_{LHA;q}$
satisfies the
graded Leibniz rule, acts trivially on $q$, and acts on the generators $c\in\CC$ according to the formula
$$
d_{LHA;q}c=\sum\limits_{\stackrel{|c|=\sum|b_j|+1+k(n-k+2)}{k=k_0 + \ldots + k_m}} n_{c;b_1\dots b_m;k_0,\dots,k_m} \, q^{k_0}b_1q^{k_1}\dots q^{k_{m-1}}b_mq^{k_m},
$$
where $n_{c;b_1\dots b_m;k_0,\dots,k_m}$ is the algebraic number of $1$-dimensional components of the moduli space $\MM^Y_\Lambda(c;b_1,\dots,
b_m;k_0,\dots,k_m)\cap \ev_k^{-1}(\underbrace{q\times\dots\times q}_k)$, see \eqref{eq:deform}.
\begin{prp}\label{prop:leg-homology-def}
We have  $d_{LHA;q}^2=0.$ The homology
$$
L\bbH A(\Lambda;\bq)=H_*(LHA(\Lambda;q), d_{LHA;q})
$$
is
independent of all choices, including the choice of cycle $q$ in the homology class $\bq$,  and it  is a Legendrian isotopy  invariant of $\Lambda$. The  special element $q$ is preserved under the corresponding isomorphisms.
\end{prp}

Given a Legendrian submanifold $\Lambda\subset Y$, consider a small Legendrian unknot $\Lambda_{\rm f}$ linked (with linking number $\pm 1$) with $\Lambda$. Let $q$ be a point generating $H_0(\Lambda)$. Consider the unital subalgebra of the differential algebra    $(LHA(\Lambda;q),d_{\Lambda,q})$ that is generated by words which end with $q$ but do not start with $q$. Take the quotient of this subalgebra by the relation $q^2=0$ and denote the resulting quotient algebra by $B$. The differential $d_{\Lambda, q}$ descends to a differential $d_B$ on $B$.

It turns out that the Legendrian homology algebra of $\Lambda_{\rm f}\cup\Lambda$ is related to $(B,d_B)$. More precisely, we have the following result.
\begin{prp}\label{prop:rel-q}
The algebra $L\bbH A(\Lambda_{\rm f};\Lambda\cup\Lambda_{\rm f})$ is isomorphic to the algebra
$H_*(B,d_B)$, i.e.~to the homology of $(B,d_B)$.
\end{prp}

Let us comment on the above isomorphism.  For an appropriate choice of representative of $\Lambda_{\rm f}$ (as a small sphere concentrated near one point of $\Lambda$), the algebra $LHA(\Lambda_{\rm f};\Lambda_{\rm f}\cup\Lambda)$ is the free differential associative graded algebra generated by words of the form $x_-wx_+$, where $w\in LHA(\Lambda)$, of grading $|x_-wx_+|=|w|+(n-2)$, and one additional element $a$ with $|a|=n-1$ and with $da=x_-x_+$.  Here $x_-\in\Cc(\Lambda,\Lambda_{\rm f})$ and $x_+\in\Cc(\Lambda_{\rm f},\Lambda)$.
The homology isomorphism in Proposition \ref{prop:rel-q} is induced by the homomorphism
$$\Phi\colon (B,d_B)\to (LHA(\Lambda_{\rm f};\Lambda\cup\Lambda_{\rm f}), d_{LHA})$$ of differential graded algebras which is given on generators $wq, w\in LHA(\Lambda)$, by the formula
$\Phi(wq)=x_-wx_+$.
}
\subsection{\texorpdfstring{The cyclic complex $LH^{\rm cyc}(\Lambda)$}{The cyclic complex LHcyc(Lambda)}}
Consider the subalgebra  $LHO(\Lambda) \subset LHA(\Lambda)$ of {\em cyclically composable} monomials, i.e.~linearly composable monomials $c_1\dots c_m$ such that the end point of $c_m$ lies on the same Legendrian sphere as the origin of $c_1$. We denote the restriction of $d_{LHA}$ to $LHO(\Lambda)$ by $d_{LHO}\colon LHO(\Lambda)\to LHO(\Lambda)$.

Since $d_{LHO}$ acts trivially on all units $e_j \in LHO(\Lambda)$, it induces a differential $d_{LHO^{+}}$ on the
reduced complex $LHO^{+}(\Lambda) = LHO(\Lambda)/R$. (Note that $LHO^{+}(\Lambda)$ is the algebra generated by non-trivial cyclically composable monomials of non-empty Reeb chords.)

Let $P : LHO^{+}(\Lambda) \to LHO^{+}(\Lambda)$ be the linear map induced by
graded cyclic permutation:
$$
P(c_1 c_2\ldots c_l) = (-1)^{|c_1|(|c_2|+\ldots +|c_l|)} c_2\ldots c_l c_1
$$
for any monomial $c_1\ldots c_l \in LHO^+(\Lambda)$. Then $\im(1 - P)$ is a subcomplex
of $LHO^{+}(\Lambda)$. Let $LH^{\rm cyc}(\Lambda)$ be the quotient complex
$LHO^{+}(\Lambda)/\im(1-P)$. We denote by $d_{\rm cyc}$ the differential induced by
$d_{LHO^+}$. Note that $LH^{\rm cyc}(\Lambda)$ is not an algebra. It is a $\K$-module generated by equivalence classes of cyclically composable monomials.

If $w=c_1\dots c_m$ is a monomial in $LHO^{+}(\Lambda)$ we will denote its image in
$LH^{\rm cyc}(\Lambda)$ by $(w)$ and define the {\em multiplicity} of $(w)$ as the largest integer $k$ such that $(w)=(v^k)$ for some monomial $v$, where $v^k$ is the monomial which is the $k$-fold product of $v$. We denote the multiplicity of $(w)$, by $\kappa((w))$. We say that a monomial $w \in LHO^{+}(\Lambda)$ is
bad if, after acting on it with a power of $P$, it is the product of an even number of copies of
an odd-graded monomial $w'$. Non-bad monomials are called good. Then for any monomial $w \in LHO^{+}(\Lambda)$, we have $(w)=0$ if and only if $w$ is bad. Moreover, $LH^{\rm cyc}(\Lambda)$ is generated by
the elements $(w)$, where $w$ is a good word in $LHO^{+}(\Lambda)$.

\begin{prp}\label{prop:cyc-homology}
We have  $d_{\rm cyc}^2=0.$ The homology  $$L\bbH^{\rm cyc}(\Lambda)=H_*(LH^{\rm cyc}(\Lambda), d_{\rm cyc})$$ is independent of all choices, and is a Legendrian isotopy  invariant of $\Lambda$.
\end{prp}

\subsection{\texorpdfstring{The complex $LH^{\Ho+}(\Lambda)$}{The complex LHHo+(Lambda)}}
Consider the complex
$$
LH^{\Ho+}(\Lambda)=\widecheck{LHO}^{+}(\Lambda)\oplus\widehat{LHO}^{+}(\Lambda),
$$
where $\widecheck{LHO}^{+}(\Lambda)=LHO^+(\Lambda)$ and $\widehat{LHO}^+(\Lambda)=LHO^{+}(\Lambda)[1]$.
Given a monomial $w=c_1\dots c_l \in LHO^{+}(\Lambda)$ we denote by $\widecheck w=\widecheck{c_1} c_2\dots c_l$ and $\widehat w=\widehat{c_1} c_2\dots c_l$  the corresponding elements of
$\widecheck{LHO}^{+}(\Lambda)$ and $\widehat{LHO}^{+}(\Lambda)$ and use the same notation for the linear maps defined in this way on generators.
We will also sometimes view  elements of  $\widecheck{LHO}^{+}(\Lambda)$ and
$\widehat{LHO}^{+}(\Lambda)$ as monomials with the ``check'' or ``hat'' mark on the variable which is not necessarily the first one. Such a monomial is meant to be identified with the word  obtained by the graded cyclic permutation which put the marked letter in the first position.

Let $S\colon LHO^{+}(\Lambda) \to\widehat{LHO}^{+}(\Lambda)$ denote the linear operator defined by the formula
\begin{equation}\label{eq:defS}
S(c_1\dots c_l)=\widehat{c_1}c_2\dots c_l + (-1)^{|c_1|} c_1\widehat{c_2}\dots c_l
+\dots+ (-1)^{|c_1\dots c_{l-1}|} c_1 c_2 \dots\widehat{c_l}.
\end{equation}
The differential $d_{\Ho+} : LH^{\Ho+}(\Lambda) \to LH^{\Ho+}(\Lambda)$ is given by the matrix
\begin{equation}\label{eq:d-Hoh}
d_{\Ho+}=\left(\begin{matrix}
\widecheck d_{LHO^{+}} &d_{M \, Ho+} \cr
0&  \widehat  d_{LHO^{+}}\cr
\end{matrix}
\right).
\end{equation}
were the maps in the matrix are defined as follows on generators.
\begin{itemize}
\item If $w\in LHO^{+}(\Lambda)$ is any monomial then
\begin{equation}
\widecheck d_{LHO^{+}} (\widecheck w) :=\sum_{j=1}^{r} \widecheck{v_j},
\end{equation}
where each $v_j$ is a monomial and $d_{LHO^{+}}(w)=\sum_{j=1}^{r} v_j$.
\item
If $c$ is a chord and $w' \in LHA(\Lambda)$ is a monomial such that $c w' \in LHO^{+}(\Lambda)$ then
\begin{equation}
\widehat d_{LHO^{+}} (\widehat c w')=S(d_{LHO^{+}} c)\, w'+(-1)^{|c|+1}\widehat c\, (d_{LHO^{+}}w').
\end{equation}
Compare Section 5 in particular Corollary 5.6 in~\cite{EK}. 
\item
If $w=c_1\dots c_l\in LHO^+(\Lambda)$ is any monomial then
\begin{equation}
d_{M \, Ho+}(\widehat w)=d_{M \, Ho+}(\widehat{c_1}\dots c_l) = \widecheck{c_1}\dots c_l - c_1\dots\widecheck{c_l}.
\end{equation}
\end{itemize}
\begin{prp}\label{prop:hoh-homology}
We have  $d_{\Ho+}^2=0.$ The homology  $$L\bbH^{\Ho+}(\Lambda)=H_*(LH^{\Ho+}(\Lambda), d_{\Ho+})$$ is independent of all choices, and is a Legendrian isotopy invariant of $\Lambda$.
\end{prp}

The following result is the Legendrian analogue of the exact homology triangle~\eqref{eq:CH-SH}. It is obtained by considering the second page of the spectral sequence associated to the filtration of
$LH^{\Ho+}(\Lambda)$ by grading of the underlying unmarked words in $LHO^{+}(\Lambda)$.

\begin{prp}
There exists an exact homology triangle
\begin{equation} \label{eq:Cyclic-Hoh}
\xymatrix{ L\bbH^{\rm cyc}(\Lambda) \ar[rr] & &L\bbH^{\rm cyc}(\Lambda)\ar[ld]\\
&L\bbH^{\Ho+}(\Lambda)\ar[lu]&}
\end{equation}
\end{prp}

\subsection{\texorpdfstring{The full complex $LH^{\Ho}(\Lambda)$}{The full complex LHHo(Lambda)}}\label{ssec:wtHoh}
Define  $LH^{\Ho}(\Lambda)= LH^{\rm Ho+}(\Lambda) \oplus C(\Lambda)$, where $C(\Lambda)$ is the vector space
generated by elements $\tau_1,\dots \tau_k$ of grading $0$, in bijective correspondence with the  Legendrian spheres $\Lambda_1,\dots, \Lambda_k$. Note that after identification of $\tau_j$ with $e_j$ for $j=1,\dots,k$, we have $C(\Lambda)=R=\K\la e_1,\dots,e_k\ra$ and we may think of $LH^{\Ho}(\Lambda)$ as
\[
LH^{\Ho}(\Lambda)=\widecheck{LHO}(\Lambda)\oplus \widehat{LHO}^{+}(\Lambda),
\]
where $\widecheck{LHO}(\Lambda)=\widecheck{LHO}^{+}(\Lambda)\oplus R=LHO(\Lambda)$ (as vector spaces).

The differential
$$
d_{\Ho}\colon LH^{\Ho}(\Lambda)=LH^{\rm Ho+}(\Lambda) \oplus C(\Lambda)
\to LH^{\rm Ho+}(\Lambda) \oplus C(\Lambda)= LH^{\Ho}(\Lambda)
$$
is given by the matrix
\begin{equation}\label{eq:d-tHoh}
d_{\Ho}=\left(\begin{matrix}
d_{\Ho+} &0\cr
\delta_{\Ho} & 0\cr
\end{matrix}
\right),
\end{equation}
where
$\delta_{\Ho}(\widecheck c) = \sum_{i=1}^k n_{ci} \tau_i$ for any chord $c \in \CC_i$, $\delta_{\Ho}(\widecheck w) = 0$ for any non-linear monomial $w \in LHO^{+}(\Lambda)$ and $\delta_{\Ho}(\widehat w)=0$ for any $w\in LHO(\Lambda)$. The coefficient $n_{ci}$ is the algebraic number of components of the $1$-dimensional moduli space $\MM^Y_\Lambda(c)$, see Section \ref{sec:mYLambda}, of holomorphic disks with one positive and no negative boundary punctures.

\begin{prp}\label{prop:thoh-homology}
We have  $d_{\Ho}^2=0.$ The homology
$$
L\bbH^{\Ho}(\Lambda)=H_*(LH^{\Ho}(\Lambda), d_{\Ho})
$$
is independent of all choices, and is a Legendrian isotopy  invariant of $\Lambda$.
\end{prp}
\subsection{Linearized Legendrian homology}\label{ssec:LH}
Let $X$ be a Liouville manifold with $\pa \oX = Y$ and let $L\subset X$ be an exact Lagrangian submanifold of $X$ which bounds a Legendrian submanifold $\Lambda = L \cap Y$. Assume that $L$ is equipped with a spin structure which is used to define a system of coherent orientations of moduli spaces. Let us assume that the restriction $H|_L$ is Morse and denote by   $\Morse(-H|_L)$   the Morse complex of the function $-H|_L$.
Consider the complex
\[
LH(\Lambda;L)=\K\langle\CC(\Lambda)\rangle\oplus \Morse(-H|_L)
\]
with differential $d_{LH}\colon LH(\Lambda;L)\to LH(\Lambda;L)$,
\[
d_{LH}=\left(
\begin{matrix}
d_{\CC} & 0\\
\delta_{LH}& d_{\rm Mo}
\end{matrix}
\right).
\]
Here $d_{\rm Mo}\colon\Morse(-H|_L)\to\Morse(-H|_L)$ is the Morse differential,
\[
d_{\CC}(c)=\sum\limits_{|b|=|c|-1}n_{cb}\,b,
\]
where $n_{cb}$ is the algebraic number of components of the 1-dimensional moduli space $\MM^Y_{\Lambda,L}(c;b)$ of holomorphic strips connecting  $c$ and $b$ and anchored in
$(X,L)$, see Section \ref{sec:mYLambda}, and
\[
\delta_{LH}(c)=\sum n_{cp}\,p,
\]
where $n_{cp}$ is the algebraic number of components of the 1-dimensional moduli space $\MM^W_L(c;p)$, see Section \ref{sec:mLL}.
\begin{prp}\label{prop:LH-homology}
We have  $d_{LH}^2=0.$ The homology
$$
L\bbH (\Lambda;L)=H_*(LH  (\Lambda;L), d_{LH})
$$
is independent of all choices, and is invariant under continuous deformations of $L$ through exact Lagrangian submanifolds with Legendrian boundaries.
\end{prp}
We call $L\bbH (\Lambda;L)$ the {\it linearized Legendrian homology} of $\Lambda$. It coincides with wrapped Lagrangian Floer homology of $L$ in the sense of \cite{AbS, FSS}.
\begin{rmk}\label{rmk:spin2}
Below we will consider the linearized Legendrian homology of co-core disks. In this case $L$ is diffeomorphic to $\R^{n}$ and hence has a unique spin structure. 
\end{rmk}
\section{Legendrian surgery exact triangles}\label{sec:triangles}
Let us return to the situation of Section \ref{sec:mW}.
Let $(\oX,\om,Z)$ be a Liouville domain and $\oX_0\subset\Int\oX$ be a subdomain such that the Liouville vector field $Z$ is outward transverse to $Y_0=\p X_0$. Then $\oW=\oX\setminus \Int \oX_0$ is a Liouville cobordism with $\p_-\oW=Y_0$ and $\p_+\oW=Y=\p\oX$. Let $W$ and $X_0$  be   completions of $\oW$ and $\oX_0$.
We fix   adjusted   almost complex  structures on $W$ and $X_0$ which agree on the common part.

Let us assume that  $\oW$ is  a Weinstein cobordism, i.e.~that there exists a Morse   function $H:\oW\to\R$  which is constant on the boundary components and  which is Lyapunov for the Liouville vector field $Z$. We assume that $H$ has critical points $p_1,\dots, p_k\in\Int W$ of index $n$ and no other critical points.
We extend $H$ to $X_0$ as any Morse function (or as a Lyapunov function for the Liouville field $Z$ if $X_0$ is Weinstein as well).

We denote the stable manifolds of the critical points of $H$ by $L_1,\dots, L_k$ and write $\Lambda_j=L_j\cap Y$, $j=1,\dots,k$ for the attaching Legendrian spheres, and let $\Lambda = \bigcup_{j=1}^k \Lambda_j$

The   Weinstein cobordism structure on
$W$ is determined by the Legendrian embeddings $\Lambda_j\colon S^{n-1} \to Y_0$,\footnote{Although the manifold $W$ may depend on the parametrization of the Legendrian spheres, the holomorphic curves invariants considered in this paper are independent thereof.} $j=1,\dots,k$, up to homotopy in the class of Weinstein cobordism structures with equivalent  Lyapunov functions. In particular, the contact manifold $Y$ and the symplectic manifold $X$ are determined by these data.
We will say, that $Y$ is obtained from $Y_0$ by a {\em Legendrian surgery } along $\Lambda_1,\dots, \Lambda_k$, and that $X$ is obtained from $X_0$ by {\em attaching Weinstein handles}
along these spheres. We will refer to the spheres $\Lambda_1,\dots, \Lambda_k \subset Y_0$ as a {\em Legendrian surgery basis}.

Let $LHA(\Lambda)$,  $LH^{\rm cyc}(\Lambda)$, $LH^{\Ho+}(\Lambda)$, and $LH^{\Ho}(\Lambda)$
be be the Legendrian homology algebra  and the three complexes associated to it, as described in Section~\ref{sec:algebra}.

\subsection{Linearized contact homology}
The next theorem  describes a complex which computes the linearized contact homology $CH(X)$.
\begin{thm}\label{thm:d-ch}
Consider the complex
$LCH(X_0,\Lambda)=CH(X_0)\oplus LH^{\rm cyc}(\Lambda)$. Let $d_{LCH}$ be the map given by
$$
d_{LCH} =
\left(\begin{matrix}
d_{CH} &0\cr
\delta_{LCH} &  d_{\rm cyc}\cr
\end{matrix}
\right),
$$
where
\begin{equation}\label{eq:dmixed}
\delta_{LCH}(\gamma)=\sum\limits_{|\gamma|-|w|=1} \frac{n_{\gamma(w)}}{\kappa((w))} \, (w).
\end{equation}
Here $\kappa((w))$ is the multiplicity of the cyclic word $(w)$ and $n_{\gamma(w)}$ is the algebraic number of components of the $1$-dimensional moduli space
$\MM^{Y_0}_\Lambda(\gamma;b_1,\dots, b_k)$, see Section \ref{sec:mLambda}, for any monomial $w=b_1\dots b_k$ which represents $(w)\in LH^{\rm cyc}(\Lambda)$. Then $d_{LCH}^2=0$.
\end{thm}

\begin{figure}
\centering
\includegraphics[width=.6\linewidth]{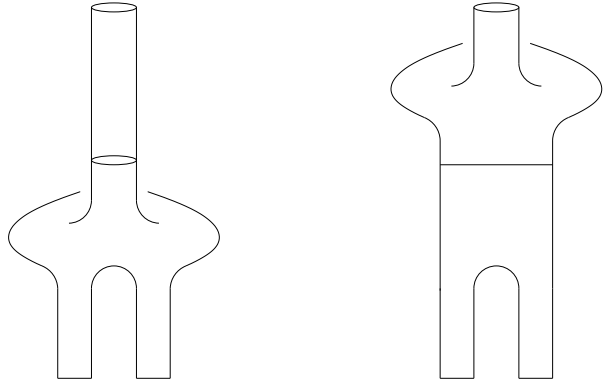}
\caption{Two level holomorphic disks at the boundary of a $1$-dimensional moduli space of holomorphic disks (up to translation) which cancel in the equation $d_{LCH}^2=0$.}
\label{fig:d2CH}
\end{figure}

Define the homomorphism
$$F^W_{LCH}=F^W_{CH}\oplus F^W_{\rm cyc}\colon CH (X)\to
LCH(X_0, \Lambda)=CH(X_0)\oplus LH^{\rm cyc}(\Lambda),
$$
where
$F^W_{CH}\colon CH(X)\to CH(X_0)$ is the homomorphism induced by the cobordism $W$, which was defined in Section \ref {sec:cobord},
and where $F^W_{\rm cyc}\colon CH(X) \to LH^{\rm cyc}(\Lambda)$ is defined by the formula
\begin{equation}\label{eq:WCyclic}
F^W_{\rm cyc}(\gamma)=\sum\limits_{|\gamma|= |w|} \frac{m_{\gamma(w)}}{\kappa((w))}  \,(w).
\end{equation}
Here $\kappa((w))$ is the multiplicity of $(w)$ and $m_{\gamma(w)}$ is the algebraic count of  the number of elements of the $0$-dimensional
moduli space
$\MM^W_L(\gamma;b_1,\dots, b_k)$  for any monomial $w=b_1\dots b_k$ representing the cyclic class $(w)$, see Section \ref{sec:mWL}.

\begin{thm}\label{thm:f-ch} The homomorphism
$F^W_{LCH}\colon CH(X) \to LCH(X_0, \Lambda)$  is  a   chain  quasi-isomorphism.
In particular, there exists an exact triangle
\begin{equation} \label{eq:CH-surgery1}
\xymatrix{C\bbH( X)\ar[rr]& &C\bbH(X_0)\ar[ld]\\
&L\bbH^{\rm cyc}(\Lambda) \ar[lu]&}.
\end{equation}
\end{thm}

\subsection{Reduced symplectic homology}
The next theorem  describes a complex which computes the reduced symplectic  homology $S\bbH^+( X)$. We use notation as in Section \ref{sec:algebra}.  

\begin{thm}\label{thm:d-sh1}
Consider the complex
$SLH^+(X_0, \Lambda)=SH^+(X_0)\oplus LH^{\Ho+}(\Lambda)$. Let $d_{SLH^{+}}$ be the map given by
$$d_{SLH^+} =
\left(\begin{matrix}
d_{SH^+} &0\cr
\delta_{SLH^+} &  d_{\Ho+}\cr
\end{matrix}
\right).
$$
Here
\[
\delta_{SLH^+}(\widehat\gamma)=
\sum\limits_{|\gamma|-|w|=1}
\frac{n_{\gamma(w)}}{\kappa(\gamma)}\, S(w),
\]
where  $\kappa(\gamma)$ is the multiplicity of $\gamma$, $n_{\gamma(w)}$ is the algebraic count of components of
the moduli space $\MM^{Y_0}_\Lambda(\gamma;b_1,\dots, b_m)$ for any $w=b_1\dots b_m$ representing $(w)$, see Section \ref{sec:mLambda}, and
\[
\delta_{SLH^+} (\widecheck\gamma)=\sum\limits_{|\gamma|-|w|=1}
\widecheck n_{\gamma w}\, \widecheck w+\sum\limits_{|\gamma|-|w|=2}\widehat n_{\gamma w} \widehat w,
\]
where $\widecheck n_{\gamma w}$ (respectively $ \widehat n_{\gamma w}$)  is the algebraic count of the number of elements
of the moduli space $\widehat\MM^{Y_0}_\Lambda(\gamma;b_1,\dots, b_m)$
(respectively $\widecheck\MM{}^{Y_0}_\Lambda(\gamma;b_1,\dots, b_m)$), where $w=b_1\dots b_m$, see Section \ref{sec:mLambda}.
Then $d_{SLH^+}^2=0$.
\end{thm}
Consider the map  
$$
F^W_{SLH^+}=F^W_{SH^+}\oplus F^W_{\Ho+}:
SH^+(X)\to  SLH^+(X_0, \Lambda)=SH^+(X_0)\oplus LH^{\Ho+}(\Lambda),
$$
where
$F^W_{SH^+}\colon SH^+(X)\to SH^+(X_0)$ is the map induced by the cobordism $W$ that was defined above in Section \ref{sec:cobord} and
$$
F^W_{\Ho+}\colon SH^+(X)=\widecheck{CH}(X)\oplus\widehat{CH}(X)\to LH^{\Ho+}(\Lambda)
$$
is defined by the formula
\begin{align*}
F^W_{\Ho+}(\widehat\gamma) &= \sum\limits_{|\gamma|-|(w)|=0} \frac{m_{\gamma (w)}}{\kappa(\gamma)}\, S(w),\\
F^W_{\Ho+}(\widecheck\gamma)&=  \sum\limits_{|\gamma|-|w|=0} \widecheck m_{\gamma w}\widecheck w+
\sum\limits_{|\gamma|-|w|=1}\widehat m_{\gamma w}\widehat w
\end{align*}
where  $m_{\gamma(w)}$ is the algebraic number of elements of
 $\MM^{W}_L(\gamma;b_1,\dots, b_m)$ for any $w=b_1\dots b_m$ representing $(w)$, and
where $\widecheck m_{\gamma w}$ (resp.~$\widehat m_{\gamma w}$)  is the algebraic number of elements
of the  space $\widehat\MM^W_L(\gamma;b_1,\dots, b_m)$
(resp.~$\widecheck\MM{}^W_L(\gamma;b_1,\dots, b_m)$), see Section \ref{sec:mWL}.

\begin{thm}\label{thm:f-sh1}
The homomorphism  $F^W_{SLH^+}\colon SH^+(X)\to SLH^+(X_0, \Lambda)$ is a chain map and a
quasi-isomorphism of the corresponding complexes.
In particular, there exists an exact triangle
\begin{equation} \label{eq:CH-surgery2}
\xymatrix{S\bbH^+(X)\ar[rr]& &S\bbH^+(X_0)\ar[ld]\\
&L\bbH^{\Ho+}(\Lambda) \ar[lu]&}.
\end{equation}
\end{thm}
\subsection{Symplectic homology}\label{sec:SHcomput}
The complex described below computes the (full) symplectic homology
$S\bbH(X)$.
\begin{thm}\label{thm:d-sh2}
Consider the complex
$SLH(X_0, \Lambda)=SH(X_0)\oplus LH^{\Ho}(\Lambda)$. Let $d_{SLH}$ be the map given by
$$ d_{SLH}= \left(\begin{matrix}
d_{S H } &0 \cr
\Delta_{SLH} &  d_{\Ho}\cr
\end{matrix}
\right),
$$
where
$$\Delta_{SLH}\colon SH (X_0)=SH^+(X_0)\oplus\Morse(-H|_{X_0})\to LH^{\Ho}(\Lambda)=LH^{\Ho+}(\Lambda)\oplus C(\Lambda)$$
is given by the matrix
$$ \Delta_{SLH}=\left(\begin{matrix}
\theta&0 \cr
\wt\Delta_{SLH}&  d_{\Morse\tau}\cr
\end{matrix}\right).
$$
Here $\wt\Delta_{SLH}\colon SH^+(X_0)=\widecheck{CH}(X_0)\oplus \widehat{CH}(X_0)\to C(\Lambda)$ is equal to $0$ on $\widehat{CH}(X_0)$, and on
$\widecheck{CH}(X_0)$ it is defined by the formula $\wt\Delta_{SLH}(\widecheck\gamma)=\sum\limits_1^kn_{\gamma j}\,\tau_j$, where
the coefficient $n_{\gamma j}$ is the algebraic number of components of the $1$-dimensional moduli space
$\MM^{Y_0}_{\Lambda_j}(\gamma)$, see Section \ref{sec:mLambda}, and the homomorphism $d_{\Morse\tau}\colon\Morse(-H|_{X_0})\to C(\Lambda)$ is the boundary operator in $\Morse(-H)$ followed by projection to $C(\Lambda)$. Then $d_{SLH}^2=0$.
\end{thm}

\begin{figure}
\labellist
\small\hair 2pt
\pinlabel $\Lambda\times\R$   at 270 38
\pinlabel $\Lambda\times\R$   at 655 130
\endlabellist
\centering
\includegraphics[width=.6\linewidth]{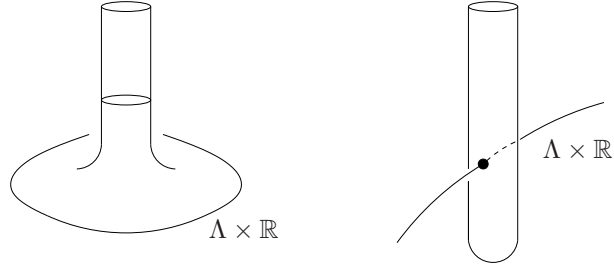}
\caption{On the left, a two level disk with boundary on $\Lambda\times\R$ at one boundary point of a $1$-dimensional moduli space of holomorphic disks (up to translation). As the moduli space is traversed, the boundary of the disk shrinks. On the right, a one level sphere which intersects $\Lambda\times\R$ in a point, which is the configuration at the other boundary point of the moduli space. These two configurations cancel in the equation $d_{SLH}^2=0$.}
\label{fig:d2SH}
\end{figure}

Note that  the complex $SLH(X_0, \Lambda)$  can be rewritten as
\begin{eqnarray*}
SLH(X_0, \Lambda)&=&(SH^+(X_0)\oplus LH^{\Ho+}(\Lambda))\oplus(\Morse(-H|_{X_0})\oplus C(\Lambda)) \\
&=&
SLH^+(X_0, \Lambda)\oplus \Morse(-H)
\end{eqnarray*}
and the differential takes the form
$$
d_{SLH}=
\left(\begin{matrix}
d_{SLH^+}&0\cr
\delta_{SLH} &d_{\Morse}\cr
\end{matrix}
\right).
$$
Using this observation we define
a homomorphism
\begin{align*}
F^W_{SLH} :SH(X)=SH^+(  X)&\oplus\Morse(- H)
\to SLH(X_0, \Lambda)\cr
& =SLH^+(X_0, \Lambda)\oplus\Morse(-H)
\end{align*}
by the block matrix
$$F^W_{SLH}= \left(\begin{matrix}
F^W_{SLH^+} &0 \cr
\Psi^W_{SLH} &\Id\cr
\end{matrix}
\right),
$$
where the homomorphism  $\Psi^W_{SLH}\colon SH^+(X)=\widecheck{CH}(X)\oplus\widehat{CH}(X)\to\Morse(-H)$ is equal to $0$ on $\widehat{CH}(X)$, and is defined by the formula
$\Psi^W_{SLH}(\widecheck\gamma)=\sum\limits_{|\gamma|=1} n_{\gamma p_i}\tau_i$, where the coefficient
$n_{\gamma p_i}$  is the algebraic number of elements of the $0$-dimensional moduli space
$\MM^W_{L_i}(\gamma)$, see Section \ref{sec:mWL}.

\begin{thm}\label{thm:SH} The map $F^W_{SLH}:SH(X)\to SLH(X_0, \Lambda)$ is a chain map and a quasi-isomorphism.
In particular, there exists an exact triangle
\begin{equation} \label{eq:CH-surgery3}
\xymatrix{S\bbH( X)\ar[rr]& &S\bbH(X_0)\ar[ld]\\
&L\bbH^{\Ho}(\Lambda) \ar[lu]&}.
\end{equation}
\end{thm}

\begin{cor}\label{cor:SH-subcrit}
Suppose that $X_0$ is a subcritical Weinstein manifold.
Then $S\bbH(X)=L\bbH^{\Ho}(\Lambda)$.
\end{cor}
Indeed, according to a theorem of K.~Cieliebak, see \cite{Ciel2}, we have $S\bbH(X_0)=0$, and hence the exact triangle
\eqref{eq:CH-surgery3} implies the claim. Note that any Weinstein manifold  $X$ with cylindrical end can be obtained by attaching Lagrangian handles to a subcritical $X_0$.

\subsection{Linearized Legendrian homology of co-core spheres}\label{sec:LH-surgery}
Let $(W,L,C)$ be as in Section \ref{sec:relmWCL}. We use notation as there: $\Lambda=L\cap Y_0$,   $\Gamma=C\cap Y$. Furthermore, $L$ and $C$ are decomposed as $L=L_1\cup\dots\cup L_k$ and $C=C_1\cup\dots\cup C_k$ with corresponding subdivisions of $\Lambda$ and $\Gamma$.

Let us notice that the linearized Legendrian homology $L\bbH(\Gamma,C)$ in this case is the homology of the complex
\[
LH(\Gamma,C)=\K\langle\CC(\Gamma)\rangle\oplus\K\langle f_1,\dots, f_k\rangle.
\]
Here $f_j$ corresponds to the unique critical point (the minimum) $m_j$ of the function $H|_{C_j}$,  has grading 
$|f_j|=n-2-\ind(m_j) = n-2$, and satisfies $d_{LH}(f_j)=0$, $j=1,\dots,k$.

We introduce a product operation $\star$ on the complex
$LH(\Gamma,C)$.  If $c_1$ and $c_2$ are Reeb chords then define their product as 
$$
c_2\star c_1=\sum\limits_{c\in \K\langle\CC(\Gamma)\rangle}N_{c_1,c_2;b}\;b + \sum_{j=1}^{k} N_{c_1,c_2; f_j}\;f_j. 
$$
Here $N_{c_1,c_2;b}$  and  $N_{c_1,c_2; f_j}$  are the algebraic number of components of the $1$-dimensional moduli spaces $\MM^Y_{\Gamma,C}(c_1, c_2; b)$ and  $\MM^W_C(c_1,c_2; m_j)$, respectively, see Sections \ref{sec:mYLambda} and \ref{sec:mLL}. Let us also  define
\[
f_j\star c =
\begin{cases}
c &\text{ if $c\in\Cc_{ji}$ for some $i$},\\
0 &\text{ otherwise},
\end{cases}
\quad\text{ and }\quad
c\star f_j =
\begin{cases}
c &\text{ if $c\in\Cc_{ij}$ for some $i$},\\
0 &\text{ otherwise},
\end{cases}
\]
where $c\in \CC_{ij}$, where $\CC_{ij}$ is the set of chords connecting $\Gamma_i$ with $\Gamma_j$, $j=1,\dots,k$.

\medskip

Consider the map
$$
F^W_{(L,C)}\colon LH(\Gamma,C)\to LHA(\Lambda)
$$
defined as follows. For generators $c\in\CC(\Gamma)$,
$$
F^W_{(L,C)}(c)=\sum\limits_{|c|= \sum_{j=1}^{m}|b_j|+(n-2)}k_{c;b_1,\dots,b_m}b_1\dots b_m,
$$
where $c\in\CC(\Gamma), b_1,\dots,b_m\in\CC(\Lambda)$, $j=1,\dots,k$, and where the coefficient $k_{c;b_1,\dots,b_m}$ is equal to the algebraic number of elements of the $0$-dimensional moduli space $\MM^W_{(L,C)}(c;b_1,\dots, b_m)$, see Section \ref{sec:relmWCL}. (If $m=0$ then $\MM^{W}_{(L,C)}(c)$ can be non-empty only if $c$ is a chord with both endpoints on one component $\Gamma_j$ and the algebraic number of elements in a moduli spaces $\MM^{W}_{(L,C)}(c)$ of rigid disks with boundary on $C_j\cup L_j$ and no negative punctures equals the coefficient of $e_j$ in the expansion of $F^{W}_{(L,C)}(c)$.) For generators $f_1,\dots,f_k$,
$$
F^W_{(L,C)}(f_j)=e_j,\quad j=1,\dots, k.
$$
  We then have the following
\begin{thm}\label{thm:LH}
The operation $\star$ descends to an associative product on $L\bbH(\Gamma,C)$. The map $F^W_{(L,C)}\colon LH(\Gamma,C)\to LHA(\Lambda)$ is  a chain map 
which induces an  algebra isomorphism on homology. (Here we view $LHA(\Lambda)$ as a module generated by monomials and with product induced by the standard (concatenation) product on $LHA(\Lambda)$.) 
\end{thm}
\begin{rmk} Similarly to the product $m_2=\star$, one can define on $LH(\Gamma, C)$ operations $m_l, l>2$, using anchored in $(X,C)$ holomorphic discs in $\R\times Y$ with $l$ positive and $1$ negative punctures. These operations define an $A_\infty$-structure
on   $LH(\Gamma, C)$. The homomorphism $F^W_{(L,C)}$ can be included  into an $A_\infty$-quasi-isomorphism between the $A_\infty$-algebra $LH(\Gamma, C)$  and $LHA(\Lambda)$ with the trivial $A_\infty$-structure, where the product is just the tensor product and all the higher operations are equal to $0$.
\end{rmk}

\subsection{Legendrian homology algebra}\label{sec:LHA-surgery}
Let $\Lambda^-_0\subset Y_0$ be a Legendrian submanifold disjoint from $\Lambda=\bigcup_1^k \Lambda_j$.
Trajectories of $Z$ in $W$ which begin at points of $\Lambda^-_0$ form a Lagrangian cylinder $L_0$ which intersects $Y$ along a Legendrian submanifold $\Lambda^+_0\subset Y$. Then
$\oL'=\bigcup_0^k\oL_j$, where $\ol L_j=L_j\cap W$ is a Lagrangian cobordism between
$\Lambda'=\Lambda_0^-\cup\Lambda\subset Y_0$ and $\Lambda^+_0\subset Y$. According to  the second part of Proposition \ref{prop:leg-homology-hom} there is defined a homomorphism
$$
\left(F^W_{L'}\right)^*:L\bbH A(\Lambda^+_0)\to L\bbH A(\Lambda^-_0;\Lambda').
$$

\begin{thm}\label{thm:LHA}
$\left(F^W_{L'}\right)^*:L\bbH A(\Lambda^+_0)\to L\bbH A(\Lambda^-_0;\Lambda')$ is an isomorphism.
\end{thm}

Note that  Theorem \ref{thm:LHA} allows us to  compute the full Legendrian homology
 algebra of the co-core sphere $\Gamma$.

{
Restricting to the case $k=1$ we note that
it is possible to disjoin  $\Gamma$  from itself by a small Legendrian isotopy. The push-off $\Gamma'$ of $\Gamma$ can then be pushed further down to $Y_0$ using the Liouville flow of $Z$. Different choices of perturbation of $\Gamma$ give different Legendrian spheres in $Y_0$ which are generally not Legendrian isotopic. It is, in particular, possible to choose perturbations either so that the push-down of $\Gamma'$  is a small Legendrian sphere $\Lambda_{\rm f}\subset Y_0$ linking the attaching sphere $\Lambda_{\rm a}$ once, or so that it is a parallel copy $\Lambda_{\rm p}$ of $\Lambda_{\rm a}$. Choosing $\Lambda_{\rm f}\subset Y_0$ as $\Lambda_0^-$ and the corresponding push-off $\Gamma'$ of $\Gamma$ as $\Lambda_0^+$, we are in a situation where Proposition \ref{prop:rel-q} applies, and it provides an explicit description of the algebra $LHA(\Lambda_{\rm f};\Lambda_{\rm f}\cup\Lambda)$ in terms of the  deformed algebra $LHA(\Lambda; q)$, where $q$ is a generator of $H_0(\Lambda)$.
According to Theorem  \ref{thm:LHA}  the homology $L\bbH A(\Lambda_{\rm f};\Lambda_{\rm f}\cup\Lambda)$ is isomorphic to $L\bbH A(\Gamma')=L\bbH A(\Gamma)$.

  Note that the complex in Theorem \ref{thm:LH} which is quasi-isomorphic to $LH(\Gamma,C)$ looks like the linearization of the differential algebra $LH A(\Lambda_{\rm f};\Lambda_{\rm f}\cup\Lambda)$ which corresponds to the trivial augmentation. Nevertheless, it is not clear that the formula in Theorem \ref{thm:LH} follows from Theorem \ref{thm:LHA} and Proposition \ref{prop:rel-q} because the algebra $LH A(\Lambda_{\rm f};\Lambda_{\rm f}\cup\Lambda)$ may have other augmentations with isomorphic linearizations.}

\section{About the proofs}\label{sec:proofs}
We sketch in this section the main ideas which enter the proofs of the results discussed in this paper.
 
\subsection{Dynamics of the Reeb flow after Legendrian surgery}\label{sec:Reeb-dynamics}
We first explain the mechanism of creation of new periodic Reeb orbits as a result of Legendrian surgery.
To clarify the picture we begin with the case when $Y_0$ is just $\R^{2n-1}$ with the standard contact  form $\lambda_0=dz-\sum\limits_1^{n-1}y_jdx_i$ and when the surgery locus $\Lambda$ consists of one Legendrian sphere. In particular, the Reeb vector field $R_0=\frac{\p}{\p z}$ has no periodic orbits. We assume further that the Reeb flow isotopy $\Phi_{R_0}^t(\Lambda)$ intersects $\Lambda$ transversely. In particular, the set $\CC(\Lambda)$ consists of finitely many Reeb chords $c_1,\dots, c_m$ in general position.

Consider the contact manifold $(Y,\lambda)$ which results from surgery on $\Lambda$. The Reeb vector field $R$ on $Y$ satisfies $R=R_0$ outside of a tubular neighborhood $U$ of $\Lambda$. Under surgery, $U$ is replaced by the unit cotangent bundle of the $n$-disk which is the Lagrangian core disk of the attached handle. One can choose the contact form $\lambda$ so that away from a small neighborhood  of $\p U$ the Reeb dynamics inside the handle is the geodesic flow of the flat disk. Moreover, the perturbation necessary for the boundary adjustment can be made arbitrarily small.
  
The dynamics of the flow of $R$ on $Y$ can then be described as follows. 
Consider the $1$-jet space $J^1(S^{n-1})\approx T^{\ast}S^{n-1}\times\R$ of the $(n-1)$-sphere endowed with its standard contact form $\alpha=dz+p\,dq$, where $(q,p)$ are canonical coordinates on $T^{\ast}S^{n-1}$ and $z$ is a coordinate in the $\R$-factor. Let $S\subset J^{1}(S^{n-1})$ denote the $0$-section and define $U_\eps=\{|z|\leq\eps, \|p\|\leq\eps\}$, $\eps>0$, where we view of $S$ as the unit sphere in $\R^{n}$ in order to define $\|p\|$. Note that $U_{\eps}$ is a neighborhood of $S$ of size $\eps$. Let $V_{\pm} =\{z=\pm\eps\}$. The   form $\omega_\pm=d\alpha|_{V+\pm}$ is symplectic and  $(V_\pm,\omega_\pm)$ is symplectomorphic to the cotangent  disk bundle of $S^{n-1}$ of radius $\eps$. For $\eps>0$ sufficiently small, there exists an embedding $j\colon U_\eps\to Y_0$ with
$j^\ast\lambda_0=\alpha$ and $j(S)=\Lambda$, and which hence allows us to identify $U_\eps$ with a neighborhood of $\Lambda\subset Y_0$. We will keep the notations $U_\eps$ and $V_\pm$ for the images of $U_\eps$ and $V_\pm$, respectively, under the contact embedding $j$.  The Reeb flow of $\lambda$ defines a symplectomorphism $i\colon (V_-,d\lambda|_{V_-})\to (V_+,d\lambda|_{V_+})$. Now the effect of the surgery on Reeb dynamics can  be described as follows.  Outside $U_\eps$ we have $R=R_0$, all the trajectories which enter $U_\eps$ through $V_-$ exit through $V_+$ and thus there is defined a holonomy symplectomorphism $\wt\tau\colon (V_-,d\lambda|_{V_-})\to (V_+,d\lambda|_{V_+})$. One can show that $\wt\tau=\tau\circ i$, where $\tau:(V_+,d\lambda|_{V_+})\to ( V_+,d\lambda|_{V_+})$ is a {\em generalized symplectic Dehn twist}, see \cite{Arnold}, where we identify $(V_+,d\lambda|_{V_+})$ with  the cotangent disk bundle of the sphere $S^{n-1}$ of radius $\eps$ .

In the case under consideration ($Y_0=\R^{2n-1}$) there are no closed orbits of $R$ in $Y$ outside $U_\eps$. Hence, if $\gamma$ is a closed orbit then it has to first exit $U_{\eps}$ and then reenter it again.  We claim that if $\eps>0$ is chosen sufficiently small then the outside portion of any closed trajectory $\gamma$   has to be  contained in a neighborhood of a union of Reeb chords  from $\CC(\Lambda)$. Indeed, the assumption that the isotopy $\Phi_{R_0}^t(\Lambda)$ intersects $\Lambda$ transversely implies that for any $\delta>0$ there exists an $\eps>0$ such that a Reeb trajectory at distance at least $\delta$ from the union of the chords in $\CC(\Lambda)$ which starts on $V_{+}$ never intersects $U_\eps$ again.

Conversely, we claim that if $w=(c_{i_1} \dots c_{i_N})\in LH^{\rm cyc}(\Lambda)$ is a cyclic word, then
for all sufficiently small $\delta>0$ there is $\eps>0$ a unique periodic orbit $\gamma$ which passes through the $\delta$-neighborhoods of chords $c_{i_1},\dots,c_{i_N}$ in the given cyclic order and exactly once through the flow tube around $c$ for each index $i_j$ such that $c_{i_j}=c$. We will consider here only the case $N=1$. The general case is similar.
Thus, assume $w=(c)$ for a chord $c\in\CC(\Lambda)$. Let $a,b\in\Lambda$ be the beginning and the end  points of the chord $c$. Let $T$ denote the action of the chord $c$, $T=\int_c\lambda$. Then $\Phi_{R_0}^T(a)=b$. Write $a_\pm=\Phi_{R_0}^{\pm\eps}(a)\in V_{\pm}$. The isotopy $\Phi_{R_0}^t(\Lambda)$, $t\geq 0$ intersects $\Lambda$ transversely at $t=T$ in the point $b$.  Fix $\sigma>0$ such that $3\sigma$  is smaller than the action of any chord in $\CC(\Lambda)$. Call a path $\gamma\colon [0,1]\to V_+$ with $\gamma(0)=
a_+$  {\em  small} if  for any $t\in[0,1]$ there exists $s\in [T-\sigma, T+\sigma]$ such that $\Phi_{R}^s(\gamma(t))\in V_-$.  Denote by $A$ the set of all end-points of small paths.
The holonomy along leaves of the Reeb flow of $R$ defines an embedding $f\colon A\to V_-$. We need to show  that   there exists a  unique point
$x\in  A$ such that $\wt\tau(f(x))=x$. If a point $x\in  A$ satisfies this equation
then
$$
x\in  Z:=\bigcap\limits_{-\infty}^{\infty}  (\wt \tau\circ f)^k(A).
$$ 
Note that  we have $Z=\bigcap\limits_1^{\infty} Z_j$, where $Z_j:= \bigcap\limits_{-j}^j (\wt\tau\circ f)^k(A_1)$,  $Z_1\supset Z_2\supset\dots \supset Z$.
 One can check that the closed sets $Z_j$ are non-empty and that ${\rm diam}(Z_j)\to 0$ as $j\to\infty$. Hence $Z$ consists of one point $x\in Z$, but if  $x\in Z$ then $\wt\tau(f(x))\in Z$. Thus,  $x=\wt\tau(f(x))$, i.e.~$x$ belongs to a unique periodic orbit of $R$  which intersects $A$ and passes once in a neighborhood of $c$.

In the case of a general contact manifold $Y_0$ we can similarly  show that if the contact form $\lambda_0$ on $Y_0$ is chosen in such a way that $\Lambda$ does not intersect any periodic orbits of its Reeb field, then
given any $C>0$ one can find a sufficiently small $\eps>0$ such that all  orbits of $R$ of period smaller than $C$ survive the surgery and all newly created orbits of  period smaller $C$  are in 1-1 correspondence with cyclic words from  $LH^{\rm cyc}(\Lambda)$.  We call these orbits {\em essential}.
Thus  essential  periodic orbits in $Y$ are in 1-1 correspondence with the generators of  the contact homology  complex
$LCH(X_0,\Lambda)=CH(X_0)\oplus LH^{\rm cyc}(\Lambda)$. On the other hand, as we discuss next, passing to the limit  $C\to\infty$ the  contribution of  other, non-essential, orbits can be discarded.

Indeed, consider monotonically decreasing and increasing sequences $\eps_n\to 0$ and $C_k\to\infty$, respectively, and let $\lambda_n$ be the corresponding sequence of contact forms on $Y$. Let $CH^{\le C_k}(\Lambda_n)$ be the subcomplex of $CH(\lambda_n)$ generated by closed Reeb orbits $\gamma$ 
with $\int_\gamma \lambda_n \le C_k$. We denote the corresponding homology by $C\bbH^{\leq C_k}(\lambda_n)$. We can assume that $\lambda_n\geq\lambda_{n+1}$, $n=1,2,\dots,$ and hence we have the following commutative diagram
  \[
  \xymatrix{
 &  \ar[d]                                               &  \ar[d]                                     &       &
  \ar[d]\\
 \ar[r]&C\bbH^{\leq C_k}(\lambda_n)\ar[r]\ar[d]&C\bbH^{\leq C_{k+1}}(\lambda_n)\ar[r]\ar[d]&\quad\dots\quad  &
 C\bbH (\lambda_n)\ar[d]\\
 \ar[r]&C\bbH^{\leq C_k}(\lambda_{n+1})\ar[r]\ar[d]&C\bbH^{\leq C_{k+1}}(\lambda_{n+1})\ar[r]\ar[d]&\quad\dots\quad &
 C\bbH (\lambda_{n+1})\ar[d]\\ 
  \vdots &\vdots &\vdots &\vdots &\vdots \\
    \ar[r]& C\bbH^{\leq C_k}(\lambda_{\infty})\ar[r]&C\bbH^{\leq C_{k+1}}(\lambda_{\infty})\ar[r]&\quad\dots\quad &\quad\quad
  }
  \]
 
Here the vertical arrows  are monotonicity homomorphisms, the horizontal arrows are action window extension homomorphisms,  and 
$$
C\bbH^{\leq C_k}(\lambda_{\infty}):={\lim_{\longrightarrow\,n}} C\bbH^{\leq C_k}(\lambda_{n}).
$$
We define an {\em essential complex} $CH^{ess}(\{\lambda_n\}) := \lim\limits_{\longrightarrow\,k}
\lim\limits_{\longrightarrow\,n} CH^{\leq C_k}(\lambda_{n})$.
Note that the vertical arrows in the last column are all isomorphisms and hence
$\lim\limits_{\longrightarrow\,n} C\bbH(\lambda_n)=C\bbH(Y)$. On the other hand,
it is straightforward to check that the vertical and the horizontal limits commute, so that
$\lim\limits_{\longrightarrow\,k} C\bbH^{\leq C_k}(\lambda_{\infty})=C\bbH(Y)$.
But $\lim\limits_{\longrightarrow\,k} C\bbH^{\leq C_k}(\lambda_{\infty})$ 
is equal to the homology of the essential complex, which is generated by essential orbits only.

 In symplectic homology orbits are parametrized, and hence every generator  of  $LH^{\rm cyc}(\Lambda)$ corresponds to a circle of orbits    in symplectic homology. To connect this Morse-Bott case to the complex $LH^{\rm Ho+}(\Lambda)$, we choose on every orbit an auxiliary Morse function with one minimum and one maximum on every of the above chords generating the orbit. The minimum and maximum  marked orbits correspond, respectively,  to the generators of the complexes $\widecheck{LHO}^{+}(\Lambda)$
and $\widehat{LHO}^{+}(\Lambda)$, which together generate $LH^{\rm Ho+}(\Lambda)$.

Similarly, the generators of  $LHA(\Lambda_0^{-};\Lambda)$ (resp.~of the vector space $LHA(\Lambda)$) correspond to (essential) chords connecting points on the Legendrian manifolds $\Lambda^{+}_{0}\subset Y$ (resp.~$\Gamma\subset Y$) after a Legendrian surgery with a basis $\Lambda$ disjoint from $\Lambda_0^{-}$.
\subsection{Filtration argument}\label{sec:filtration}

The statements that $d^2=0$ in Theorems \ref{thm:d-ch}, \ref{thm:d-sh1} and \ref{thm:d-sh2}, as well as the chain map statements in Theorems \ref{thm:f-ch}, \ref{thm:f-sh1}, \ref{thm:SH}   and \ref{thm:LH}  follow from
a detailed analysis of boundaries in the compactification of the moduli spaces involved. In order to prove the isomorphism part in Theorems \ref{thm:f-ch}, \ref{thm:f-sh1}, \ref{thm:SH},  and \ref{thm:LH} we use a filtration argument. Observe first that all the maps preserve the natural action filtrations. Consider, for instance, the map $F^W_{LCH}\colon CH(X)\to LCH(X,\Lambda)=CH(X_0)\oplus LH^{\rm cyc}(\Lambda)$. As was pointed out above, for an appropriate choice of contact form on $Y$, the new orbits in $Y$ are
in 1-1 correspondence with the generators of the complex $LH^{\rm cyc}(\Lambda)$, while the old orbits correspond to generators of $CH(X_0)$. If we order these generators according to their action values, then the fact that the homomorphism  $F^W_{LCH}$ is action decreasing translates into the fact that the matrix of $F^W_{LCH}$ is lower triangular with respect to action. The most delicate part of the proof is the analysis of  the diagonal elements of this matrix.
In the next section we sketch an argument that for an appropriate regular choice of the almost complex structure, all the diagonal elements are equal to $\pm 1$, and hence for such a choice of extra data, the map  $F^W_{LCH}$  is an isomorphism already on the chain level. The proofs of Theorems   \ref{thm:LH} and  \ref{thm:LHA} follow similar lines of arguments.

The claim that the product is preserved in Theorem \ref{thm:LH} follows from analyzing the boundary of the 1-dimensional   moduli spaces $\MM^W_{(L,C)}(c_1,c_2;b_1,\dots,b_m)$ obtained by gluing moduli spaces  $\MM^Y_{\Lambda, L}(c_1,c_2;c)$  and $ \MM^W_{(L,C)}(c;b_1,\dots,b_m)$. This two-story   holomorphic building   represents the push-forward $F^W_{(L,C)}(c)\in LHA(\Lambda)$ of the product $c=c_2\star c_1$ by the isomorphism  $F^W_{(L,C)}$. Besides splittings which contribute to boundary terms, the only other possible splitting is a configuration of  two curves from $ \MM^W_{(L,C)}(c_2;b_1,\dots,b_l)$ and  $ \MM^W_{(L,C)}(c_1 ;b_{l+1},\dots,b_m)$,
 $l=1,\dots, m-1$, on the same level  connected at an intersection point from  $L\cap C$. This 
configuration can  be interpreted as the push-forward of the pair $(c_2,c_1)$ to the tensor product
$F^W_{(L,C)}(c_2)F^W_{(L,C)}(c_1)$ of the images $F^W_{(L,C)}(c_2)$ and $F^W_{(L,C)}(c_1)$ in the algebra $LHA(\Lambda).$
 \subsection{Analysis of diagonal elements}\label{sec:diagonal}
In this subsection we sketch an argument that shows that the lower triangular chain map $F^{W}_{LCH}$ has all diagonal elements equal to $\pm 1$. As explained above in Section
\ref{sec:Reeb-dynamics}, it is sufficient to deal only with essential orbits. For the old orbits this is obvious: the diagonal elements correspond to trivial holomorphic cylinders in the symplectization of $Y_0$. For new orbits corresponding to words of Reeb chords we will prove this as follows.
\begin{itemize}
\item[$(\mathrm{A})$] First, we establish the following: for each Reeb chord $a$ of $\Lambda$ and  the Reeb chord $a'$ of the co-core sphere $\Gamma$ that corresponds to it there is algebraically $\pm 1$ holomorphic strip in $W$ with boundary on $L\cup C$ positive puncture at $a'$ and negative puncture at $a$. 
\item[$(\mathrm{B})$] Second, we construct holomorphic disks counted in the map $F^{W}_{LCH}$ from   these disks by gluing, and arguing by compactness and small action we find that the count of these disks must equal $\pm 1$ as well.
\end{itemize}    

To verify  $(\mathrm{A})$  take  a Reeb chord $a$ of $\Lambda$. Choose a parameterization of the handle so that the endpoints of $a$ on $\Lambda$ correspond to antipodal points of the  attaching sphere. In a suitable model for the handle the existence of one disk $D$ with required properties is then obvious: the disk lies in a flat complex line in this model. Furthermore, in this model it is immediate that the linearization of the $\bar\pa$-operator at this disk is surjective. In order to show uniqueness of the disk one thus need to show that any other disk with the given properties lies in a small neighborhood of $D$. To see that this is the case we use a local projection to the normal bundle of $D$ in combination with an action argument: if the disk is not contained in the neighborhood of $D$ its area is bigger than the action difference between $a'$ and $a$. Thus we conclude that for this choice of attaching map there is exactly one geometric disk from $a'$ to $a$. Deforming the attaching map we get a parameterized moduli space of disks connecting $a'$ to $a$. By compactness this parameterized moduli space is a $1$-manifold up to splittings. However, in this case there can   be no splittings due to the small action difference between $a'$ and $a$. We conclude that $(\mathrm{A})$ holds.  

In order to show $(\mathrm{B})$ let us first verify that $(\mathrm{A})$ implies $(\mathrm{B})$ for orbits corresponding to $1$-letter words of Reeb chords and also that it implies that for every Reeb chord $w'$ of $\Gamma$ corresponding to a $2$-letter word $w=a_1a_2$ of Reeb chords there is algebraically $\pm 1$ disk connecting $w'$ to $w$. In the orbit case we consider the orbit $\tilde a$ corresponding to a Reeb chord $a$ on $\Lambda$ and the gluing of the disk connecting $a'$ to $a$ to itself at the Lagrangian intersection punctures. This gives a $1$-dimensional space of holomorphic annuli with one boundary puncture on each component mapping to $a'$ and $a$ respectively. The other end of such a moduli space must correspond to a broken curve. Since the orbit $\tilde a$ is the only orbit or chord with action between $a'$ and $a$ we find that the annulus must break into two disks, one in the symplectization of $Y$ connecting $a'$ to $\tilde a$ and one in $W$ connecting $\tilde a$ to $a$. Counting curves in the boundary shows that $(\mathrm{B})$ holds for $1$-word orbits $\tilde a$. 

The claim for $2$-letter word chords $w$ of $\Gamma$can be  proved in a similar way by gluing the strip connecting $a_1'$ to $a_1$ to that connecting $a_2'$ to $a_2$ at a Lagrangian intersection puncture. Studying the boundary of the resulting $1$-dimensional moduli space one finds one disk in the symplectization of $Y$ connecting $a_1'a_2'$ to $w$ and one disk in $W$ connecting $w$ to $a_1a_2$. 

The general case then follows by applying this argument inductively: self-gluing of a disk  from a chord $w'$ of $\Gamma$ to the corresponding word of chords $w$ gives a disk from the orbit $\tilde w$ to $w$, and gluing the basic strip connecting a chord $c'$ of $\Gamma$ to the corresponding chord $c$ of $\Lambda$ to the disk connecting $w'$ to $w$ gives a disk connecting the chord of $\Gamma$ corresponding to the word $cw$ to the word $cw$.       

\section{Examples and applications}\label{sec:examples}
\subsection{Legendrian unknots}
Assume that $\Lambda\subset\pa X_0$ is a Legendrian sphere. Then $LHA(\Lambda)$ is a unital algebra with unit $e$ corresponding to the empty word of Reeb chords. For simplicity, we write $1$ instead of $e$ for the empty word of Reeb chords in this case.

The Legendrian unknot $\Lambda_{U}\subset \R^{2n-1}=J^{1}(\R^{n-1})$ is depicted in Figure \ref{fig:unknot}.
It has one Reeb chord $a$ with grading $|a|-1$, see \cite{EES2}. Using an adapted almost complex structure $J_0$ on $J^{1}(\R^{n-1})\times\R=T^{\ast}\R^{n-1}\times\C$ induced from the standard complex structure on $\C^{n-1}=T^{\ast}\R^{n-1}$, the moduli space $\MM^{\R^{2n+1}}_{\Lambda_U}(a;\varnothing)$ of $J_0$-holomorphic disks with positive puncture at $a$ and boundary on $\Lambda_{U}\times\R$ is $C^1$-diffeomorphic to $S^{n-2}\times\R$ and the boundary evaluation map to $\Lambda_U$ has degree $1$. (For $n=2$, there is an orientation issue here: the degree is one provided we use the orientation on the moduli space induced by the null-cobordant spin structure on $\Lambda_{U}\approx S^{1}$.)

The algebra $LHA(\Lambda_{U})$ is then the algebra on one generator
\[
LHA(\Lambda_{U})=\Q[a]
\]
with trivial differential $d a=0$, where we write $d=d_{LHA}$. We next describe the computation of $LH^{\rm cyc}(\Lambda_U)$ and $LH^{\Ho}(\Lambda_U)$. Note that $LHO^{+}(\Lambda_U)$ can be identified with the vector
space generated by the set of monomials $\{a^{k}\}_{k=1}^{\infty}$.
\begin{figure}
\labellist
\small\hair 2pt
\pinlabel $x_1$   at 5 24
\pinlabel $x_{n-1}$   at 324 102
\pinlabel $z$ at 123 210
\pinlabel $a$ at 160 102
\endlabellist
\centering
\includegraphics[width=.6\linewidth]{unknot}
\caption{The front of the Legendrian unknot, i.e.~the projection of $\Lambda_U$ to $J^{0}(\R^{n-1})$, $(x_1,\dots,x_{n-1})$ are coordinates on $\R^{n-1}$ and $z$ is a coordinate for the function values of $0$-jets.}
\label{fig:unknot}
\end{figure}

Since $|a|=n-1$ we find that $(a^{2j})=0$ in $LH^{\rm cyc}(\Lambda_U)$ if and only if $n$ is even. The grading of an element $(w)$ in $LH^{\rm cyc}(\Lambda_U)$ is given by $|(w)|=|w|$. Thus, if $n$ is odd then $LH^{\rm cyc}(\Lambda_U)$ is as in Table \ref{tab:cycunkotodd}, and if $n$ is even then it is as in Table \ref{tab:cycunknoteven}.

\begin{table}[htp]
\begin{center}
{\renewcommand{\arraystretch}{1.5}
\renewcommand{\tabcolsep}{0.2cm}
\begin{tabular}{|c|c|}
\hline
{Degree}& Generator\\
\hline
$n-1$ & $(a)$\\
\hline
$2n-2$ & $(a^{2})$\\
\hline
\vdots & \vdots\\
\hline
$k(n-1)$ & $(a^{k})$\\
\hline
\vdots & \vdots\\
\end{tabular}}
\end{center}
\caption{The complex $LH^{\rm cyc}(\Lambda_U)$ for $n$ odd. The differential on the complex is trivial.}\label{tab:cycunkotodd}
\end{table}

\begin{table}[htp]
\begin{center}
{\renewcommand{\arraystretch}{1.5}
\renewcommand{\tabcolsep}{0.2cm}
\begin{tabular}{|c|c|}
\hline
{Degree}& Generator\\
\hline
$n-1$ & $(a)$\\
\hline
$3n-3$ & $(a^{3})$\\
\hline
\vdots & \vdots\\
\hline
$(2k+1)(n-1)$ & $(a^{2k+1})$\\
\hline
\vdots & \vdots\\
\end{tabular}}
\end{center}
\caption{The complex $LH^{\rm cyc}(\Lambda_U)$ for $n$ even. The differential on the complex is trivial.}\label{tab:cycunknoteven}
\end{table}

Decorated cyclically composable monomials are in 1-1 correspondence with decorated monomials. We write ${\widehat a}^{k}$ and ${\widecheck a}^{k}$ for the class of the monomial $a \stackrel{k}{\dots} a$ decorated with a ``hat'' and with a ``check'' respectively. If $w$ is a monomial of Reeb chords and $\widehat w$ ($\widecheck w$) denotes $w$ with one variable decorated with ``hat'' (``check'') then  the grading on the complex $LH^{\Ho}(\Lambda_U)$ is $|\widehat w|=|w|+1$ and $|\widecheck w|=|w|$ and the complex $LH^{\Ho}(\Lambda_U)$ is as in Tables \ref{tab:hohunknotodd} and \ref{tab:hohunknoteven} when $n$ is odd and even respectively.

\begin{table}[htp]
\begin{center}
{\renewcommand{\arraystretch}{1.5}
\renewcommand{\tabcolsep}{0.2cm}
\begin{tabular}{|c|c|}
\hline
{Degree}& Generator\\
\hline
$0$  & $\tau$\\
\hline
$n-1$ & $\widecheck a$\\
\hline
$n$  & $\widehat a$\\
\hline
$2n-2$ & ${\widecheck a}^{2}$\\
\hline
$2n-1$ & ${\widehat a}^{2}$\\
\hline
\vdots & \vdots\\
\hline
$k(n-1)$ & $\widecheck a^{k}$\\
\hline
$k(n-1)+1$ & $\widehat a^{k}$\\
\hline
\vdots & \vdots\\
\end{tabular}}
\end{center}
\caption{The complex $LH^{\Ho}(\Lambda_U)$ for $n$ odd. The differential on the complex is trivial.}\label{tab:hohunknotodd}
\end{table}

\begin{table}[htp]
\begin{center}
{\renewcommand{\arraystretch}{1.5}
\renewcommand{\tabcolsep}{0.2cm}
\begin{tabular}{|c|c|}
\hline
{Degree}& Generator\\
\hline
$0$  & $\tau$\\
\hline
\hline
$n-1$ & $\widecheck a$\\
\hline
$n$  & $\widehat a$\\
\hline
\hline
$2n-2$ & ${\widecheck a}^{2}$\\
\hline
{} & $\uparrow$\\
\hline
$2n-1$ & ${\widehat a}^{2}$\\
\hline
\hline
\vdots & \vdots\\
\hline
\hline
$(2k-1)(n-1)$ & $\widecheck a^{2k-1}$\\
\hline
$(2k-1)(n-1)+1$ & $\widehat a^{2k-1}$\\
\hline
\hline
$2k(n-1)$ & ${\widecheck a}^{2k}$\\
\hline
{} & $\uparrow$\\
\hline
$2k(n-1)+1$ & ${\widehat a}^{2k}$\\
\hline\hline
\vdots & \vdots\\
\end{tabular}}
\end{center}
\caption{The complex $LH^{\Ho}(\Lambda_U)$ for $n$ even. The differential on the complex is as indicated by the arrows.}\label{tab:hohunknoteven}
\end{table}

We next consider $\Lambda_{U}\subset\R^{2n-1}\subset S^{2n-1}=\pa B^{2n}$ where we include $\Lambda_{U}$ by embedding it in a (small) Darboux chart. Consider the contact form on $S^{2n-1}$ induced by looking at it as the ellipsoid
\[
E=\left\{(z_1,\dots,z_n)\in\C^{n}\colon \sum_{j=1}^{n}\frac{|z_j|^{2}}{a_j}=1\right\},
\]
where $a_j$ are real numbers with $a_1>a_2>\dots>a_n>0$ which are linearly independent over $\Q$. To facilitate calculations we take $a_1=1$ and $a_j\gg 1$ for $j\ne 1$. Reeb orbits in $E$ correspond to intersections of $E$ with the complex coordinate lines.   Furthermore, the Conley-Zehnder index increases linearly with the length of Reeb orbits and thus for homology computations we may take $a_j$, $j>1$ sufficiently large and consider input only from the shortest Reeb orbit in the $z_1$-line and its iterates. If $\gamma^{k}$ denotes the $k^{\rm th}$ multiple of the shortest orbit then its Conley-Zehnder index satisfies
\[
|\gamma^{k}|=n-1+2k.
\]
Consider next the inclusion of $\Lambda_{U}$ into a small Darboux chart located near the shortest Reeb orbit. Straightforward calculation shows that the Reeb flow on $S^{2n-1}$ outside the Reeb orbit preserve the $n$-tori which are products of circles in the coordinate planes and that the Reeb flow in such an $n$-torus is translation along a vector with components which are independent over $\Q$.   It follows that if we choose the Darboux chart sufficiently small then the only Reeb chords of $\Lambda_{U}$ which have action (and hence index) below a given constant is the one completely contained in the chart. As above we denote it $a$.

We first note that for grading reasons the augmentation induced by the filling of $S^{2n-1}$ is trivial. Furthermore, a straightforward monotonicity argument shows that any holomorphic disk which contributes to the contact homology differential of $\Lambda_{U}\subset S^{2n-1}$ must be entirely contained in the Darboux chart. In particular it follows that the complex $LH^{\Ho}(\Lambda_U)$ above computes the symplectic homology of the manifold which results from attaching a handle to $B^{2n}$ along $\Lambda_{U}$. Attaching a Lagrangian handle to $B^{2n}$ along $\Lambda_{U}$ gives the manifold $T^{\ast} S^{n}$. Using Table \ref{tab:hohunknotodd} we get
\[
S\bbH_{j}(T^{\ast}S^{2m+1})=
\begin{cases}
\Q &\text{ for } j=0,\\
\Q &\text{ for } j=2rm,\quad r=1,2,\dots,\\
\Q &\text{ for } j=2rm+1,\quad r=1,2,\dots,\\
0 &\text{ otherwise. }
\end{cases}
\]
and similarly using Table \ref{tab:hohunknoteven} we find
\[
S\bbH_{j}(T^{\ast}S^{2m})=
\begin{cases}
\Q &\text{ for } j=0,\\
\Q &\text{ for } j=r(2m-1),\quad r=1,3,5\dots,\\
\Q &\text{ for } j=r(2m-1)+1,\quad r=1,3,5,\dots,\\
0 &\text{ otherwise. }
\end{cases}
\]

We consider next the linearized contact homology of $\pa (T^{\ast} S^{n})$. The complex which gives that homology is $CH(B^{2n})\oplus LH^{\rm cyc}(\Lambda_U)$. The grading of the generators of $CH(B^{2n})$ is
\[
|\gamma^{k}|=n-1+2k
\]
and of the generators of $LH^{\rm cyc}(\Lambda_U)$
\[
|(a^{j})|=j(n-1)
\]
Thus, if $n=2m+1$ is odd then all generators have even degree and the differential on the complex vanishes. We find that
\[
\{\gamma^{k}, (a^{j})\}_{j,k=1}^{\infty}
\]
gives a basis of
\[
C\bbH(T^{\ast} S^{2m+1}).
\]

If on the other hand $n=2m$ is even then $(a^{2j})=0\in LH^{\rm cyc}(\Lambda_U)$ since $|a|=n-1$ is odd. A non-trivial differential on $CH(B^{2n})\oplus LH^{\rm cyc}(\Lambda_U)$ would have to take an orbit $\gamma^{k}$ to a cyclic word $(a^{l})$ and we would have
\[
n-1+2k -l(n-1)=1.
\]
Hence $l$ is even and $(a^{l})=0$. We conclude that there is no non-trivial differential in this case either and that
\[
\{\gamma^{k}, (a^{2j-1})\}_{j,k=1}^{\infty}
\]
gives a basis of
\[
C\bbH(T^{\ast} S^{2m}).
\]

Finally, we consider the reduced symplectic homology. The corresponding complex is $SH^{+}(B^{2n})\oplus LH^{\Ho+}(\Lambda_U)$. The complex $SH^{+}(B^{2n})$ is generated by Reeb orbits with critical points on them. We write ${\widehat\gamma}^{k}$ and $\widecheck{\gamma}^{k}$ for $\gamma^{k}$ with a maximum and minimum, respectively. Then
\begin{align*}
|\widecheck{\gamma}^{k}|&=n-1+2k,\\
|\widehat{\gamma}^{k}|&=n-1+2k+1.
\end{align*}
Furthermore, there is a holomorphic curve with marker connecting $\widecheck\gamma^{k+1}$ to $\widehat\gamma^{k}$ and we find that the reduced symplectic homology is given by
\[
S\bbH_j^{+}(T^{\ast}S^{n})=S\bbH_j(T^{\ast}S^{n})
\]
if $j\ne 0,n+1$ and that
\[
S\bbH_{n+1}^{+}(T^{\ast}S^{n})=\Q,\quad S\bbH_0^{+}(T^{\ast}S^{n})=0.
\]

\begin{rmk}
Note that $S\bbH(T^{\ast}S^{n})$ is a ring with unit where the grading of the unit is $n$. This implies that the element in $LH^{\Ho}(\Lambda_U)$ which represents the unit is $(\widehat a)$. Here it can be noted that $a$ represents the fundamental class of the linearized Legendrian contact homology of $\Lambda_{U}$, see \cite{EESa}. Generally, the fundamental class of Legendrian contact homology seems intimately related to the multiplicative unit in symplectic homology.
\end{rmk}

\begin{rmk}
More generally, the structure of the complex $SH(X)\oplus LH^{\Ho}(\Lambda)$ can be used to prove the chord conjecture for Legendrian spheres $\Lambda$ in the boundary of a symplectic manifold $X_0$ with vanishing symplectic homology: $SH(X_0)=0$. If the manifold resulting from handle attachment along $\Lambda$ has non-vanishing symplectic homology, then there must exists some element in $LH^{\Ho}(\Lambda)$ which is homologous to the fundamental class of the manifold (the minimum of the Morse function). For grading reasons it is not $\tau$. Thus there is a chord in this case. (As well as a holomorphic curve with marker connecting an orbit decorated by a minimum to a cyclic word decorated by a maximum. For the unknot it can be found explicitly.)
If on the other hand $1=0$ then something in $LH^{\Ho}(\Lambda)$ must be mapped by the differential to $\tau$ and there is again a chord.
\end{rmk}

\subsection{Vanishing results}\label{sec:vanish}
We show that if the equation $d c=1$ holds for some Reeb chord $c$ in the Legendrian homology algebra $LHA(\Lambda)$, where we write $d=d_{LHA}$, then it follows that surgery along $\Lambda$ does not change symplectic homology: $S\bbH(X)=S\bbH(X_0)$, where $X$ is the result of attaching a handle along $\Lambda$. To this end we use a left-right module over $LHA(\Lambda)$ defined as follows. Let $\CC$ denote the set of Reeb chords of $\Lambda$ and let  $M(\Lambda)$ denote the left-right $LHA(\Lambda)$-module generated by
\begin{itemize}
\item hat-decorated Reeb chords $\widehat c$ for $c\in\CC$, where $|\widehat c|=|c|+1$,
\item and an auxiliary variable $x$ of grading $|x|=0$.
\end{itemize}
Define $d_{M}\colon M(\Lambda)\to M(\Lambda)$, where
\begin{itemize}
\item $d_M$ acts as the algebra differential on coefficients: $d_{M}c= d_{LHA} c$, where $d_{LHA}$ is the contact homology differential,
\item $d_{M}\widehat c=xc-cx-S(d_Mc)$, where $S$ is as in the differential on $LH^{\Ho+}(\Lambda)$, see \eqref{eq:defS}, 
\item and $d_{M} x=0$.
\end{itemize}
Consider the quotient $M^{\rm cyc}(\Lambda)=M(\Lambda)/\!\!\sim$ where
\[
c_1\dots c_m\, u\, b_1\dots b_l \,\sim\, (-1)^{|\mathbf{c}|(|u\mathbf{b}|)}u\,b_1\dots b_lc_1\dots c_m.
\]
 Here $c_j\in\CC$, $j=1,\dots,m$, $b_j\in\CC$, $j=1,\dots,l$, $u=\widehat a$, $a\in\CC$ or $u=x$, and $\mathbf{c}=c_1\dots c_m$, $\mathbf{b}=b_1\dots b_l$. Then $M^{\rm cyc}$ is a $\K$-module and the map $d_M$ descends to a map on $M^{\rm cyc}(\Lambda)$ which we still denote $d_M$.

\begin{lma}[Ekholm-Ng]\label{lma:E-N}
The map $d_{M}$ is a differential. The homology of  $M^{\rm cyc}(\Lambda)$ is isomorphic to the homology of $LH^{\Ho}(\Lambda)$. The isomorphism takes $x$ to $\tau$, takes words of the form $c_1\dots c_jxc_{j+1}\dots c_m$ to $c_1\dots\widecheck{c_j} c_{j+1}\dots c_m$, and takes $c_1\dots c_j\widehat c_{j+1} c_{j+2}\dots c_m$ to the corresponding hat-decorated monomial word.
\end{lma}

It is straightforward to check that stable tame isomorphisms of $LHA(\Lambda)$ induces isomorphisms on the homology of $M(\Lambda)$. Consequently, Lemma \ref{lma:E-N} implies that Legendrian spheres with stable tame isomorphic DGAs give rise to symplectic manifolds with isomorphic symplectic homologies.
\begin{rmk}\label{rmk:parallel}
The left-right module $M(\Lambda)$ also has a geometric interpretation as follows. Consider the Legendrian sphere $\Lambda'$ obtained by pushing $\Lambda$ slightly off of itself using the Reeb flow. Then we introduce a Morse-Bott situation with one short Reeb chord connecting $\Lambda$ to $\Lambda'$ starting at each point of $\Lambda$. Perturbing out of this degenerate situation using a Morse function on $\Lambda$ with one maximum and one minimum, we find that the Reeb chords connecting $\Lambda$ to $\Lambda'$ are the following: one chord $\widehat c$ near each chord $c$ of $\Lambda$ and two short chords which we denote by $x$ and $y$ where $x$ corresponds to the minimum of the Morse function and $y$ to the maximum. 

Consider next the Lagrangian $n$-plane $L$ in the cobordism $W$ which looks like $\Lambda\times(-\infty,0]$ in the negative end of $W$. In order to continue the shift of $\Lambda$ along all of $L$ we extend the Morse function on $\Lambda$ to a Morse function on $L$ with exactly one maximum. Letting $L'$ denote the shift of $L$ determined by this function we find that $L\cap L'$ consists of one transverse double point which we take to lie near the critical point $\tau$ in the added handle and which we denote by $z$. 

Consider the module $FH(L,L')$ generated by Reeb chords starting on $\Lambda$ and ending on $\Lambda'$ and by the intersection point $z=L'\cap L$ which is a left $LHA(\Lambda)$-module
and right $LHA(\Lambda')$-module. This complex has a differential $d_{FH}$ which when acting on a mixed chord $\widehat c$ counts holomorphic disks with boundary on $\Lambda\cup\Lambda'$, positive puncture at $\widehat c$ and exactly one mixed negative puncture, which when acting on $z$ counts holomorphic disks in $W$ with one puncture at $z$ and exactly one mixed negative puncture, and which acts as $d_{LHA}$ on the coefficients. We call this complex the
{\em Floer homology of $L$ with coefficients in $LHA(\Lambda)$}. (Similar constructions were considered in \cite{E2} and the invariance of the homology this complex under Legendrian isotopies can be deduced from there.)

It is straightforward to check that
\[
d_{FH}(z)=y.
\]
 Indeed, the holomorphic strip corresponds to a flow line from the interior maximum to the maximum on the boundary and an action argument shows this strip is unique. Observing that there is a canonical isomorphism $LHA(\Lambda)\approx LHA(\Lambda')$, this implies that $FH(L,L')$ is quasi-isomorphic to $M(\Lambda)$. In particular, if the homology or $FH(L,L')$ vanishes so does the homology of $M(\Lambda)$. It follows that the homology of $M^{\rm cyc}(\Lambda)$ vanishes as well.
\end{rmk}

If there exists $c\in LHA(\Lambda)$ with $dc=1$ then the homology of $FH(L,L')$ vanishes: if $w$ is any cycle then $d_{FH}c w=w$. It thus follows from Remark \ref{rmk:parallel} that the homology of $M(\Lambda)$ vanishes and hence by Lemma \ref{lma:E-N} that $LH^{\Ho}(\Lambda)=0$.
The surgery exact triangle for symplectic homology then shows that if
a Legendrian sphere $\Lambda$ in the boundary of a Weinstein domain  has a Reeb chord $c$ with $dc=1$, then attaching a handle along it  does not change the symplectic homology.

\medskip

We use this result to construct exotic Weinstein symplectic structures on $T^{\ast} S^{n}$.\footnote{The first examples of exotic symplectic  Weinstein structures on $T^{\ast}S^n$ were constructed by  M. McLean in \cite{McL} and also  by M.~Maydanskiy and P.~Seidel in
\cite{MS}. We refer the reader to these papers and also to \cite{AbS2} as a guide on how this result can be also used for construction of other exotic symplectic objects.}  

Note that attaching a Weinstein handle  to $B^{2n}$ along a Legendrian sphere $\Lambda$  in $S^{2n-1}$ produces
a Weinstein manifold diffeomorphic to $T^{\ast}S^n$  and with the standard homotopy class of its almost complex structure, provided that the knot is topologically trivial, its Thurston-Bennequin invariant is equal  to $\tb(\Lambda_U)=(-1)^{n-1}$, and it is  Legendrian regularly homotopic to the Legendrian unknot $\Lambda_{U}$. The Legendrian spheres described below have these properties, while they have chords $c$ with $dc=1$, and  hence the corresponding symplectic homology vanishes. Thus this construction produces exotic Weinstein structures on $T^*S^n$.

\begin{figure}[ht!]
\labellist
\small\hair 2pt
\pinlabel $x_\theta$   at 673 19
\pinlabel $z$ at 380 300
\pinlabel $z$ at 19 498
\pinlabel $y$ at 416 498
\pinlabel $x$ at 310 375
\pinlabel $x$ at 735 375
\pinlabel $c$ at 621 432
\pinlabel $e_1$ at 608 407
\pinlabel $e_2$ at 543 405
\pinlabel $e_3$ at 455 400
\pinlabel $b_1$ at 563 473
\pinlabel $b_2$ at 489 412
\endlabellist
\centering
\includegraphics[width=\linewidth]{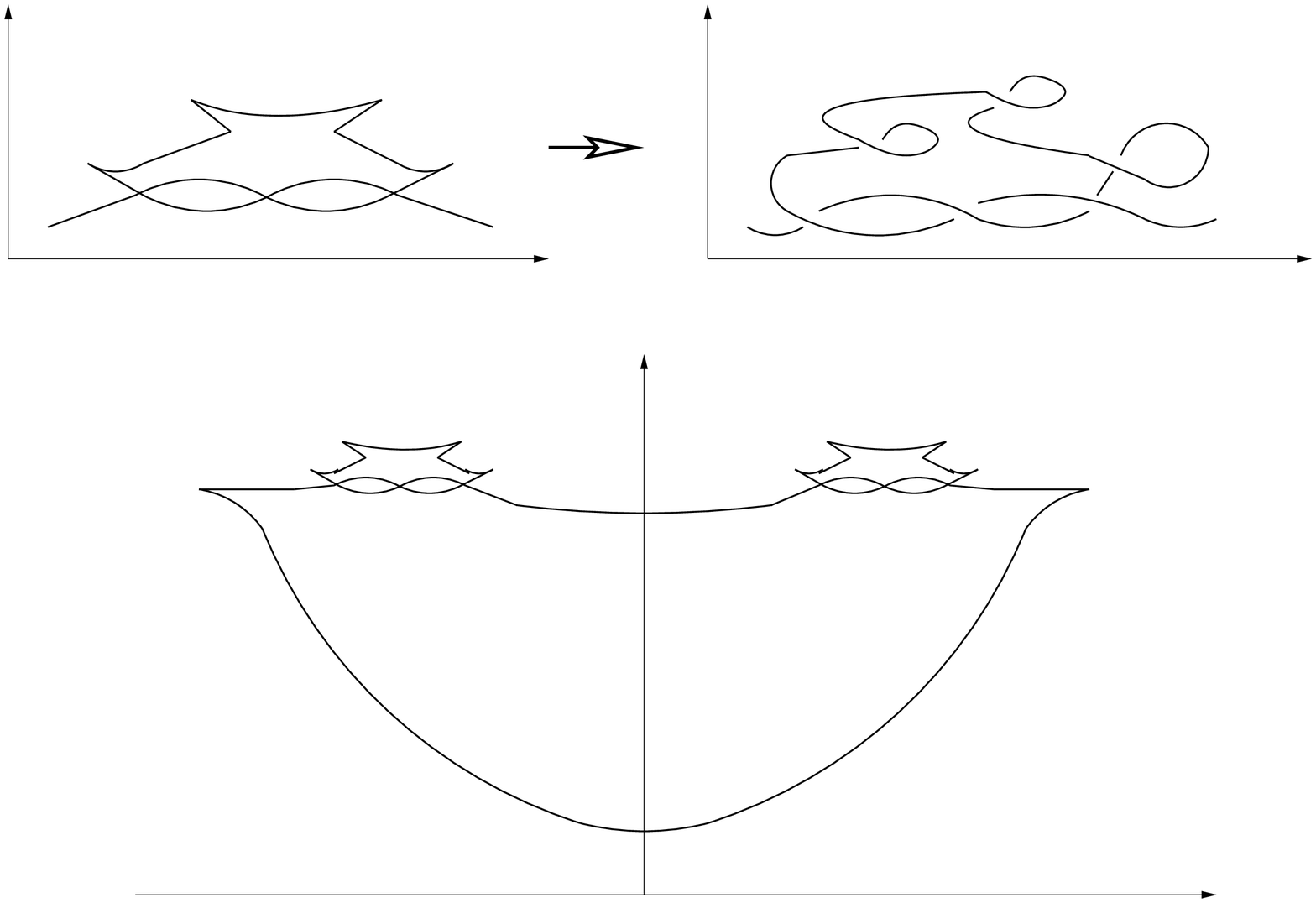}
\caption{The Legendrian sphere $\Lambda_T$. The lower picture shows the front of $\Lambda_{T}$ by showing its intersection with any $2$-plane spanned by a unit vector $\theta\in\R^{n-1}$ and a unit vector in the $z$-direction. The upper two pictures indicates how Reeb chords arise.}
\label{fig:thubenn}
\end{figure}

Consider the Legendrian sphere $\Lambda_{T}$ depicted in Figure \ref{fig:thubenn}. Note that each of the shown Reeb chords   gives rise to a Morse-Bott manifold $S^{n-2}$ of Reeb chords for $\Lambda_{T}$. After perturbation with a function with only one max and one min, each Morse-Bott sphere gives two Reeb chords. Thus, using self-evident notation, the chords of $\Lambda_{T}$ are $a$ (similar to the Reeb chord of $\Lambda_{U}$), $c^{\rm min},\,\,c^{\rm max}$, $e_j^{\rm min},\,\, e_j^{\rm max}$, $j=1,2,3$, and $b_k^{\rm min},\,\,b_{k}^{\rm max}$, $k=1,2$. Their gradings are as follows
\begin{align*}
|a|&=|c^{\rm max}|=|b_1^{\rm max}|=|b_2^{\rm max}|-1,\\
|e_j^{\rm max}|&= n-2,\\
|c^{\rm min}|&=|b_1^{\rm min}|=|b_2^{\rm min}|=1,\\
|e_j^{\rm min}|&=0.
\end{align*}
In particular,
\[
\sum_{c\in\CC(\Lambda_{T})}(-1)^{|c|}=(-1)^{|a|}
\]
and we conclude that the self linking number of $\Lambda_{T}$ agrees with that of $\Lambda_{U}$.

Let us verify that $\Lambda_{T}$ is Legendrian regularly homotopic to $\Lambda_{U}$. Figure \ref{fig:reghom} shows the fronts of a Legendrian regular homotopy connecting the knot used to construct $\Lambda_{T}$ to a straight line. Using the instances of this regular homotopy in the same way that the original knot was used to construct $\Lambda_{T}$ gives a Legendrian regular homotopy connecting $\Lambda_{T}$ to $\Lambda_{U}$.

\begin{figure}
\centering
\includegraphics[width=.4\linewidth]{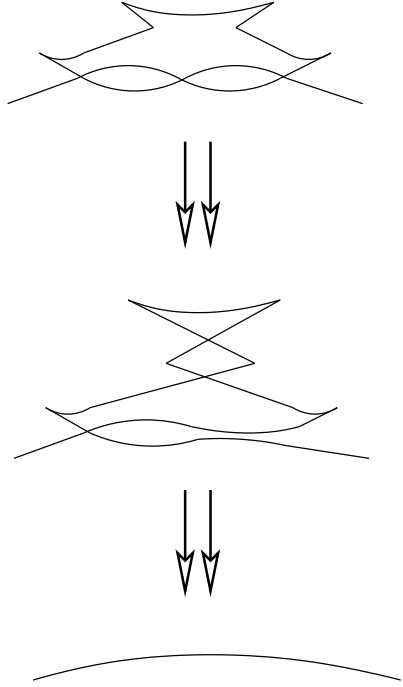}
\caption{A Legendrian regular homotopy. The first step involves crossing changes (corresponding to self tangencies of the front), the second step is an isotopy which consists of three consecutive first Reidemeister moves.}
\label{fig:reghom}
\end{figure}

Finally, we show that $d b_{k}^{\rm min}=1$, $k=1,2$. To this end, we use the characterization of rigid disks in terms of Morse flow trees, see \cite{E}. The flow trees in question start at $b_{k}^{\rm min}$, $k=1,2$. Since these Reeb chords correspond to local minima of the height between sheets no flow line goes out, $b^{\rm min}_{k}$ splits into two flow lines each of which ends in the cusp-edge, see Figure \ref{fig:stabil}.

\begin{figure}
\labellist
\small\hair 2pt
\pinlabel $x_\theta$   at 580 17
\pinlabel $z$ at 304 149
\pinlabel $b_1$ at 195 83
\pinlabel $b_2$ at 389 83
\endlabellist
\centering
\includegraphics[width=.8\linewidth]{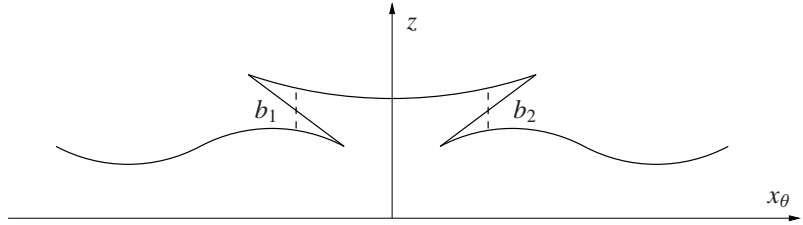}
\caption{Rigid flow trees giving $d b_{k}^{\rm min}=1$.}
\label{fig:stabil}
\end{figure}

We conclude that surgery on $\Lambda_{T}$ gives exotic symplectic structures on $T^{\ast}S^{n}$ for $n>3$.  For $n=2$, the Legendrian knot in Figure \ref{fig:thubenn} is knotted and the surgery does not give the same manifold as the surgery on the unknot.

\subsection{Surgery on the Chekanov  knots give different contact manifolds}
Consider the two Legendrian knots  (see \cite{Chek}) $\Lambda_a$ and $\Lambda_c$ in $\R^{3}=J^{1}(\R)$ depicted in Figure \ref{fig:chek}. As in Section \ref{sec:vanish} we write $1$ for the empty word of Reeb chords in $LHA(\Lambda_\ast)$, $\ast\in\{a,c\}$.

Their Legendrian homology DGAs are as follows: The algebra $LH(\Lambda_a)$ is generated by $a_1,\dots,a_9$ of gradings
\[
|a_j|=1,\, 1\le j\le 4;\quad |a_5|=2;\quad |a_6|=-2;\quad |a_j|=0,\, j\ge 7.
\]
The differential is (up to signs which are of no importance for our argument below) given by
\begin{align*}
d a_1 &= 1 + a_7 + a_7a_6a_5,\\
d a_2 &= 1 + a_9 + a_5a_6a_9,\\
d a_3 &= 1 + a_8a_7,\\
d a_4 &= 1 + a_8a_9,\\
d a_5&=\pa a_6=\pa a_7 =\pa a_8 =\pa a_9=0.
\end{align*}
The algebra $LH(\Lambda_c)$ is generated by $c_1,\dots, c_9$ of gradings
\[
|c_j|=1,\, 1\le j\le 4;\quad |c_j|=0,\, j\ge 5.
\]

\begin{figure}
\labellist
\small\hair 2pt
\pinlabel $a_2$ at 218 235
\pinlabel $a_1$ at 75 231
\pinlabel $a_4$ at 219 54
\pinlabel $a_3$ at 66 55
\pinlabel $a_5$ at 128 261
\pinlabel $a_6$ at 152 186
\pinlabel $a_9$ at 195 162
\pinlabel $a_8$ at 144 10
\pinlabel $a_7$ at 95 142
\pinlabel $c_2$ at 643 235
\pinlabel $c_1$ at 500 231
\pinlabel $c_4$ at 640 110
\pinlabel $c_3$ at 491 55
\pinlabel $c_5$ at 553 261
\pinlabel $c_6$ at 577 186
\pinlabel $c_9$ at 620 172
\pinlabel $c_8$ at 569 10
\pinlabel $c_7$ at 520 142
\endlabellist
\centering
\includegraphics[width=\linewidth]{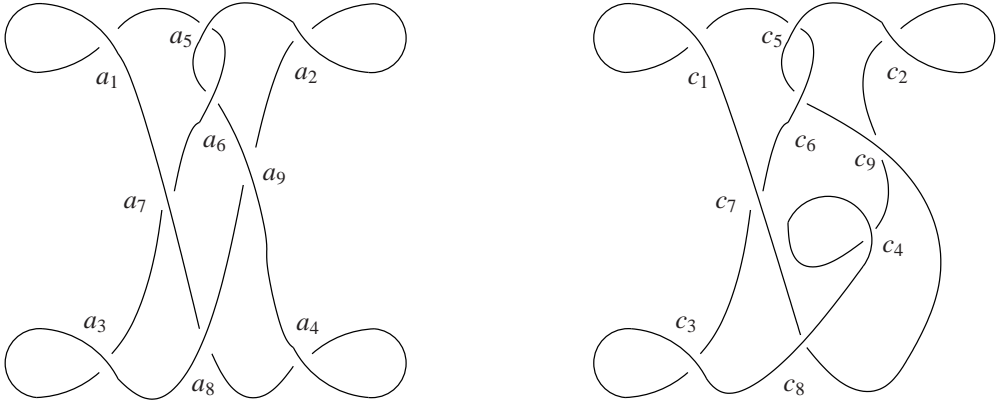}
\caption{Projections of the Chekanov knots into $T^{\ast}\R=\R^{2}$, $\Lambda_{a}$ left and $\Lambda_{c}$ right.}
\label{fig:chek}
\end{figure}

Consider now inclusions of $\Lambda_a$ and $\Lambda_c$ respectively in the sphere $S^{3}$ in a small coordinate chart as above. The complex which computes the linearized contact homology of the manifold which results from Legendrian surgery on $\Lambda_{\ast}$ is then $CH(B^{4})\oplus LH^{\rm cyc}(\Lambda_\ast)$. In particular, the gradings of the generators of $CH(B^{4})\oplus LH^{\rm cyc}(\Lambda_c)$ is non-negative. On the other hand, the generator $(a_6)$ of $LH^{\rm cyc}(\Lambda_a)$ has degree $-2$. We show that it survives in homology. To see this, we use an argument suggested by L.~Ng.  Consider the algebra $A=\Q\la a_6\ra$ and the map of differential graded algebras $\phi\colon LHA(\Lambda_a)\to A$ given by
\[
\phi(a_6)=a_6,\quad \phi(a_7)=\phi(a_9)=1,\quad\phi(a_8)=-1,\quad \phi(a_j)=0,\,\,j<6.
\]
This map induces a map $LH^{\rm cyc}(A)\to A^{\rm cyc}$ which is non-trivial on homology.

Recall from the remark on linearizations in Section~\ref{sec:lin-contact} that the collection of linearized Legendrian homologies for all augmentations is a contact invariant. The above argument shows that,
for some augmentation, the linearized contact homology of the result of $\Lambda_a$-surgery on $S^3$ has a generator in degree $-3$ and that, for any augmentation, all generators of the linearized contact homology of the result of $\Lambda_c$-surgery on $S^3$ have degree at least $-1$. Since these properties are independent of the choice of an augmentation, we conclude that the contact manifolds which arise, respectively, through surgery on $\Lambda_a$ and $\Lambda_c$ are different.

\section{Case of a  Lefschetz fibration}\label{sec:Lefschetz}
Symplectic Lefschetz presentation of a Weinstein manifold can be viewed as a special case of  Weinstein handlebody presentation.
\subsection{Lefschetz fibrations and Lefschetz type Legendrian surgeries}
\label{ssec:LefschetzFibr} 

Given two Liouville manifolds,
$(X_1,\lambda_1)$ and $(X_2,\lambda_2)$, the product Liouville structure is $(X_1\times X_2,\lambda_1\oplus\lambda_2)$. Consider $\R^{2}=\C$ with coordinates $u+iv$. The standard Liouville structure on $\R^{2}$ is given by the Liouville form $\frac12(u\,dv-v\,du)$.
The  {\em stabilization} $(X^{\st},\lambda^{\st})$ of a  $(2n-2)$-dimensional Liouville manifold $(X,\lambda)$ is the product of $X$ and $\R^{2}$ with the standard Liouville structure:
$$( X^\st, \lambda^\st)=(X,\lambda)\times \left(\R^2, \tfrac12(u\,dv -v\,du)\right).$$

Let  $Y=\p X$ be the ideal contact boundary  of $(X,\lambda)$.
We make a special choice of contact form on the ideal boundary $Y^\st $ of $ X^\st$ as follows.
Let $H:X\to [0,\infty)$ be a smooth exhausting function which is equal to $ 0$ on $\oX$, and which, in the end $E$, is of the form $h(s)$, where $s$ is the parameter of the flow $\Phi_{Z}$ of the Liouville vector field $Z$, and $h'(s)>0$. Consider the function $H^\st:X\times\R^2\to\R$ defined
by
$$
H^\st(x,w)=  H(x) +|w|^2,\quad x\in X,\,\, w=u+iv\in\C.
$$
Let $\oX^\st=\oX^\st_\eps :=\{(x,w)\colon H^\st(x,w)\leq \eps\}$.  Then  $\oX^\st $ is star-shaped with respect to the flow $\Phi_{Z^\st}$ of the Liouville vector field $Z^{\st}$ corresponding to the Liouville form $\lambda^\st$. Hence $X^\st$ is the completion of the Liouville domain $\oX^\st $ and we can identify $Y^\st$ with $\p \oX^\st$.
Note that the projection $X\times\R^2\to X$ maps $\oX^\st$ onto $\oX_\eps=\{H\leq\eps\}\subset X$.
 Consider the decomposition $Y^\st = Y^\st_1\cup Y_2^\st$, where $Y^\st_1=Y^\st \cap (\oX\times\C)$ and $Y^\st_2=Y^\st\setminus\Int Y_1^\st$. Note that $Y_1^\st$ splits as $\oX\times S^1$ and that the induced contact form  $\alpha=\lambda^\st|_{Y_1^\st}$ takes the form $\lambda+\eps dt$ where $t\in\R/{2\pi\Z}$.

We say that a collection of parameterized Legendrian spheres $\Lambda_1,\dots, \Lambda_k$ in $Y^\st_1=\oX\times S^{1}$ form a {\em basis for Lefschetz type Legendrian surgery} if they satisfy the following conditions.
\begin{itemize}
\item For $j=1,\dots,k$, the projection of $\Lambda_j$ to the $\oX$-factor is an embedded sphere.
\item For $j=1,\dots,k$, the projection of $\Lambda_j$ to the $S^{1}$-factor is contained in an arc $\Delta_j\subset S^{1}$ and $\Delta_{i}\cap \Delta_{j}=\varnothing$ if $i\ne j$.
For convenient notation, we choose numbering so that $\Delta_1,\dots,\Delta_{k}$ is the order in which the arcs appear if we start at $\Delta_{1}$ and traverse the circle in the counter-clockwise direction.
\end{itemize}

The connection of this construction with Lefschetz fibration presentations is as follows.
\begin{prp}\label{prop:Wein-Lef}
${\quad}$
\begin{enumerate}
\item   Let
$\Lambda_1,\dots,  \Lambda_k\subset Y^\st_1$
be a basis for Lefschetz type Legendrian surgery. Let $L_1,\dots, L_k$ denote the Lagrangian  spheres obtained by projecting
$\Lambda_1,\dots, \Lambda_k$  to $\oX\subset X$. Then the result $\wh X$ of attaching Weinstein handles along  $\Lambda_1,\dots, \Lambda_k$ is symplectomorphic to the Lefschetz fibration with fiber $X$ determined by the vanishing cycles $L_1,\dots, L_k$.
\item Conversely, if $L_1,\dots, L_k$ are exact Lagrangian spheres in $\oX$ which are the vanishing cycles of a Lefschetz fibration, then $L_1,\dots, L_k$ can be lifted to $Y^\st$ as a basis $\Lambda_1,\dots,\Lambda_k$ of Lefschetz type Legendrian surgery. Moreover, $\eps>0$ in the definition of $\oX^\st_\eps$ can be chosen arbitrarily small.
\end{enumerate}
\end{prp}
Let us comment here on the second statement. The Liouville form $\lambda$ is exact on the vanishing cycles $L_1,\dots, L_k$: $\lambda|_{L_j}=dF_j$, $j=1,\dots, k$.  If
\begin{equation}\label{eq:norm-F}
||F_j||=\max\, F_j-\min\, F_j<\frac{2\pi\eps}k
\end{equation}
then the vanishing cycle $L_j$ lifts to the Legendrian spheres
$$
\Lambda_j=\{(x,t)\in X\times\R/2\pi\Z\colon x\in L_j,\,\,t=-F_j(x)+C_j\},
$$
where $C_j$ is a constant, and for an appropriate choice of constants $C_1,\dots,C_k$, the Legendrian spheres $\Lambda_1,\dots,\Lambda_k$ form a basis of Lefschetz type Legendrian surgery.

As above, let $Z$ be the Liouville vector field on $X$ dual to the Liouville form $\lambda$ and let $\Phi^{s}_Z:X\to X$ denote the corresponding time $s$ flow. Then $(\Phi_Z^{-s})^*\lambda=e^{-s}\lambda$,
and hence, replacing the vanishing cycles
$L_j$ by $\Phi^{-s}_{Z}(L_j)$, which in view of the exactness condition  are Hamiltonian isotopic  to $L_j$, we can  make the $C^0$-norms $||F_j||$, $j=1,\dots, k$ sufficiently small so that the bound \eqref{eq:norm-F} holds for any given $\eps>0$.

\subsection{Legendrian homology algebra and Fukaya-Seidel categories}
\label{ssec:LefschetzLHA} 

Let $\Lambda_1,\dots, \Lambda_k\subset Y^\st_1$ be a basis for Lefschetz type Legendrian surgery and let $L_1,\dots, L_k\subset X$ be their Lagrangian projections. We assume that for $i\ne j$, $L_i$ and $L_j$ intersect transversely. We also can assume, without loss of generality, that the intervals $\Delta_1, \dots, \Delta_k$ which contain projections of the Legendrian spheres to the $S^1$-factor of $Y_{1}^{\st}$ are all contained in an arc $[0,\frac\pi2]\subset S^1$ (see (2) in Proposition \ref{prop:Wein-Lef} above).

\subsubsection{Generators of the Legendrian algebra}
The Reeb vector field of the form $\alpha=\lambda^\st|_{Y_1^\st}=\lambda+\eps dt$  is proportional to $\frac\p{\p t}$. Therefore, all Reeb orbits in $Y_{1}^{\st}$ are closed and we are in a Morse-Bott situation. J.~Sabloff  has developed, see \cite{Sabloff}, a formalism for describing Legendrian homology algebras in $S^1$-bundles which is appropriate for this situation and which we discuss below\footnote{The authors thank J.~Sabloff for adjusting his formalism specially for our purposes.}.

The generators of the Legendrian homology algebra are the Reeb chords of the Legendrian link $\Lambda=\Lambda_{1}\cup\dots\cup \Lambda_k$, which can be described as follows. For convenience we choose a trivialization of the canonical bundle on $X^{\st}$ which on $\oX\times\C$ is induced by the projection $\oX\times\C\to \oX$ from a trivialization of the canonical bundle of $\oX$. We first describe mixed chords with their endpoints on different components. Let $i<j$ and recall that we have $\Delta_j\subset [0,\tfrac{\pi}{2}]$ for $j=1,\dots,k$. Each intersection point $a\in L_i\cap L_j$ corresponds to infinitely many Reeb chords connecting
$\Lambda_i$ to $\Lambda_j$ and $\Lambda_j$ to $\Lambda_i$. We use the following notation:
\begin{equation}\label{eq:intchord1}
\raq_{a}^{(p)},\quad p=0,1,2,\dots,
\end{equation}
denotes the Reeb chord corresponding to $a$ which connects $\Lambda_i$ to $\Lambda_j$ and which passes $p$ times through the point $\pi\in S^1$, and similarly,
\begin{equation}\label{eq:intchord2}
\laq_{a}^{(p)},\quad  p=1,2,3,\dots,
\end{equation}
denotes the Reeb chord corresponding to $a$ which connects $\Lambda_j$ to $\Lambda_i$ and which passes $p$ times through the point $\pi\in S^1$. We will refer to the integer $p$ as the {\em multiplicity} of the chord. Increasing the multiplicity of a chord by $1$ increases its grading by $2$. It is convenient to organize all the chords corresponding to $a$ into two power series. We write
$$
\Raq_{a}=\sum_{p=0}^\infty \raq_{a}^{(p)}T^p\quad\text{and}\quad \Laq_{a}=\sum_{p=1}^\infty \laq_{a}^{(p)}T^p,
$$
where $T$ is a formal variable. If $|a|$ denotes the grading of the shortest Reeb chord corresponding to $a$ then the grading of these generators are as follows:
\[
|\raq_a^{(p)}|=|a|+2p,\quad |\laq_a^{(p)}|=(n-3)-|a|+2p.
\] 

Second, we describe pure chords which connect a component to itself. For each $\Lambda_i$, there is a sphere's worth of chords of given multiplicity connecting $\Lambda_i$ to itself. In the spirit of Morse-Bott formalism, see the sketched proof of Proposition \ref{prop:Lefschetz-differential} below for a more detailed description, we fix two points $m_{i\pm}$ on $L_i$ and associate with them two series of chords,
$$
\q_{i\pm}=\sum\limits_{p=1}^\infty q_{i\pm}^{(p)}\,T^{p}.
$$
Together with idempotents $e_1,\dots, e_k$ corresponding to the empty words of pure chords, the mixed and pure chords described above are all the generators of the algebra $LHA(\Lambda)$. The gradings of the pure generators are the following:
$$
|q^{(p)}_{i-}|=2p-1,\; |q^{(p)}_{i+}|=2p-1+(n-1),\;  |e_i|=0, \quad i=1,\dots,k.
$$

Let  $\II=\bigcup_{i<j}\,L_i\cap L_j$ be the set of all intersection points between distinct components  $L_1,\dots, L_k$. Write
\[
\II_{+\infty}:= \II\cup\{m_{1-},\dots,m_{k-}\},\qquad
\II_{-\infty}:= \II\cup\{m_{1+},\dots,m_{k+}\}.
\]
\begin{rmk} In our ``Morsification'' scheme explained below  in the sketched proof 
 of Proposition \ref{prop:Lefschetz-differential},  $m_{i-}$ and $m_{i+}$ will correspond, respectively  to the minimum and the maximum of an auxiliary Morse function on $L_i$.
 The sign $\pm$ in $\II_{\pm\infty}$ refers  to conditions at the positive and negative ends
 of the symplectization.  In the Morse-Bott formalism prescribing an image of a point at $+\infty$
means constraining it to a minimum, while at $-\infty$ to a maximum.
This is the reason for the apparent discrepancy in the above notation. 
\end{rmk}

\subsubsection{Admissible words and monomials}
We say that a word $w=a_0(L_{i_1})a_1(L_{i_2})a_2\dots a_r(L_{i_{r+1}})$, with $a_0\in\II_{+\infty}$ and $a_1,\dots, a_r\in\II_{-\infty}$ is {\em admissible} if the following conditions hold:
\begin{itemize}\item if $i_{r+1}\neq i_1$ then $a_0\in L_{i_1}\cap L_{i_{r+1}}$, otherwise
$a_0=m_{i_1-}$;
\item if  $i_j\neq i_{j+1}$, $j=1,\dots, r$, then $a_j\in L_{i_j}\cap L_{i_{j+1}}$, otherwise
$a_j=m_{i_j+}$.
\end{itemize}
We associate two monomials $\bQ_{w}$ and $\q_w$ to an admissible word $w$ as follows. Let $w=a_0(L_{i_1})a_1(L_{i_2})a_2\dots a_r(L_{i_{r+1}})$. The monomial $\bQ_w$ is the monomial in the variables
\[
\left\{\q_{i+},\Laq_{a},\Raq_{a}\right\}_{1\le i\le k;\,\, a\in\II}
\]
obtained from $w$ by removing $a_0$ and all $(L_{i_j})$
and replacing each $a_j$, $j>0$ by $\Raq_{a_j}$ if $i_{j+1}<i_j$, by $\Laq_{a_j}$ if $i_{j+1}>i_j$, or by $\q_{i_j+}$ if $i_{j+1}=i_j$. We write
\[
\bQ_{w}=\sum_{p=0}^{\infty} Q_{w}^{(p)}\,T^{p},
\]
and note that each $Q_{w}^{(p)}$ is a sum of monomials in the variables
\[
\left\{q_{i+}^{(p)},\laq_{a}^{(p)},\raq_{a}^{(p)}\right\}_{1\le i\le k;\,\, a\in\II;\,\, p\ge 0}.
\]
The monomial $\q_w$ is defined as
$$
\q_w=\begin{cases}\q_{j-},&\hbox{if}\;\;a_0=m_{j-}\in\{m_{1-},\dots,m_{k-}\}\cr
\Raq_a,&\hbox{if}\;\; a_0=a\in \II\;\;\hbox{and}\;\; i_{r+1}<i_1,\cr
\Laq_a,&\hbox{if}\;\; a_0=a\in \II\;\;\hbox{and} \;\;i_{r+1}>i_1.
\end{cases}
$$
We write
\[
\q_{w}=\sum_{p=0}^{\infty} q_{w}^{(p)}\,T^{p}.
\]

Our next goal is to define the moduli spaces which enter in the expression for the differential $d_{LHA}$ on $LHA(\Lambda)$. 
\subsubsection{Generalized holomorphic disks}\label{ssec:genholdisk}
We first define the notion of generalized holomorphic disks. Let $T$ be a {\it  fat rooted tree}, where the word ``fat'' means that the edges adjacent to a given vertex are cyclically ordered.
Let us  orient the  edges away from the root, so that    any vertex in $T$ which is not the root has one incoming edge and several outgoing edges.  If $v$ is a vertex of a fat rooted tree we write $o(v)$ for the number of outgoing edges at $v$. Fix Morse functions 
\begin{equation}\label{eq:chordfunctions}
f_i\colon L_i\to\R,\quad i=1,\dots,k,
\end{equation}
with exactly two critical points: maximum at $m_{i+}$ and minimum at $m_{i-}$. We write $f$ for the function on the disjoint union of the components of $L$ such that $f|_{L_j}=f_j$, $j=1,\dots,k$ and call it the {\em chord function}. 
 
A \emph{generalized holomorphic disk} with underlying fat rooted tree $T$ consists of the following data:
\begin{itemize}
\item[$(\mathrm{v})$] for each vertex $v\in T$ there is a holomorphic disk $u_v\colon (D,\pa D)\to (X,L)$ with boundary punctures of two types, \emph{external and internal}; 
 if $v$ is the root then $u_v$ has one distinguished external boundary puncture $\zeta_+$;
\item[$(\mathrm{e})$] for each edge $e$ connecting vertices $v_1$ and   $v_2$ there are   interior punctures $\zeta_{v_1,e}$ and $\zeta_{v_2,e}$ of the holomorphic disks $u_{v_1}$ and $u_{v_2}$ and a component $L_j$ of $L$ such that $u_{v_1}(\zeta_{v_1,e})$ and $u_{v_2}(\zeta_{v_2,e})$ both lie in $L_j$, and an orientation preserving diffeomorphism of the edge $e$ onto a non-zero length segment of a gradient trajectory of $f_j$ connecting $u_{v_1}(\zeta_{v_1,e})$ and  $u_{v_2}(\zeta_{v_2,e})$. Each internal puncture corresponds to the endpoint of an edge in $T$ and the internal punctures corresponding to different edges are disjoint.
\end{itemize}
We think of a generalized holomorphic map as a map of a collection of disks joined by line segments according to the underlying fat rooted tree and write $u\colon\Delta\to(X,L)$.
Note that there is a canonical ``counter-clockwise'' order of the external boundary punctures of any generalized holomorphic disk  $u\colon \Delta\to (X,L)$, where the distinguished puncture on the root holomorphic disk is chosen as the first one. See Figure \ref{fig:generalizeddisk}.

\begin{figure}
\labellist
\small\hair 2pt
\pinlabel $+$   at 157 344
\pinlabel $-\nabla f$ at 356 312
\pinlabel $-\nabla f$ at 338 225
\pinlabel $-\nabla f$ at 225 170
\pinlabel $-\nabla f$ at 470 101
\endlabellist
\centering
\includegraphics[width=.6\linewidth]{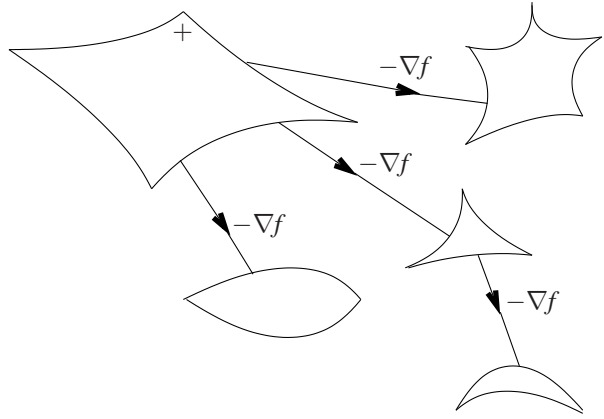}
\caption{A generalized holomorphic disk. The $+$ sign indicates the distinguished puncture in the root disk.}
\label{fig:generalizeddisk}
\end{figure}
\subsubsection{Moduli spaces of generalized holomorphic disks for the differential}\label{sec:mod-for-diff}
Consider an admissible word $w=a_0(L_{i_1})a_1(L_{i_2})a_2\dots a_r(L_{i_{r+1}})$. A maximal sub-word of $w$ of the form $m_{j+}(L_j)m_{j+}(L_j)\dots m_{j+}(L_{j})$ will be called an \emph{$m_{j+}$-block}. We will associate with $w$ generalized holomorphic disks with distinguished external (positive) puncture corresponding to $a_0$ and external (negative) punctures corresponding to $a_j$, $1\le j\le r$. The description of the differential requires moduli spaces where punctures corresponding to letters in the same block may collide. This is the reason for the jet conditions that appear below. 

Associate to $w$ the half open interval $(0,r]$.  Let $\Sigma(w)$ denote the set of all partitions 
\begin{equation}\label{eq:intervalpart}
(0,r]=\bigcup_{t=1}^{N(\sigma)}(s_{t-1},s_{t}]
\end{equation}
of $(0,r]$ into half open intervals with the following properties:
\begin{itemize}
\item all interval endpoints $s_t$, $t=0,\dots,N(\sigma)$ are integers;
\item if $a_j\in\II$ then $j=s_t$ for some $t$ and $s_{t-1}=s_t-1$ (i.e.~if $a_j$ corresponds to an intersection point then it has a corresponding interval in $\sigma$ of length equal to $1$);
\item if the length  of the interval $(s_{t-1},s_t]$ in $\sigma$ is $>1$ then there exists $j\in\{1,\dots,k\}$ such that for any two integers $s',s''\in(s_{t-1},s_t]$, the letters $a_{s'}$ and $a_{s''}$ belong to the same $m_{j+}$-block of $w$. 
\end{itemize}
If $\sigma\in\Sigma(w)$ then, as indicated above, we write $N(\sigma)$ for the number of intervals in $\sigma$.  We write $\ell(t)$ for the length of the interval $(s_{t-1},s_t]$, $1\le t\le N(\sigma)$.  
Thus $\ell(t)\geq 1$, $t=1,\dots, N(\sigma)$, and $\sum_{t=1}^{N(\sigma)}\ell(t)=r.$

Choose for each $j=1,\dots, k$ a germ at the point $m_{j+}$ of a smooth oriented $1$-dimensional submanifold $\lambda_j\subset L_j$. Given  $\sigma\in\Sigma(w)$ we define the (non-compactified) moduli space $\MM_{w,\sigma}$ of generalized 
holomorphic disks in $X$ with $N(\sigma)+1$ external boundary punctures denoted by $z_+,z_{s_1},\dots,z_{s_{N(\sigma)}}$,
\[
u\colon (\Delta,\p \Delta)\to  (X, L= \cup_{i=1}^k L_i),
\]
which satisfies the following conditions:
\begin{enumerate}
\item  $u$ maps the distinguished
puncture $z_+$ to the intersection point $a_0$ if $a_{0}\in\II$ or to the marked point $a_0=m_{i-}$ if $a_0\in\{m_{1-},\dots,m_{k-}\}$;
\item  if $a_{s_j}\in\II$, $j\in\{1,\dots,N(\sigma)\}$ then  the map $u$ maps the
puncture $z_{s_j}$ to the intersection point $a_{s_j}$; 
\item  $u$ maps the boundary arc bounded by $z_+$ and $z_{s_1}$ to $L_{i_{s_1}}$, that bounded by $z_{s_j}$ and $z_{s_{j+1}}$ to $L_{s_{j+1}}$, and that bounded by $z_{s_N(\sigma)}=z_r$ and $z_+$ to $L_{i_{r+1}}$;
\item  if $a_{s_j}=m_{i+}$ then $u(z_{s_j})=m_{i_+}$ and if, in addition $\ell(j)=s_j-s_{j-1}>1$ then the derivative $u'(z_{s_j})$ does not vanish and defines the given orientation of $\lambda_i$, and furthermore, the germ of the curve $u|_{\p D}$ is $(\ell(j)-1)$-tangent to $\lambda_i$ at $u(z_{s_j})=m_{i+}$ (i.e.~fixing a $1$-submanifold $\ol\lambda_i\subset L_i$ which represents the germ $\lambda_i$ at $m_{i+}$, the curve $u|_{\p D}$ is contained in $\ol\lambda_i$ up to order $\ell(j)-1$ at $m_{i+}$).
\end{enumerate}

 Write
\[
\M_w:=\mathop{\bigcup}\limits_{\sigma\in\Sigma(w)}\M_{w,\sigma}.
\]
The moduli space $\M_w$ can be compactified. The boundary strata include, besides the standard Deligne-Mumford-Gromov type strata corresponding to nodal disks also strata, where the derivative $u'(z_{s_j})$ vanishes at some of the punctures and strata where the length of gradient segments degenerates to $0$.  We denote the union of these strata by $\M_w^{\rm sing}$. Details about $\M_{w}^{\rm sing}$ will not be important below, what will be important is that generically one has $\dim(\M_w^{\rm sing})<\dim\M_w$, and hence if $\dim(\M_w)=0$ then, under our transversality assumption, we have $\M_w^{\rm sing}=\varnothing$ and $\M_w$ is compact.
 
Lemma \ref{lm:implicit-function} below, which follows from the implicit function theorem, clarifies the meaning of the moduli space $\M_w$.
 For each $i=1,\dots, k$  let us consider a parameterization $\wt\lambda_i:(-1,1)\to L_i$  of the curve $\lambda_i$
 such that $\wt\lambda_i(0)=m_{i+}$ and the vector $\wt\lambda_i'(0)$ defines the given orientation of $\lambda_i$.
 Let $N_i(w)$ be the maximal length of an $m_{i+}$-block in $w$. Take $N_i(w) $ points
 $0\leq t_i^1<\dots<t^{N_i(w)}_i<1$ and for a   $\delta\in(0,1)$ consider the points
 $m_{i+}^{j,\delta}=\wt\lambda_i(\delta t^j_i), j=1,\dots, N_i(w)$.


We say that the boundary punctures $\zeta_1,\dots,\zeta_r$ corresponding to an $m_{i+}$-block constitutes an {\em $m_{i+}$-cluster}. Define the moduli space  $\M_{w,\sigma}^{\delta}$ by replacing condition (4) above by
 \begin{description}
 \item{$(4^{\delta})$} If $\zeta_j$ is the $j$-th point in an $m_{i+}$-cluster then $u(\zeta_j)=m_{i+}^{j,\delta}$,
 \end{description}
see Figure \ref{fig:cluster}.
Note that the formal dimensions satisfies $\dim(\M^{\delta}_{w,\sigma})=\dim(\M_{w,\sigma})$ for any $\sigma\in\Sigma(w)$.
\begin{figure}
\labellist
\small\hair 2pt
\pinlabel $\zeta_1$  at 100 265
\pinlabel $\zeta_2$  at 121 265
\pinlabel $\zeta_3$  at 145 265
\pinlabel $\zeta_4$  at 372 265
\pinlabel $\zeta_5$  at 398 265
\pinlabel $m_{i+}^{1,\delta}$ at 158 30
\pinlabel $m_{i+}^{2,\delta}$ at 170 80
\pinlabel $m_{i+}^{3,\delta}$ at 230 50
\endlabellist
\centering
\includegraphics[width=.8\linewidth]{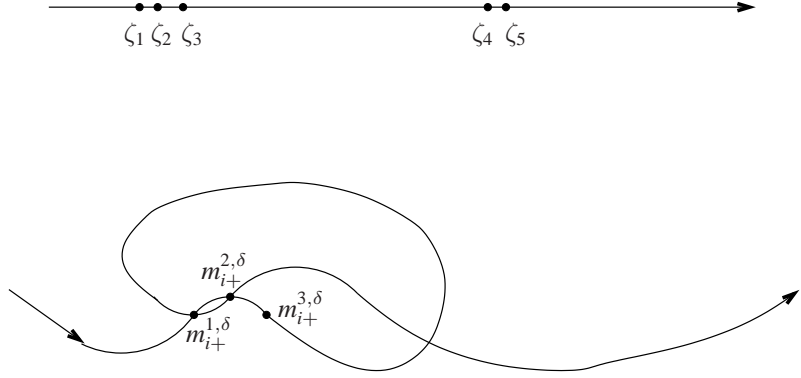}
\caption{The top picture shows the boundary punctures of two clusters in the source and the bottom picture shows the corresponding picture in the target.}
\label{fig:cluster}
\end{figure}

\begin{lma}\label{lm:implicit-function}
Suppose   that $\M_w$ is $0$-dimensional, regular and  $\M_w^{\rm sing}=\varnothing$.   Then for any $\sigma \in \Sigma(w)$ and for all sufficiently small $\delta>0$  the moduli space  $\M_{w,\sigma}^{\delta}$ 
is  regular and diffeomorphic  to $\M_{w,\sigma}$. 
\end{lma}

The following proposition describes the differential $d=d_{LHA}$. 


\begin{rmk}\label{r:n=2extra}
As mentioned above, the perturbation scheme for $\dim(L)=n-1$ used to compute the differential involves Morse theory on the $(n-1)$-dimensional sphere $S^{n-1}$. Here the lowest dimensional case $n=2$ is spacial since Morse flow lines have codimension $0$ in $S^{1}$. For this reason we us the following constructions when $n=2$. Consider $L=L_1\cup\dots L_k$. The orientation of $L_j$ together with the choice of the point $m_{j+}$ on $L_j$ induces an order on the intersection points $a\in\II$  which lie in $L_j$. If $a$ and $b$ are intersection points in $L_j$ we write $a>_j b$ if $a$ precedes $b$ in this order. We next introduce a global order on $\II$ as follows
for $a,b\in\II$: Let $a\in L_{i}\cap L_{j}$, $i<j$ and $b\in L_{s}\cap L_{t}$, $s<t$ then  
\[
a>b\quad\text{ if }\quad
\begin{cases}
&j>t\text{ or,}\\
&j=t\text{ and }i>s,\text{ or,}\\
&j=t,\;i=s,\text{ and }a>_jb.
\end{cases}
\]
Furthermore we will use the following subsets of $\II$: for $1\le i\le k$, let 
\begin{align*}
\II_{<i}&=\{a\in \II\colon a\in L_i\cap L_j,\; j<i\},\\
\II_{>i}&=\{a\in \II\colon a\in L_i\cap L_j,\; j>i\},\\
\II_i&=\II_{<i}\cup\II_{>i}.
\end{align*}
\end{rmk}

\begin{prp}\label{prop:Lefschetz-differential}
The differential $d$ can be written as $d=d_{\const}+d_{{\rm MB}}+d_{{\rm h}}$, where the summands satisfy the following.
\begin{itemize}
\item[${\rm (i)}$] If $a\in\II$ then
\[
d_{\const}\Raq_{a}=d_{\const}\Laq_{a}=0.
\]
For $1\le i\le k$,
\[
d_\const q_{i-}^{(1)}=e_i,\quad d_\const q_{i+}^{(1)}= d_{\const} q_{i\pm}^{(p)}=0, \text{ for all } p>1.
\]

\item[${\rm (ii)}$] For $1\le i\le k$,
\[
d_{{\rm h}}\q_{i+}=0,
\]
and on other generators $d_{h}$ is determined by the moduli spaces $\MM_{w}$, $w$ admissible, with $\dim(\MM_{w})=0$ as follows: If the moduli space $\MM_w$ is $0$-dimensional then every element in it contributes the term $\pm \bQ_w$ to $d_{{\rm h}}\q_w$.

\item[${\rm (iii)}$]  For $1\le i\le k$, let  us denote

\begin{enumerate}\item Suppose $n>2$. Then
\begin{align*}
d_{{\rm MB}}\q_{i-}&=\q_{i-}\q_{i-},\cr
d_{{\rm MB}}\q_{i+}&=
\q_{i-}\q_{i+}+(-1)^{n-1}\q_{i+}\q_{i-}
+\sum_{a\in\II_{<i}}\Raq_a\Laq_a 
+\sum_{a\in\II_{>i}}\Laq_a\Raq_a.
\end{align*}
 If $a\in\II$ is an intersection point in $L_i\cap L_j$, $i<j$,  then  
\begin{align*}
d_{{\rm MB}}(\Raq_a)&=\q_{j-}{\Raq}_a +(-1)^{|a|-1}\Raq_a\q_{i-},\\
d_{{\rm MB}}(\Laq_a)&=\q_{i-}{\Laq}_a +(-1)^{(n-2)-|a|}\Laq_a\q_{j-},
\end{align*}

\item Suppose $n=2$. Then
\begin{align*}
d_{{\rm MB}}\q_{i-}&=\q_{i-}\q_{i-},\cr
d_{{\rm MB}}\q_{i+}&=
 \q_{i-}\q_{i+}-\q_{i+}\q_{i-}
+\sum_{a\in\II_{<i}}\Raq_a\Laq_a 
+\sum_{a\in\II_{>i}}\Laq_a\Raq_a \cr
&+\sum_{a\in \II_{<i}}\q_{i+}\Raq_a\Laq_a
+\sum_{a\in \II_{>i}}\q_{i+}\Laq_a\Raq_a,
\end{align*}
 If $a\in\II$ is an intersection point in $L_i\cap L_j$, $i<j$,  then
 \begin{align*}
d_{{\rm MB}}\Raq_a&=\q_{j-}{\Raq}_a +(-1)^{|a|-1}\Raq_a\q_{i-}\\
&+(-1)^{|a|-1}\left(\sum_{b\in\II_{<i},\;b<a}\Raq_a\Raq_b\Laq_b + \sum_{b\in\II_{>i},\;b<a}\Raq_a\Laq_b\Raq_b\right)\\
&-\left(\sum_{c\in\II_{<j},\;c<a}\Raq_c\Laq_c\Raq_a +\sum_{c\in\II_{>j},\;c<a}\Laq_c\Raq_c\Raq_a\right),\\
d_{{\rm MB}}(\Laq_a)&=\q_{i-}{\Laq}_a +(-1)^{|a|}\Laq_a\q_{j-}\\
&-\left(\sum_{b\in\II_{<i},\;b<a}\Raq_b\Laq_b\Laq_a + \sum_{b\in\II_{>i},\;b<a}\Laq_b\Raq_b\Laq_a\right)\\
&+(-1)^{|a|}\left(\sum_{c\in\II_{<j},\;c<a}\Laq_a\Raq_c\Laq_c +\sum_{c\in\II_{>j},\;c<a}\Laq_a\Laq_c\Raq_c\right).
\end{align*}
\end{enumerate}
\end{itemize}
\end{prp}
Here $d_\const$ is the contribution of disks which are not entirely contained in the symplectization of the $S^1$-invariant part $Y^{\st}_1$,  $d_{{\rm MB}}$ is responsible  for the contribution
of gradient trees of auxiliary Morse functions on the spheres $L_j$  which are used for the ``Morsification'' of the Morse-Bott picture, and $d_{{\rm h}}$ counts the contribution of generalized holomorphic disks in 
$Y^{\st}_1$.
We sketch a proof of  Proposition \ref{prop:Lefschetz-differential} in the next section.

\subsection
{Sketch of proof of Proposition \ref{prop:Lefschetz-differential}}\label{sec:Lefschetz-sketch}

To compute the  differential one needs to  count holomorphic disks in the symplectization $Y^{\rm st}\times\R$. We first consider disks that are entirely contained in $Y^{\rm st}_1\times\R\subset Y^{\rm st}\times\R$.
Recall that each intersection point $a\in \II$ corresponds to two series $\raq_{a}^{(p)}$, $p=0,1,2,\dots$ and  $ \laq_{a}^{(p)}$, $p=1,2,3,\dots$ of mixed Reeb chords between distinct components of $\Lambda$, see \eqref{eq:intchord1} and \eqref{eq:intchord2}, and that each component $L_i$ of the Lagrangian $L$ contributes with Morse-Bott families  $\{q_{x,i}^{(p)}\}_{x\in L_i}$, $i=1,\dots, k$, $p=1,2,\dots$ of pure chords connecting a component of $\Lambda$ to itself. In fact, we are in a ``double Morse-Bott'' situation since $\ol X$ is a Bott-family of periodic orbits.

Note that in this Bott degenerate situation, holomorphic disks in $Y_1^{\rm st}\times\R\approx \ol X\times\C^{\ast}$ with boundary on $\Lambda\times\R$ project to holomorphic disks in $\ol X$ with boundary on $L$. Conversely, a holomorphic disk in $\ol X$ with boundary on $L$ and a meromorphic section over it determines a holomorphic disk in $Y^{\rm st}_1\times\R$ with boundary on $\Lambda\times\R$ and with punctures mapping to Reeb chords and Reeb orbits as determined by the meromorphic section. In line with the Morse-Bott philosophy one expects that holomorphic disks after small perturbation can be described in terms of holomorphic disks in the Bott degenerate case in combination with Morse theory on the chord and orbit spaces. As we shall see below, this is indeed the case, and rigid holomorphic disks in $(Y^{\st}_1\times\R,\Lambda\times\R)$  can be described in terms of  rigid generalized holomorphic disks in $(\ol X,L)$. 

In order to establish this relation between rigid disks in the symplectization and rigid generalized holomorphic disks in the base of the fibration, we begin with a brief description of a perturbation scheme that achieves transversality for $\MM_{w,\sigma}$. We consider the source of a generalized holomorphic disk as a collection of disks with boundary punctures connected by finite length edges according to the underlying fat rooted tree. The generalized disk maps the disks holomorphically to $(\ol X,L)$   and the edges to gradient flow segments. The space of generalized disk sources has a natural compactification consisting of broken configurations. Here there are two types of breaking, either two disks can be joined at a puncture mapping to an intersection point in  $L_i\cap L_j$, $i\ne j$ or the length of some edge can go to infinity. We call this space {\em extended Deligne-Mumford space}. 

We next consider transversality for generalized disks. Note that any holomorphic component of  generalized holomorphic disk must have at least two punctures. Thus disks with $\le 3$ punctures cannot be multiply covered and it is possible to achieve transversality by perturbing the almost complex structure. In order to make the whole moduli space regular we use a slightly modified version of Gromov's transversality scheme \cite{Gromov} and consider an inhomogeneous $\overline{\partial}$-equation, where the right-hand side depends not only on points of the source disk, but also of its domain viewed as a point of the extended Deligne-Mumford space. The perturbation required here is constructed inductively over strata of the extended Deligne-Mumford space. Standard arguments show that we can achieve transversality for the holomorphic parts in the generalized disk using such perturbations that are supported in the interiors of the disks in the source, also including transversality with respect to jet conditions.

 We also require transversality with respect to the Morse flow as follows. Consider the unstable sphere $S_i$ of the maximum $m_{i+}$ of $f_i$ and note that the Morse flow determines a projection $\pr_i\colon L_{i}-\{m_{i+},m_{i-}\}\to S_i$. We require that the following transversality condition holds: the evaluation maps at additional boundary punctures on any moduli space of holomorphic disks forming  a generalized tree with additional internal boundary punctures are transverse to $\{m_{i+},m_{i-}\}$ and in general position with respect to the projections $\pr_i$ outside neighborhoods of the inverse images under the evaluation maps of small neighborhoods of $m_{i+}$ and $m_{i-}$. We also require transversality for the products of evaluation maps.  In this set up, $\MM_{w,\sigma}$ is related to intersection points of the evaluation maps of the disks components of a generalized holomorphic disk composed with the projections $\pr_i$, and it is possible to include transversality of this Morse data in the Gromov scheme.   

Note that in general some of the disks forming a generalized holomorphic disk could be {\em small}, i.e.~perturbations of constant holomorphic disks.
However, we show that in the situation under study here it is possible to organize the perturbation scheme in such a way that generalized disks with constant components and at least one nonconstant component do not contribute to the differential.  Indeed, let us note that each small disk must have at least three internal punctures. By choosing a sequence of perturbations which converge to $0$ on small disks we find gradient lines that intersect more than two non-small disks, and a dimension count shows that configurations with $r$ constant disks $c_1,\dots,c_r$, have codimension $d=\sum_{j=1}^{r}(i(c_j)-2)$, where $i(c_j)$ is the number of internal punctures in $c_j$. In particular, since $d\ge 1$ for configurations with at least one constant disk, our transversality assumptions implies that configurations with constant disks and at least one nonconstant disk are impossible in $0$-dimensional moduli spaces.
          
Consider the dynamics of the Reeb vector field of the contact form $\alpha^\st=\lambda^\st|_{Y^\st}$ on $Y^\st$. The part $Y^\st_1=\oX\times S^1$ is foliated by closed Reeb trajectories in the class of the factor $S^1$. The action of these orbits is equal to $2\pi\eps$. The action of all other orbits is bounded below by the action of the minimal closed orbit of the contact form $\lambda|_Y$. We refer to the former orbits as {\em short} and to the latter ones as {\em long}. Thus the  moduli space of  short orbits is diffeomorphic to $\oX$.      
To deal with the double Morse-Bott situation, we choose in addition to the chord functions $f_i : L_i \to \R$  as in \eqref{eq:chordfunctions}, an auxiliary Morse function $\phi\colon\oX\to\R$ which is constant on the boundary $\p \oX$ and attains its minimum there. We call $\phi$ the {\em orbit} function. The differential of  the algebra $LHA(\Lambda)$ can then be computed using moduli spaces of {\em generalized
holomorphic curves} in the sense of \cite{Bourgeois-thesis}. (We will call them \emph{Bott curves} here in order not to confuse them with generalized holomorphic disks.)  
 According to \cite{Bourgeois-thesis} (see also \cite{FOOO,FO,E}), Bott curves are multilevel objects,  which have a certain number of genuine holomorphic levels, as well as cylinders swept by orbits of $R$ which project to gradient arcs of the orbit function $\phi$ and strips swept by flow segments which project to gradient arcs of the chord function $f$. We call the latter {\em cylindrical} and {\em striplike} horizontal levels, respectively.  

We first argue that if a Bott disk $D$ which contributes to the differential has a non-empty generalized holomorphic part, then it does not contain horizontal cylindrical parts over flow lines of $\phi$. Indeed, suppose that a holomorphic level  $C$ is connected to a horizontal cylindrical part at at least one of its negative ends. Then the multiplicity of the chord at the positive end of $C$ is strictly bigger than the sum of the multiplicities of chords at the negative end of $C$. Hence, one can decrease the multiplicity of the chord at the positive end and still find a lift to the symplectization. But increasing the multiplicity of a chord increases its grading. It follows that the dimension of the moduli space containing the Bott disk $D$ is $>0$ and thus $D$ does not contribute to the differential. We also claim that given any integer $N>0$ we can arrange that any holomorphic part of a Bott disk entering into the differential of a chord of multiplicity $\leq N$  projects to a generalized holomorphic disk which does not intersect $Y=\pa\oX$. Indeed, the symplectic area of the projection of the holomorphic part of such a generalized curve is bounded above by the action of the chord at its positive puncture, which is $\leq2\pi\eps N$. According to Proposition \ref{prop:Wein-Lef} the constant $\eps$ can be made arbitrarily small, and hence the claim follows from a monotonicity argument. We conclude that the differential can be described in terms of generalized holomorphic disks.


We conclude that any disk which contributes to the differential limits to a Bott disk which project to a generalized holomorphic disk. Conversely, any rigid generalized holomorphic disk in $\M_{w,\sigma}$ lifts to a unique Bott disk  in the symplectization of $Y_1^\st$
if we assign appropriate multiplicities to chords  which project to images of boundary punctures of the generalized disk. We conclude that the contribution ${\rm (ii)}$ to the differential is as claimed.

We show below that the descriptions of the other summands of the differential are as claimed in ${\rm (i)}$ and ${\rm (iii)}$.  

We first study Bott curves without non-constant holomorphic part and without horizontal cylindrical parts over flow lines of $\phi$. Such curves are disks projecting to gradient trees and give rise to the summand $d_{{\rm MB}}$. 
The disks giving
  \[
d_{{\rm MB}}\Raq_a=\Raq_a\q_{i-}+\q_{j-}\Raq_a; \;\;d_{{\rm MB}}\Laq_a=\Laq_a\q_{j-}+\q_{i-}\Laq_a.
\]
correspond to gradient trajectories connecting the intersection point $a$ in $L_i \cap L_j$ to the 
minimum of the chord function in $L_i$ or in $L_j$. 
The disk giving
\[
d_{{\rm MB}}\q_{i-}=\q_{i-}\q_{i-}
\]
corresponds to a small disk with three punctures, including a positive puncture at the minimum of the chord function $f_i$ where the disk is asymptotic to $q_{i-}^{(p)}$. To the negative punctures are attached two flow lines which end at the minimum of $f_i$ where the disk is asymptotic to $q_{i-}^{(s)}$ respectively 
$q_{i-}^{(t)}$, $p=s+t$. 
The terms $\q_{i-}\q_{i+}+(-1)^{n-1}\q_{i+}\q_{i-}$ in  the expression for $d_{{\rm MB}}\q_{i+}$ correspond to 
small disks with three punctures, including a negative puncture at the maximum of the chord function $f_i$ where the disk is asymptotic to $q_{i+}^{(s)}$. A flow line emanating from the maximum of $f_i$ is attached
to the positive puncture of the small disk, where the disk is asymptotic to $q_{i+}^{(p)}$. To the second
negative puncture is attached a flow line which ends at the minimum of $f_i$ where the disk is asymptotic to $q_{i-}^{(t)}$, $p=s+t$.  
The terms   $\sum_{a\in\II_{<i}}\Raq_a\Laq_a$ and 
$\sum_{a\in\II_{>i}}\Laq_a\Raq_a$    in  the expression for $d_{{\rm MB}}\q_{i+}$
correspond to gradient trajectories connecting the maximum of the chord function with intersection points of $L_i$ with other components.
 Using grading arguments   one can check that these are the only rigid configurations if $n>2$.
  
   If $n=2$  there are several more  gradient tree configurations contributing to $d_{\rm MB}$.
The easiest way to account for all of them is probably via an explicit perturbation scheme, as in \cite{Sabloff}.  The additional cubic terms correspond to narrow quadrangles formed by the mesh of perturbed $1$-dimensional Lagrangian submanifolds. The exact expression for the differential, for example the order of the intersection points, is related to a specific choice of perturbation. To get the differential as stated it is important to choose chord functions so that the minima $m_{i-}$ lies very close to the maxima $m_{i+}$, $i=1,\dots,k$, and so that the following hold: if $a\in L_s\cap L_t$ then $f_i(m_{i+})> f_s(a)+f_t(a)> f_i(m_{i-})$ for all $i$, and if $b\in\II_i$ is such that $b\le_i c$ for all $c\in\II_i$ then $f_i(b)>f_s(a)$ for all $a\in\II_{s}$, $s<i$ such that $a\notin\II_k$ for some $k\ge i$.

   It remains to show that all other contributions are covered by $({\rm i})$.
Since rigid holomorphic curves the projection of which contains a non-constant polygon does not contain any horizontal cylindrical parts over gradient flow lines of $\phi$, any Bott curve which contributes to the differential and which
contains such a flow line must consist of that flow line only, i.e.~it must project to a  gradient trajectory of the orbit function $\phi$. Our assumption that $\phi$ is constant and achieves its minimum on $Y=\p\oX$ implies that through every point of $L$ there is a unique (negative) gradient trajectory of $\phi$ going to the boundary. Disks projecting to gradient trajectories are generalized holomorphic curve which consists of the following parts:
\begin{itemize}
\item A  holomorphic disk with a boundary puncture and an interior puncture which is mapped   
onto the fiber of the projection $Y^\st_1\times\R\to X$.
At the interior puncture it is asymptotic at $-\infty$  to a small (possibly multiply-covered) orbit, and at the boundary puncture  to a pure chord $q^{(p)}_{i\pm}$. 
\item  A horizontal part 
which projects to the gradient arc  of $\phi$ in $\oX$ that connects $m_{i-}$ to a (multiple of a) short orbit $\gamma_y$  over a  point $y\in Y=\p X$.
\item A plane in $Y^\st_2\times\R$ asymptotic at $+\infty$ to the orbit $\gamma_y$.
\end{itemize}
Such a curve contributes only $\pm e_i$ to the differential. Hence, for grading reasons, we have $|q^{(p)}_{i\pm}|=1$, which implies that $q^{(p)}_{i\pm}=q^{(1)}_{i-}$.

To establish $(\mathrm{i})$, we prove that for $q^{(1)}_{i-}$ such Bott disks indeed exist and that the algebraic count of them equals $\pm 1$. To this end, recall, see \cite{SFT}, that  given   a general contact manifold $Y$ for which one can define the    linearized  contact homology $C\bbH(Y)$,
    there is an invariantly defined element $\bh_1=\bh_1(Y)\in C\bbH(Y)$ which in the 
    corresponding  contact homology complex 
    $CH(Y)=(\K\langle\PP_\good(Y)\rangle, d_{CH})$  is represented by  
$$ \sum\limits_{\gamma\in \PP_\good, |\gamma|=n+1}n_{\gamma}\gamma,$$
where $n_\gamma$ is an algebraic count of components of the $1$-dimensional moduli space $\MM^Y(\gamma,p)$
of augmented holomorphic planes in the symplectization of $Y$ asymptotic to $\gamma$ and having one marked point mapped to $\R\times p\subset \R \times Y$  for a generic point $p\in Y$.  To simplify the language we will use the term ``holomorphic curve'' also  for  the projection of a  holomorphic curve in $\R\times Y$ to $Y$, and hence  equivalently  we can say that $n_\gamma$ counts the algebraic  number of holomorphic planes in $Y$ asymptotic to $\gamma$ and passing through a fixed point $p$.

Denote the critical points of the orbit function $\phi$ by $a_0,\dots, a_K$, where $a_0$ is its unique maximum. Choose an almost complex structure $J_0$  on $X$ adjusted to the Liouville form $\lambda$ and an almost complex structure $J$ on the symplectization 
 of $Y^\st$ which is adjusted to the contact form $\alpha^\st$ and which is such that the projection $\R\times Y^\st_1\to \oX$ is $(J,J_0)$-holomorphic.
 
 We compute the invariant $\bh_1(Y^\st)$.
 According to \cite{MLYau} and  \cite{Bourgeois-thesis} (see also \cite{SFT}), and taking into account an   argument  from Section  \ref{sec:Reeb-dynamics} to discard the contribution of long orbits,  we come to the following algorithm for computing  the linearized contact homology $C\bbH(Y^\st)$.  Associate with each critical point
  $a_j$ a sequence of variables  $\gamma^k_{a_j}$, $k=1,2,\dots$ corresponding to multiples of the orbit over the point $a_j$, and assign to  $\gamma^k_{a_j}$ the grading 
   $|\gamma^k_{a_j}|=\ind(a_j)-n+2k+1$, where $\ind(a_j)$ is the Morse index of the critical point $a_j$. In particular, $|\gamma^1_{a_0}|=n+1$.
Let $d_{\Morse}(a_j)=\sum\limits_i m_j^ia_i$ be the Morse differential on the Morse complex for the relative homology $H_*(\oX,Y)$ generated by critical points $a_j,\; j=0,\dots, K$.
Then the complex   $$C=CH(Y^\st):=(\K\langle\{\gamma_{a_j^k},\; j=0,\dots, K, k=1,\dots\}\rangle, d_M),$$ where $d_M(\gamma_{a_j}^k)=\sum\limits_i m_j^i\gamma_{a_i}^k$, computes the homology $C\bbH(Y^\st)$.
   \begin{lma}\label{lm:h1-invariant}
   The invariant $\bh_1(Y^\st)\in C\bbH(Y^\st)$ is represented  by the element $\pm\gamma_{a_0}^1\in C=CH(Y^\st)$. 
   \end{lma}
\begin{pf}[Sketch of proof]
Our proof uses the fact that since $X$ is a Weinstein manifold, $Y^\st$ bounds the Weinstein manifold $W=X\times\C$ with $SH(W)=0$ (in fact $W$ is even subcritical).

 Note that the index formula implies that  $\bh_1\in C\bbH_{n+1}(Y^\st)$, and hence it is represented by a linear combination $h_1\in C$ of classes $\gamma_{a_0}^1$ and  $\gamma_{a_j}^{k}$, where $\ind(a_j)=2n-2k$. We claim that the  coefficients with $\gamma_{a_j}^{k}$ is $0$ unless $j=0, k=1$. Indeed any Bott curve which is a 
plane asymptotic to $\gamma^k_{a_j}$  must contain a flow line part that converges to $a_j$. 
 Hence, if $j\neq 0$ and one chooses the point $p$ not over the stable manifold of $a_j$ then there  will be no generalized holomorphic curves converging to $\gamma^k_{a_j}$. Thus, $h_1=N\gamma^1_{a_0}$ for some integer $N$.

It remains to prove that the coefficient $N$  with $\gamma_{a_0}^1$ is equal to $\pm 1$. 
Here we use $SH(W)=0$ and the Morse-Bott description of symplectic homology. The minimum of the Morse function on $W$ is a cycle in $SH(W)$ which must then be a boundary. Thus there is a holomorphic sphere with positive puncture at $\gamma_{a_0}^{1}$ through every point in $W$ and in particular through every point in $p\times\R$.
\end{pf}

We can now prove that there is algebraically one  holomorphic plane  in $Y^\st_2\times\R$
  asymptotic to a small  simple  orbit  $\gamma_y$ over a fixed point  $y\in Y$.
  Note that generically the gradient trajectory from $y$ ends at the maximum point 
  $a_0$. Let us choose a point $p\in Y^\st_1$ which projects to a point on the arc $\delta$ of this trajectory connecting $a_0$ and $y$.  
  Then the moduli space $\MM^Y(\gamma_{a_0}^1,p)$ consists of generalized 2-level curves with a flow line part which projects to $\delta$ and a genuine holomorphic plane in the symplectization of $Y_2^\st$ asymptotic to $\gamma_y$. But
  according to  Lemma \ref{lm:h1-invariant} the algebraic count of such planes is $\pm 1$, and hence the  algebraic count 
  of planes   in $Y^\st_2\times\R$ asymptotic  to   $\gamma_y$ is $\pm 1$ as well.

This implies that ${\rm (i)}$ gives a description of the remaining part of the differential and concludes the sketch of the proof of Proposition \ref{prop:Lefschetz-differential}.

\bigskip

In the $4$-dimensional case the above description of the differential algebra $LHA(\Lambda)$, together with Theorem \ref{thm:SH} and Corollary \ref{cor:SH-subcrit} provides a purely combinatorial recipe for computing the complex $SH(X)$ for any  symplectic Lefschetz fibration.

\begin{rmk}\label{rmk:killing-handle}
${\quad}$
\begin{enumerate}
\item  An argument, similar to one used in the above proof  also shows
that {\it if there exists an embedded Lagrangian disk $\Delta\subset \oX$ with Legendrian boundary $\Gamma\subset Y$, and  which transversely intersects $L_i$  in exactly one point and does not intersect $L_j$ for $j\neq i$, then $dq^{(1)}_{i-}=e_i$.}
\item More generally, {\it
let $W$ be a $2n$-dimensional Liouville manifold with cylindrical end, and $\overline {W}$ the corresponding compact Liouville domain. Let $C$  be  the coisotropic unstable manifold of a zero  of index $n-1$ of the Liouville field $Z$ on $W$.
Suppose that $C$ intersects $V=\p\overline{W}$ along an $n$-sphere $S\subset V$.
Let  $\Lambda\subset  V$ be a Legendrian sphere which intersects $S$ transversely at one point.
Then for an appropriate choice  of a contact form on $Y$  and an almost complex structure, there exists a  Reeb chord $c$ of $\Lambda$, such that $dc= 1$ in the algebra $LHA(\Lambda)$.}
\end{enumerate}
\end{rmk}

Indeed, to verify Remark \ref{rmk:killing-handle} $(1)$ consider a small neighborhood $U$ of the Lagrangian disk $\Delta$. Then $U$ intersects $L_i$ in a disk $O_i$ and we can view $O_i$ and $\Delta$ as the core and co-core disks, respectively, of a Lagrangian handle attachment yielding $\oX$. It follows in particular that there is a Liouville vector field $Z$ on $\oX$ which vanishes at $p=L_i\cap\Delta$, with $L_i$ contained in the stable manifold of $p$, with $\Delta$ equal to the unstable manifold of $p$, and with a Morse function $H$ for which $Z$ is gradient like such that $H(p)>H(q)$ for every critical point $q\ne p$ of $H$. The maximum principle then implies that there is no non-constant holomorphic disk in $\oX$ with boundary on $L$ which passes through $p$. Thus, for a chord function $f_i$ with minimum $m_{i-}=p$, Proposition \ref{prop:Lefschetz-differential} implies that $d q^{(1)}_{i-}=d_{\rm const}q^{(1)}_{i-}=e_i$. We also note, for future reference, that we can choose the orbit function so that the gradient flow line connecting $m_{i-}$ to $Y$ lies near $\Delta$.

To show Remark \ref{rmk:killing-handle} $(2)$ we observe that  the contact structure in a  neighborhood $N$ of the coisotropic sphere $S\subset V$ is contactomorphic  to a neighborhood of the sphere $T=\pi^{-1}(\Delta)\cap Y^\st_\eps\subset Y^\st_\eps$. Moreover, one can arrange that this contactomorphism sends the part $\Lambda\cap N$ of the Legendrian sphere $\Lambda\subset V$ to  the part of the Legendrian lift $\Lambda_i\subset Y^{\st}_\eps$ of the Lagrangian sphere $L_i$ (we recall that we denote by $\pi$ the projection $X^\st=X\times\C\to X$) which lies in the neighborhood of $T$. Furthermore, one can arrange that the Reeb chord corresponding to $q^{(1)}_{i-}$ is arbitrarily short. In particular any holomorphic disk with positive puncture at this Reeb chord must stay inside $N$ and the argument we used for Remark \ref{rmk:killing-handle} $(1)$ also applies to prove Remark \ref{rmk:killing-handle} $(2)$.

Remark  \ref{rmk:killing-handle} $(2)$  is consistent with the fact that attaching handles of index $(n-1)$ and $n$   canceling each other does not change the  symplectic manifold, and hence its symplectic homology.

\newpage
\appendix
\centerline{\Large {\bf Appendix} }\medskip
\centerline{\Large Legendrian surgery formula and P.~Seidel's conjecture}
\medskip

\centerline{\large Sheel Ganatra \qquad   Maksim Maydanskiy}

\section{Introduction}
The goal of this text is to provide a dictionary between the language used
in the current paper and the formalism of \cite{shhh} and \cite{ainfnat}, in
particular establishing that Conjectures 6.3 and 6.4 in \cite{ainfnat}
follow from the results of the current paper. In an attempt to assist
readability, some of the material is repeated from these references. In this
preliminary version we ignore all signs and gradings, although some degree
shifts appear in the notation in anticipation of a latter version with
gradings.  

\section{The categories}
We start with an exact Lefschetz fibration $\pi:E \mapsto \D$, with fiber
$F$ and vanishing cycles $L_1, \ldots L_m$.
We recall the construction of a curved $\ainf$ category $\D$
which proceeds in several steps.  The material describing this construction
is taken from \cite{shhh}. Throughout we adopt the perspective from Section
3 of \cite{shhh}  and Remark 2.2 in \cite{ainfnat} that identifies $\ainf$
categories with $m$ objects $L_1, \ldots L_m$ with $\ainf$ algebras over
the semisimple ring $R=\K^m=\K e_1\op \ldots \op \K e_m$, by taking
category $\ZZ$ to the algebra $\op_{i,j}  \hom_{\ZZ} (L_i, L_j)$ (see also
Section 4.1 above).
To begin, we have the category $\B$ which is the full subcategory of the
Fukaya category of the fiber with vanishing cycles $L_i$ as objects.
Then we have its directed version $\A$, with same objects and morphisms
\[  
\hom_{\A}(L_i, L_j) = \left\{ \begin{array}{ll}
       \hom_{\B}(L_i, L_j)  & \mbox{if $j>i$};\\
      \K e_i & \mbox{if $i=j$};\\
      0 & \mbox{if $j<i$}  \end{array}. \right. 
\]
Here $\K e_i$ is a one dimensional vector space over our base field $\K$
spanned by the formal symbol $e_i$.  The $\ainf$ operations $\mu^k$ when
none of the inputs are $e_i$ are inherited from $\B$. When one of the
inputs is $e_i$ we declare $\mu^2(e_j, c)= c$ and $\mu^2(c, e_i)=c$ (here
we use the \textit{reverse} of Seidel's convention for composition so $c
\in \hom_{\A} (L_j, L_i)$).        All other $\mu^k$'s with $e_i$ inputs are
set to zero. This is what it means for $\A$ to be a strictly unital
category with strict identity morphisms $e_i$. Thinking of $\A$ as
an $\ainf$ algebra over $R$, we have an obvious projection to $R$ and the
kernel of it we call $\A_+$. This is the non-unital version,  with identity
morphisms taken out. Note that $\A_+$ injects into $\B$. Assuming $B$ is strictly unital (by, say passing to a quasi-isomorphic $\ainf$ structure as in \cite{book}, more on this in section \ref{sec:Reconstruction}), we can extend this embedding to all of $\A$.
Next,  the category $\CC= \A \op t \B [[t]]$ has the same objects as $\A$
and $\B$, and morphisms given by formal power series $x=x_0+ tx_1+\ldots $
with $x_0$ a morphism of $\A$ and the rest of $x_i$'s - morphisms of $\B$.
The $\ainf$ structure is extended from $\B$ by t-linearity (using the
inclusion $\A \hookrightarrow \B$; the strict unitality of $e_i$'s implies that
$\ainf$ relations still hold in $\CC$).  To get from $\CC$ to $\D$ we add a
curvature term $\mu^0_{\D}= t e_i \in \hom_{\D} (L_i, L_i)$. We need to verify
the curved $A_\infty$ relations for $\D$ and this again relies on strict
unitality of the $t e_i$'s.  Finally, the corresponding (non-unital) reduced
version of $\D$ is  denoted $\D_+$.

We should say a few words about the perturbation scheme arising in the
construction of the Fukaya category. In \cite{book}, the maps $\mu^k$ are
determined by counting perturbed pseudo-holomorphic discs with appropriate
Lagrangian boundary conditions, i.e. solutions to an inhomogenous
Cauchy-Riemann equation with inhomogeneous term a disc-dependent Hamiltonian.
Such a count depends on a choice of {\it perturbation datum}, which, roughly,
assigns to each representative of the Stasheff moduli space a choice of
(disc-dependent) perturbing Hamiltonian term and time-shifting maps. That this
can be done in such a way as to be smoothly compatible with the Deligne-Mumford
compactification of the moduli space is shown in \cite{book}.

There is a version of this construction that is Morse-Bott in flavor, which
uses a variant of the Cornea-Lalonde ``cluster'' technology \cite{CL}. This was
discussed in \cite{Se-g2} and rigorously constructed in \cite{Sheridan}. In
this scheme, morphisms are given by counting trees of holomorphic discs joined
by gradient flow lines of some choices of Morse function on each Lagrangian
submanifold. The resulting category is quasi-isomorphic to the version in
\cite{book}; see \cite{Sheridan} for more details.

For what follows, we will implicitly use the Morse-Bott variant of $\D$. See
also Section \ref{sec:Reconstruction} for more on this issue.
      
We will use the following notation for individual coefficients of $\mu_{\D}$:
\[ 
\mu^k_{\D}(a_1, \ldots, a_k) = \sum_{a}n^{\D}(a;a_1, \ldots, a_k) a.
\]
Note that we are still composing our morphisms from left to right - the
target of $a_1$ is the source of $a_2$ etc.

\section{The Legendrian homology algebra $LHA$ is dual to $T(\D_+ [1])$.}
\subsection{The vector spaces.}
As in Section \ref{ssec:LefschetzLHA} the $LHA$ is generated by $q_a^{(k)}$ for
intersection points $a \in L_i \cap L_j$ with  $k
\geq 0$ if $j>i$ (these are $\raq_{a}^{(p)}$ in notation of
\ref{ssec:LefschetzLHA}) and  $k \geq 1$ if $j<i$ ($\laq_{a}^{(p)}$
correspondingly), together with points $q_{i\pm}^{(k)}$ for each $i \geq 1$
(Morse-Bott case).    Specifically, $q_a^{(k)}$ is the Reeb chord between
$\L_i,\L_j$ corresponding to the intersection point $a \in L_i \cap L_j$
passing $k$ times above the reference point $e^{i \pi} \in S^1$. Given a chord
$c$ between $\L_i$ and $\L_j$, we denote by $p(c)$ the corresponding
intersection point of $L_i$ and $L_j$.  Note that the morphisms of
$\mathcal{A_+}$ are exactly the intersection points $a$ for $j>i$ and the
morphisms of $\mathcal{B}$ are all intersection points of
$L_i$ and $L_j$ (and critical points of a Morse function on $L_i$ if $i=j$).
This allows us to identify the Reeb chords with morphisms in $\mathcal{D}_+
=\mathcal{A}_+ \oplus t\mathcal{B}[[t]]$:  the Reeb chord $q_a^{(k)}$ for
$k\geq 1$ corresponds to the morphism $t^k a \in t^k \B$ and similarly the
chord $q_a^{(0)}$ corresponds  to the morphism in $a = t^0 a \in \A_+$.
Note that the morphisms in $\D_+$ are precisely the t-adic completion of
the $\K$ vector space spanned by Reeb chords, and so are naturally identified
with the dual of that vector space.
The algebra $LHA$ is the free algebra generated by the Reeb chords and the
underlying vector space is the span of words $c_1 \ldots c_k$. In our
dictionary this is dual to the tensor algebra $T(\D_+[1])$; note that $LHA$
includes empty words $e_i$ based at each component of $\Lambda$.
Correspondingly, the zeroth tensor power piece in $T(\D_+[1])$ is a copy or
$R$. Indeed, as a vector space, the tensor algebra is the t-adic completion of
the same span, and so is naturally identified with its dual.
We remark on a technical note, that to go back from morphism of $\D_+$ to Reeb
chords and from $T(\D_+[1])$ to $LHA$, respectively, one needs to take
the continuous dual with respect to the t-adic topology.  In what follows we
will use $c$ to denote both chords $q_a^{(k)}$ and the corresponding morphisms
$t^k a$.

\subsection{The differentials.}
We now proceed to establish the duality of differentials.  In the general setup
of \ref{sec:Leg-algebra}, the differential on $LHA$ is given on generators by:
\[ dc = \sum_{m}\sum_{c_1, \ldots, c_m} n_{c;c_1, \ldots, c_m} c_1\cdots c_m \]
where $n_{c;c_1, \ldots, c_m}$ the algebraic number of 1-dimensional components
of the moduli space $\mathcal{M}_{\Lambda}^Y(c;c_1, \ldots, c_m)$.
In the present case $Y$  is the boundary of a stabilization $X^{st}$, and the Morse-Bott version of $\mathcal{M}_{\Lambda}^Y(c;c_1, \ldots,
c_m)$ should be used.  We see that $Y$ is split into the
vertical (or mapping torus) part $Y^v$ ($Y^{st}_1$ in the notation of
Section \ref{ssec:LefschetzFibr})  and horizontal  (or binding) part $Y^h$
(correspondingly $Y^{st}_2$), with $\Lambda \subset Y^v$.  Correspondingly $n$
contains contributions from discs contained in $Y^v \times \R$, $n^v$, and the
ones not contained in $Y^v \times \R$, $n^{bind}$ (the contributions from
$d_{const}$ in Section \ref{sec:Lefschetz}).

\subsection{Reconstruction.}\label{sec:Reconstruction}

The basic idea for relating differentials is that holomorphic discs contained in
$Y^v \times \R = F \times S^1 \times \R$ and can be projected
to $F$. Conversely, given a collection of chords $c_1, \dots, c_m$, and a holomorphic disc $u$ with punctures asymptotic to $p(c_1), \dots, \p(c_m)$ with appropriate Lagrangian boundary conditions, one can uniquely lift $u$ to a disc in
$Y^v \times \R = F \times S^1 \times \R$, provided that the total power of t
for the negative chords is equal to the power of t of the positive chord. This
last requirement is the condition for lifting $u$ as a continuous map and
corresponds in the ``dictionary'' to t-linearity of $A_\infty$ operations in
$\D$.  Furthermore, the algebraic counts from the moduli spaces with and
without being anchored (see Section \ref{sec:anchored}) agree.

If structure constants $n^v_{c;c_1, \ldots,
c_m}$ and $n^{\D} (c,c_1, \ldots, c_m)$ were given simply by counts of such holomorphic curves, this would imply

\label{v} \[  n^v_{c;c_1, \ldots,
c_m} = n^{\D} (c,c_1, \ldots, c_m)\]
whenever the set of $c_i$'s is not empty.

However, the present case of a Lefschetz fibration does not fall into the
general scheme of Section \ref{sec:algebra}. Rather, it's a ``double
Morse-Bott'' situation as explained in Section \ref{sec:Lefschetz-sketch}. As a
result the moduli spaces involved are more complicated, and part ii of
Proposition \ref{prop:Lefschetz-differential} establishes a more refined
version of the correspondence. The result is that the relevant moduli spaces in
$F$ are the generalised holomorphic discs of Section \ref{ssec:genholdisk}, or
holomorphic pearly trees from Section 4.3 in \cite{Sheridan}. This is why we
work with Morse-Bott versions of categories $\A$, $\B$, $\CC$ and $\D$.

We will now discuss perturbations and unitality. As discussed above, while
constructing the Morse-Bott Fukaya category of the fiber $F$, in order to
define the self-hom space $\hom(L,L)$ one picks a Floer datum for the pair
$(L,L)$, one part of which is a Morse function $H$ supported on $L$. On a
sphere, such a function then can be chosen to have just two critical points,
one maximum and one minimum, producing a 2-dimensional $\hom(L,L)$. The
generator $e_L$ corresponding to the minimum then descends to a homological
level identity morphism. However, picking all additional data in such a way
that the Fukaya category of the fiber is strictly unital on the chain level is
a non-trivial matter. That this can be done for (the Morse-Bott version of) the
category $\B$ follows calculations made in Section \ref{sec:Lefschetz-sketch}
and the above mentioned correspondence, in a manner we will now explain.

Proposition \ref{prop:Lefschetz-differential} computes the $LHA$ differential
in the Lefschetz case. Families of multi-level Bott curves in $Y_1^{\rm st}$
(with different powers of $t$, and including ones without any actual
holomorphic components) contributing to the $LHA$ differential correspond to
the generalized holomorphic discs computing Morse-Bott version of $\B$ (again,
including those without actual holomorphic components). The computation of
Section \ref{sec:Lefschetz-sketch} therefore implicitly contains part of the
computation of the structure of $\B$.

Let's check that it gives strict unitality of $e_i$'s. The algebra $LHA$ is
dual to $T(\D_+ [1])$, which involves the reduced version of $\D$, and hence
has no generators directly corresponding to the minima on $L_i$. It has only
the generators corresponding to the $t$-multiples of them, $q_{i-}^{k}$, $k>0$.
It is enough to check that those satisfy a suitable $t$-linear version of
unitality.
Determining the value of $\mu^r(\ldots, t^k e_i, \ldots)$ is dual to involves
finding all $q$'s with $dq$ containing $q_{i-}^{k}$.  By Proposition
\ref{prop:Lefschetz-differential} these and the corresponding $\ainf$ equations
are as follows. First there are 
\begin{eqnarray} 
    \nonumber d_{{\rm
    MB}}\q_{i-}&=&\q_{i-}\q_{i-}\\ 
    \nonumber d_{{\rm
    MB}}\q_{i+}&=&\q_{i-}\q_{i+}+(-1)^{n-1}\q_{i+}\q_{i-}, 
\end{eqnarray}
and correspondingly
\begin{eqnarray}
\mu_2(t^k e_{L_i},t^l e_{L_i})&=&t^{l+k} e_{L_i},\\
\mu_2(t^k e_{L_i},t^l m_{L_i})&=&t^{k+l} m_{L_i},\\
\mu_2(t^l m_{L_i},t^k e_{L_i})&=&t^{k+l} m_{L_i},
\end{eqnarray}
where $m_{L_i} \in Hom(L_i, L_i)$ is the top dimensional generator corresponding to the maximum of a Morse function. There are also
\begin{eqnarray}
\nonumber d_{{\rm MB}}(T\Raq_a)&=&\q_{j-}{\Raq}_a +(-1)^{|q_a|}\Raq_a\q_{i-} \\
\nonumber d_{{\rm MB}}(T\Laq_a)&=&\q_{i-}{\Laq}_a +(-1)^{|q_a|}\Laq_a\q_{j-}
\end{eqnarray}
corresponding to
\begin{eqnarray}
\mu_2(t^k e_{L_i}, t^l a)&=& t^{l+k} \\
\mu_2(t^l a,t^k e_{L_j})&=& t^{k+l} a.
\end{eqnarray}  

Moreover, no differential has a degree three or greater term containing $\q_{i-}$, which is dual to the fact that
\begin{equation}
\mu^k(\ldots, t^l e_i, \ldots) = 0\mathrm{\ for\ }k\geq 3.
\end{equation}

From Section \ref{sec:Lefschetz-sketch} we see that all parts of $d$ except
$d_{const} q_{i-}^{1}$ are $T$-linear, which corresponds to $t$-linearity of
$\mu^k$'s in
$\CC$.

\subsection{Contributions from the binding.}
The additional contribution from $n^{bind}$ has been calculated in Proposition \ref{prop:Lefschetz-differential} ($d_{const}$ in that notation)
to be: $n^{bind}_{q_{i-}^{(1)};1} = 1$ and $n^{bind}_{c;c_1, \ldots, c_m}=
0$ in
all other cases. This means that
\label{bind} \[{n^{bind}_{c;}}^{\D}(c;)\]
\subsection{Total differential.}
We have
\label{full}  \[n_{c;c_1, \ldots,
c_m} =  n^v_{c;c_1, \ldots,
c_m}+{n^{bind}_{c;c_1, \ldots,
c_m}}   ^{\D}(c_1, \ldots, c_m).\]
On the other side, the $\ainf$ structure on $\D_+$ viewed as a map $\mu_\D:
T(\D_+[1])
\rightarrow \D_+[1]$ uniquely extends by the co-Leibniz rule to a
coderivation $\hat{\mu}_\D$ on
$T(\D_+[1])$, or, equivalently, its dual extends to a derivation on
$T(\D_+[1])^{*}= LHA$. This last is given by
\[d_{\D} = (\mu)^*= d_{\D}^0 + d_{\D}^1+ d_{\D}^2 +\ldots , \]
where $d_{\D}^l: \D_+[1] \rightarrow T(\D_+[1])^{\ot l}$, $d_\D$ is
extended by the Leibniz rule, and
\begin{eqnarray*}
d_{\D}^0 (c) &=&  (\mu^0_\D)^*(c) = \sum_c n^{\D}(c;) \\
d_{\D}^1 (c) &=&  (\mu^1_\D)^*(c) = \sum_{a_1} n^{\D}(c;a_1) a_1\\
d_{\D}^2 (c) &=&  (\mu^2_\D)^*(c) = \sum_{a_1,a_2} n^{\D}(c;a_1, a_2)
a_1\otimes a_2
\end{eqnarray*}
The equalities  in \ref{v}-\ref{full} mean that the differential on $LHA$
which only includes the contributions of $n^v$s comes from the $\ainf$
structure
of $\CC$, and that the full differential is the one coming from the full
curved $\ainf$
structure of $\D$, namely $d_{\D}$.
\section{$\lhwt$}
\subsection{Generators}
On the level of generators, $\lhwt$ consists of
\[ \lhwt = C(\L) \op \lhcc \op \lhch. \]
Here, $\lhcc$ and $\lhcc$ are simply two copies of $\lhc$, the subcomplex
of
words that are cyclically composable divided by empty words. Under the
identification described previously, this is dual to $\td_+^{diag}$
Words in
$\lhcc$ will be denoted $\check{c_1}\cdots c_k$ and words in $\lhch$ are
denoted $\hat{c_1}\cdots c_k$. $C(\L)$ is the vector-space generated by
elements $\tau_1, \ldots, \tau_k$, in one-to-one correspondence with
$\L_1,\ldots,\L_k$.
To identify this on the level of generators with $CC(\D) = \ccd$, we note
that
\begin{align}
\ccd &= ((R \op \D_+)\ot \td)^{diag}\\
&= (\td \op \td_+[-1])^{diag} \\
\label{ccexpr}&= R \op (R \ot\td_+)^{diag} \op (\td_+[-1])^{diag}\\
&= R \op (\td_+)^{diag} \op (\td_+[-1])^{diag}.
\end{align}
Here, recall that $TV = \K \oplus V \oplus V^{\otimes 2}\oplus \cdots$, and
$(TV)_+$ refers to the non-zero length part $V \oplus V^{\otimes 2} \oplus
\cdots$.
This last expression exactly gives us the correspondence to the generators
of
$\lhwt$.  We should remark however that to compute the Hochschild
differentials, one needs to remember (\ref{ccexpr}).
\begin{center}
\begin{table}
\centering
\begin{tabular}{|c|c|c|}
\hline
$C(\Lambda)$ & $\lhcc$ & $\lhch$ \\ \hline
$\tau_i$ & $\check{c}_1\cdots c_k$ & $\hat{c}_1\cdots c_k$\\
$\updownarrow$ & $\updownarrow$ & $\updownarrow $ \\
$e_i$ & $e_i \cdot c_1 \cdots c_k$ & $c_1\cdots c_k$\\ \hline
$R$ & $(R \ot \td_+)^{diag}$ & $\td_+[-1]^{diag}$\\
\hline
\end{tabular}
\caption{A dictionary between the generators of $\lhwt$ and $CC(\D,\D)$}
\end{table}
\end{center}
\subsection{Hochschild Differentials}
The Hochschild differential on $\delta$ has three primary constituents.
There
is the (acyclic) Hochschild differential (without any $\mu^0$):
\[ b_a : x_1 \cdots x_k \mapsto \sum_{i\geq 1,s\geq 1} x_1 \cdots
\mu(x_{i}, \ldots,
x_{i+s-1})\cdot x_{i+s}\cdots x_{k}.\]
The differential also includes cyclic terms
\[ b_c:  x_1 \cdots x_k \mapsto \sum_{i\geq 1,j\geq 1} \mu(x_{k-j}, \ldots,
x_k, x_1, \ldots x_i) \cdot x_{i+1}\cdots x_{k-j-1} \]
and terms obtained by inserting $\mu_0$ everywhere except for the beginning
\[ b_0: x_1\cdots x_k \mapsto \sum_{i\geq 1} x_1\cdots x_i \cdot \mu_0
\cdot x_{i+1} \cdots x_k.\]
Note that $b_a + b_0$ viewed as a map on $\td$ is not quite the differential
$\hat{\mu}$ induced by extending the $A_\infty$ maps $\mu :\td \rightarrow
\D$
to all of $\td$ via the co-Leibniz rule. One needs to add a final
contribution
$b_0^0$, where \[ b_0^0: x_1 \cdots x_k \mapsto \mu^0 \cdot x_1\cdots
x_k.\]
Now, look at the differential on various components of $\lhwt$ and dualize.
$\hat{d}_{LHO}: \lhch \rightarrow \lhch$ is defined as:
\[
\hat{d}_{LHO}(\hat{c}_1\cdots c_n) = S(d(\hat{c}_1))\cdot c_2\cdots c_n +
\hat{c}_1
d_{LHA}(c_2\cdots c_n),\]
where $S(w_1\cdots w_k) = \sum_{i}(-1)^*\prod w_1\cdots \hat{w}_i \cdots
w_k$
and $S(1) = 0$.
Write $\hat{d}_{LHO} = S^1\circ\d_{LHA} + \tilde{d}$, where $d_{LHA}$ is
simply the differential on $LHA(\L)$, and $S^1(c_1\cdots c_k) =
\hat{c}_1\cdots c_k$. Then, if
$d_{LHA}c_1$ does not contain 1 as a contribution (i.e. $c_1\neq
q_{i-}^{(1)}$), $\tilde{d}(c_1\cdots
c_k) = (S(dc_1) - S^1(dc_1))c_2\cdots c_k$. If $c_1 = q_{i-}^{(1)}$,
$\tilde{d}(c_1\cdots c_k) = -\hat{c}_2\cdots c_k + (S(dc_1 - 1) - S^1(dc_1
-
1))c_2\cdots c_k$ or $S(dc_1 - 1) - S^1(dc_1-1)$ if $k=1$.
%
In the previous section, we argued that $d_{LHA}^* = \hat{\mu} = b_a + b_0
+
b_0^0$, so it remains to see that the dual of the map $\tilde{d}$ is $b_c -
b_0^0$. But $\tilde{d}(q_{i-}^{(1)}\cdot c_2\cdots c_n)$ containing a term
equal to $-\hat{c}_2\cdots c_n$ is exactly dual to $-b_0^0$ so it remains
to
analyze the dual of terms in $\tilde{d}(c_1\cdots c_k)$ not containing
$-\hat{c}_2\cdots c_k$.
For $dc_1$ containing terms $\sum n(c_1;y_1, \ldots, y_l) y_1\cdots y_l$,
\[\tilde{d}(c_1\cdots c_k) = \sum n(c_1; y_1,\ldots,y_l)\left(\sum_{j\geq
2}
y_1\cdots \hat{y}_j \cdots y_l\right) c_2\cdots c_n.\]
Dually  the contribution of $b_c (\hat{y}_j\cdots y_l c_2\cdots c_n y_1\cdots y_{j-1})$  coming from\\
$\mu(y_1,\ldots y_l)\cdot c_2\cdots c_n$  is equal to  $n_F(c_1;y_1,\ldots, y_l)c_1\cdots c_n$. All of these together give us
exactly
the desired duality.
On the level of generators, $\left(\lhcc\right)^*$ is identified with
$(R\ot
\td)^{diag}$, meaning the subspace of the chain complex generated by words
that
start with $e_i$. Let's explicitly analyze the Hochschild differential
$\delta$
on such words. There are contributions where $\mu$ does not take the first
element $e_i$ as an input: \[e_i\cdot c_1\cdots c_k \rightarrow e_i\cdot
(b_a+b_0 + b_0^0)(c_1\cdots c_k)\] which  is exactly $\check{d}_{LHO}$ under the identification
$(R\otimes
\td)^{diag} \simeq \left(\lhcc\right)^*$ which sends $e_i\cdot c_1\ldots c_k$ to
$\check{c}_1\cdots c_k$.  The other
contributions are the
ones where $\mu^r$ has $e_i$ as an input, which is 0 unless $r=2$.
Explicitly,
those are:
\begin{eqnarray}
e_i\cdot c_1\cdots c_k \mapsto \mu^2(c_k,e_i)\cdot c_1\ldots
c_{k-1} + \mu^2(e_i,c_1)\cdot c_2\cdots c_k \\
= \mp c_1\cdots c_k \pm c_k\cdot
c_1\cdots c_{k-1}
\end{eqnarray}
Under the identification $$(R\otimes \td)^{diag} \lr \left(\lhcc \right)^*$$
and
$$(\td_+)^{diag} \lr \left(\lhch\right)^*,$$ this is precisely $\alpha^*.$
In $\lhwt$, the differential going from $C(\L)$ was 0, which dually
corresponds to
the fact that $\mu(x_1\ldots x_s)$ is never $e_i$. The final term in the
differential is the map $T$, which takes the single chord $c_i$ going from
$\L_i$ to $\L_i$ to $n_{c_i}\tau_i$, where $n_{c_i} = 1$ for $c_i =
q_{i-}^{(1)}$ and $0$ otherwise. This dualizes to $T^*(e_i) =
\hat{q}_{i-}^{(1)}$. Under the identification, this is $e_i \ot te_i$,
which is
exactly $\delta(e_i)$.
Conclusion: $\left( \left(\lhwt\right)^*,d_{\mathrm{Ho}}^* \right) =\left(
\left(\ccd\right), \delta \right)$ as chain complexes.

\bigskip

{\noindent The authors acknowledge the support   of:\\
{\it F. Bourgeois:}  The Fonds National de la Recherche Scientifique (Belgium)\\
 {\it T. Ekholm:} The G{\"o}ran Gustafsson Foundation for Research in Natural Sciences and Medicine\\
 {\it Y. Eliashberg:} NSF grants DMS-0707103 and DMS 0244663}
 \nocite{*}





\end{document}